\documentclass[11pt]{article}
\usepackage[active]{srcltx}
\usepackage[applemac]{inputenc}
\usepackage{amsmath,amsthm}
\usepackage{amsfonts,txfonts,fontenc}
\usepackage{amssymb}
\usepackage[colorlinks=true,urlcolor=blue,
citecolor=red,linkcolor=blue,linktocpage,pdfpagelabels,bookmarksnumbered,bookmarksopen]{hyperref}
\usepackage{graphicx, tikz}

\tikzset { domaine/.style 2 args={domain=#1:#2} }
\tikzset{
xmin/.store in=\xmin, xmin/.default=-3, xmin=-3,
xmax/.store in=\xmax, xmax/.default=3, xmax=3,
ymin/.store in=\ymin, ymin/.default=-3, ymin=-3,
ymax/.store in=\ymax, ymax/.default=3, ymax=3,
}

\usepackage[english]{babel}
\usepackage{marginnote}
\usepackage[left=2.95cm,right=2.95cm,top=2.8cm,bottom=2.8cm]{geometry}
\numberwithin{equation}{section}
\usepackage{mathrsfs}
\usepackage{color}
\usepackage{amsmath, bbm}
\usepackage{amsfonts}
\usepackage{amssymb}
\usepackage{graphicx}%
\providecommand{\U}[1]{\protect\rule{.1in}{.1in}}
\definecolor{linkcolor}{rgb}{0.00,0.50,0.00}
\providecommand{\U}[1]{\protect\rule{.1in}{.1in}}
\newtheorem{theorem}{Theorem}[section]
\newtheorem{proposition}[theorem]{Proposition}
\newtheorem{lemma}[theorem]{Lemma}

\newtheorem{definition}{Definition}[section]
\newtheorem{remark}{Remark}[section]

\numberwithin{equation}{section}

\newcommand{\pical}{\mathcal{P}}
\newcommand{\dd}{\mathrm{d}}

\newcommand{\res}{\mathop{\hbox{\vrule height 7pt width .5pt depth 0pt \vrule height .5pt width 6pt depth 0pt}}\nolimits}
\newcommand{\deb}{\rightharpoonup}

\newcommand{\haus}{\mathcal H}

\newcommand{\vv}{\mathbf{v}}
\newcommand{\uu}{\mathbf{u}}

\newcommand\ip{\overset{\circ}}

\newcommand{\M}{\mathcal M}
\newcommand{\Wc}{\mathcal{T}_c}
\newcommand{\lcal}{\mathcal L}

\newcommand{\pp}{\mathcal{P}_p(\Omega)}
\newcommand{\WW}{\mathbb W}

\newcommand{\mm}{\mathfrak m}

\newcommand{\id}{id}
\newcommand{\impl}{\Rightarrow}
\newcommand{\ve}{\varepsilon}

\newcommand{\R}{\mathbb R}

\newcommand{\nn}{\mathbf{n}}
\newcommand{\T}{\mathrm{T}}

\DeclareMathOperator{\argmin}{argmin}
\DeclareMathOperator{\spt}{spt}
\DeclareMathOperator{\Lip}{Lip}

\DeclareMathOperator{\len}{Length}
\DeclareMathOperator{\diam}{diam}

\def\spt{{\rm{spt}}} 
\def\dd{{\rm d}} 

\usepackage{babel}
\reversemarginpar

\title{$\big\{$Euclidean, Metric, and Wasserstein$\big\}$ Gradient Flows:\\ an overview}
\author{Filippo Santambrogio\thanks{\scriptsize\ Laboratoire de Math\'ematiques d'Orsay, Univ. Paris-Sud, CNRS, Universit\'e Paris-Saclay, 91405 Orsay Cedex, France,
\texttt{filippo.santambrogio@math.u-psud.fr, http://www.math.u-psud.fr/$\sim$santambr}}}
\date{}
\begin{document}
\maketitle

\begin{abstract}This is an expository paper on the theory of gradient flows, and in particular of those PDEs which can be interpreted as gradient flows for the Wasserstein metric on the space of probability measures (a distance induced by optimal transport). The starting point is the Euclidean theory, and then its generalization to metric spaces, according to the work of Ambrosio, Gigli and Savar\'e. Then comes an independent exposition of the Wasserstein theory, with a short introduction to the optimal transport tools that are needed and to the notion of geodesic convexity, followed by a precise desciption of the Jordan-Kinderleher-Otto scheme, with proof of convergence in the easiest case: the linear Fokker-Planck equation. A discussion of other gradient flows PDEs and of numerical methods based on these ideas is also provided. The paper ends with a new, theoretical, development, due to Ambrosio, Gigli, Savar\'e, Kuwada and Ohta: the study of the heat flow in metric measure spaces.
\end{abstract}


\tableofcontents

\section{Introduction}

{\it Gradient flows}, or {\it steepest descent curves}, are a very classical topic in evolution equations: take a functional $F$ defined on a vector space $X$, and, instead of looking at points $x$ minizing $F$ (which is related to the statical equation $\nabla F(x)=0$), we look, given an initial point $x_0$, for a curve starting at $x_0$ and trying to minimize $F$ as fast as possible (in this case, we will solve equations of the form $x'(t)=-\nabla F(x(t))$). As we speak of gradients (which are element of $X$, and not of $X'$ as the differential of $F$ should be), it is natural to impose that $X$ is an Hilbert space (so as to identify it with its dual and produce a gradient vector). In the finite-dimensional case, the above equation is very easy to deal with, but also the infinite-dimensional case is not so exotic. Indeed, just think at the evolution equation $\partial_t u =\Delta u$, which is the evolution variant of the statical Laplace equation $-\Delta u=0$. In this way, the Heat equation is the gradient flow, in the $L^2$ Hilbert space, of the Dirichlet energy $F(u)=\frac12\int|\nabla u|^2$, of which $-\Delta u$ is the gradient in the appropriate sense (more generally, one could consider equations of the form $\partial_t u =\delta F/\delta u$, where this notation stands for the first variation of $F$).

But this is somehow classical\dots\ The renovated interest for the notion of gradient flow arrived between the end of the 20th century and the beginning of the 21st, with the work of Jordan, Kinderleherer and Otto (\cite{JKO}) and then of Otto \cite{Otto porous}, who saw a gradient flow structure in some equations of the form $\partial_t\varrho-\nabla\cdot(\varrho \vv)=0$, where the vector field $\vv$ is given by $\vv=\nabla[\delta F/\delta \varrho]$. This requires to use the space of probabilities $\varrho$ on a given domain, and to endow it with a non-linear metric structure, derived from the theory of optimal transport. This theory, initiated by Monge in the 18th century (\cite{Monge}), then developed by Kantorovich in the '40s (\cite{Kantorovich}), is now well-established (many texts present it, such as \cite{villani,villani06,OTAM}) and is intimately connected with PDEs of the form of the {\it continuity equation} $\partial_t\varrho-\nabla\cdot(\varrho \vv)=0$.

The turning point for the theory of gradient flows and for the interest that researchers in PDEs developed for it was for sure the publication of \cite{AmbGigSav}. This celebrated book established a whole theory on the notion of gradient flow in metric spaces, which requires careful definitions because in the equation $x'(t)=-\nabla F(x(t))$, neither the term $x'$ nor $\nabla F$ make any sense in this framework. For existence and - mainly - uniqueness results, the notion of geodesic convexity (convexity of a functional $F$ defined on a metric space $X$, when restricted to the geodesic curves of $X$) plays an important role. Then the theory is particularized in the second half of \cite{AmbGigSav} to the case of the metric space of probability measures endowed with the so-called Wasserstein distance coming from optimal transport, whose differential structure is widely studied in the book. In this framework, the geodesic convexity results that McCann obtained in \cite{MC} are crucial to make a bridge from the general to the particular theory.

It is interesting to observe that, finally, the heat equation turns out to be a gradient flow in two different senses: it is the gradient flow of the Dirichlet energy in the $L^2$ space, but also of the entropy $\int\varrho\log(\varrho)$ in the Wasserstein space. Both frameworks can be adapted from the particular case of probabilities on a domain $\Omega\subset\R^d$ to the more general case of metric measure spaces, and the question whether the two flows coincide, or under which assumptions they do, is natural. It has been recently studied by Ambrosio, Gigli, Savar\'e and new collaborators (Kuwada and Ohta) in a series of papers (\cite{GIGLIHeat,GigKuwOht,AmbGigSav-heat}), and has been the starting point of recent researches on the differential structure of metric measure spaces.

The present survey, which is an extended, updated, and English version of a Bourbaki seminar given by the author in 2013 (\cite{Bourbaki2013}; the reader will also remark that most of the extensions are essentially taken from \cite{OTAM}), aims at giving an overview of the whole theory. In particular, among the goals, there is at the same time to introduce the tools for studying metric gradient flows, but also to see how to deal with Wasserstein gradient flows without such a theory. This could be of interest for applied mathematicians, who could be more involved in the specific PDEs that have this gradient flow form, without a desire for full generality; for the same audience, a section has been added about numerical methods inspired from the so-called JKO (Jordan-Kinderleherer-Otto) scheme, and one on a list of equations which fall into these framework. {\it De facto}, more than half of the survey is devoted to the theory in the Wasserstein spaces and full proofs are given in the easiest case.

The paper is organized as follows: after this introduction, Section \ref{2} exposes the theory in the Euclidean case, and presents which are the good definitions which can be translated into a metric setting; Section \ref{metricth} is devoted to the general metric setting, as in  the first half of \cite{AmbGigSav}, and is quite expository (only the key ingredients to obtain the proofs are sketched); Section \ref{W2} is the longest one and develops the Wasserstein setting: after an introduction to optimal transport and to the Wasserstein distances, there in an informal presentation of the equations that can be obtained as gradient flows, a discussion of the functionals which have geodesic convexity properties, a quite precise proof of convergence in the linear case of the Fokker-Planck equation, a discussion about the other equations and functionals which fit the framework and about boundary conditions, and finally a short section about numerics. Last but not least, Section \ref{heat} gives a very short presentation of the fascinating topic of heat flows in arbitraty metric measure spaces, with reference to the interesting implications that this has in the differential structure of these spaces.

This survey is meant to be suitable for readers with different backgrounds and interests. In particular, the reader who is mainly interested gradient flows in the Wasserstein space and in PDE applications can decide to skip sections \ref{metricth}, \ref{geodconvW2} and \ref{heat}, which deal on the contrary with key objects for the - very lively at the moment - subject of analysis on metric measure spaces.

\section{From Euclidean to Metric}\label{2}

\subsection{Gradient flows in the Euclidean space}

Before dealing with gradient flows in general metric spaces, the best way to clarify the situation is to start from the easiest case, i.e. what happens in the Euclidean space $\R^n$. Most of what we will say stays true in an arbitrary Hilbert space, but we will stick to the finite-dimensional case for simplicity.

Here, given a function  $F:\R^n\to \R$, smooth enough, and a point $x_0\in\R^n$, a gradient flow is just defined as a curve $x(t)$, with starting point at $t=0$ given by $x_0$, which moves by choosing at each instant of time the direction which makes the function $F$ decrease as much as possible. More precisely, we consider the solution of the {\it Cauchy Problem}
\begin{equation}\label{nabF}
\begin{cases}x'(t)=-\nabla F(x(t))&\mbox{ for }t>0,\\
			x(0)=x_0.	\end{cases}
			\end{equation}
This is a standard Cauchy problem which has a unique solution if $\nabla F$ is Lipschitz continuous, i.e. if $F\in C^{1,1}$. We will see that existence and uniqueness can also hold without this strong assumption, thanks to the variational structure of the equation. 

A first interesting property is the following, concerning uniqueness and estimates. We will present it in the case where $F$ is convex, which means that it could be non-differentiable, but we can replace the gradient with the subdifferential. More precisely, we can consider instead of \eqref{nabF}, the following differential inclusion: we look for an absolutely continuous curve $x:[0,T]\to\R^n$ such that
\begin{equation}\label{subnabF}
\begin{cases}x'(t)\in-\partial F(x(t))&\mbox{ for a.e. }t>0,\\
			x(0)=x_0,	\end{cases}
			\end{equation}
where $\partial F(x)=\{p\in\R^n\,:\,F(y)\geq F(x)+p\cdot(y-x)\;\mbox{ for all }y\in\R^n\}$. We refer to \cite{Rockafellar} for all the definitions and notions from convex analysis that could be needed in the following, and we recall that, if $F$ is differentiable at $x$, we have $\partial F(x)=\{\nabla F(x)\}$ and that $F	$ is differentiable at $x$ if and only if $\partial F$ is a singleton. Also note that $\partial F(x)$ is always a convex set, and is not empty whenever $F$ is real-valued (or $x$ is in the interior of $\{x\,:\,F(x)<+\infty\}$), and we denote by $\partial^\circ F(x)$ its element of minimal norm. 
\begin{proposition}\label{convex case}
Suppose that $F$ is convex and let $x_1$ and $x_2$ be two solutions of \eqref{subnabF}. Then we have
$|x_1(t)-x_2(t)|\leq |x_1(0)-x_2(0)|$ for every $t$.
In particular this gives uniqueness of the solution of the Cauchy problem.
\end{proposition}
\begin{proof}
Let us consider $g(t)=\frac 12|x_1(t)-x_2(t)|^2$ and differentiate it. We have 
$$g'(t)=(x_1(t)-x_2(t))\cdot(x'_1(t)-x'_2(t)).$$
Here we use a basic property of gradient of convex functions, i.e. that for every $x_1,x_2,p_1,p_2$ with $p_i\in\partial F(x_i)$, we have
$$(x_1-x_2)\cdot (p_1-p_2)\geq 0.$$
From these considerations, we obtain $g'(t)\leq 0$
and $g(t)\leq g(0)$. This gives the first part of the claim.

Then, if we take two different solutions of the same Cauchy problem, we have $x_1(0)=x_2(0)$, and this implies $x_1(t)=x_2(t)$ for any $t>0$. 
\end{proof}

We can also stuy the case where $F$ is semi-convex. We recall that $F$ semi-convex means that it is $\lambda$-convex for some $\lambda\in\R$ i.e. $x\mapsto F(x)-\frac\lambda 2 |x|^2$ is convex. For $\lambda>0$ this is stronger than convexity, and for $\lambda<0$ this is weaker. Roughly speaking, $\lambda$-convexity corresponds to $D^2F\geq \lambda I$. Functions which are $\lambda$-convex for some $\lambda$ are called {\it semi-convex}. The reason of the interest towards semi-convex functions lies in the fact that on the one hand, as the reader will see throughout the exposition, the general theory of gradient flows applies very well to this class of functions and that, on the other hand, they are general enough to cover many interesting cases. In particular, on a bounded set, all smooth ($C^2$ is enough) functions are $\lambda$-convex for a suitable $\lambda<0$. 

For $\lambda$-convex functions, we can define their subdifferential as follows
$$\partial F(x)=\left\{p\in\R^n\,:\,F(y)\geq F(x)+p\cdot(y-x)+\frac\lambda 2|y-x|^2 \;\mbox{ for all }y\in\R^n\right\}.$$	
This definition is consistent with the above one whenever $\lambda\geq 0$ (and guarantees $\partial F(x)\neq\emptyset$ for $\lambda<0$). Also, one can check that, setting $\tilde F(x)=F(x)-\frac\lambda 2 |x|^2$, this definition coincides with $\{p\in\R^n\,:\,p-\lambda x\in \partial \tilde F(x)\}$. Again, we define $\partial^\circ F$ the element of minimal norm of $\partial F$.
\begin{remark}
From the same proof of Proposition \ref{convex case}, one can also deduce uniqueness and stability estimates in the case where $F$ is $\lambda$-convex. Indeed, in this case we obtain $|x_1(t)-x_2(t)|\leq |x_1(0)-x_2(0)|e^{-\lambda t},$ which also proves, if $\lambda>0$, exponential convergence to the unique minimizer of $F$. The key point is that, if $F$ is $\lambda$-convex  it is easy to prove that $x_1,x_2,p_1,p_2$ with $p_i\in\partial F(x_i)$ provides
$$(x_1-x_2)\cdot (p_1-p_2)\geq \lambda|x_1-x_2|^2.$$
This implies $g'(t)\leq -2\lambda g(t)$ and allows to conclude, by Gronwall's lemma, $g(t)\leq g(0) e^{-2\lambda t}$. For the exponential convergence,  if $\lambda>0$ then $F$ is coercive and admits a minimizer, which is unique by strict convexity. Let us call it $\bar x$. Take a solution $x(t)$ and compare it to the constant curve $\bar x$, which is a solution since $0\in \partial F(\bar x)$. Then we get $|x_1(t)-\bar x|\leq e^{-\lambda t}|x_1(0)-\bar x|$.
\end{remark}
%
%

Another well-known fact about the $\lambda$-convex case is the fact that the differential inclusion $x'(t)\in-\partial F(x(t))$ actually becomes, a.e., an equality: $x'(t)=-\partial^\circ F(t)$. More precisely, we have the following.

\begin{proposition}\label{minimal norm}
Suppose that $F$ is $\lambda$-convex and let $x$ be a solutions of \eqref{subnabF}. Then, for all the times $t_0$ such that both $t\mapsto x(t)$ and $t\mapsto F(x(t))$ are differentiable at $t=t_0$, the subdifferential $\partial F(x(t_0))$ is contained in a hyperplane orthogonal to $x'(t_0)$. In particular, we have $x'(t)=-\partial^\circ F(x(t))$ for a.e. $t$.
\end{proposition}
\begin{proof}
Let $t_0$ be as in the statement, and $p\in\partial F(x(t_0))$. From the definition of subdifferential, for every $t$ we have 
$$F(x(t))\geq F(x(t_0))+p\cdot (x(t)-x(t_0))+\frac{\lambda}2|x(t)-x(t_0)|^2,$$
but this inequality becomes an equality for $t=t_0$. Hence, the quantity 
$$F(x(t))- F(x(t_0))-p\cdot (x(t)-x(t_0))-\frac{\lambda}2|x(t)-x(t_0)|^2$$ 
is minimal for $t=t_0$ and, differentiating in $t$ (which is possible by assumption), we get
$$\frac{d}{dt}F(x(t))_{|t=t_0}=p\cdot x'(t_0).$$
Since this is true for every  $p\in\partial F(x(t_0))$, this shows that $\partial F(x(t_0))$ is contained in a hyperplane of the form $\{p\,:\,p\cdot x'(t_0)=const\}$. 

Whenever $x'(t_0)$ belongs to $\partial F(x(t_0))$ (which is true for a.e. $t_0$), this shows that $x'(t_0)$ is the orthogonal projection of $0$ onto $\partial F(x(t_0))$ and onto the hyperplane which contains it, and hence its element of minimal norm. This provides $x'(t_0)=-\partial^\circ F(x(t_0))$ for a.e. $t_0$, as the differentiability of $x$ and of $F\circ x$ are also true a.e., since $x$ is supposed to be absolutely continuous and $F$ is locally Lipschitz.
 \end{proof}

Another interesting feature of those particular Cauchy problems which are gradient flows is their discretization in time. Actually, one can fix a small time step parameter $\tau>0$ and look for a sequence of points $(x^\tau_k)_k$ defined through the iterated scheme, called {\it Minimizing Movement Scheme},
\begin{equation}\label{MMeu}
x^\tau_{k+1}\in\argmin_x F(x)+\frac{|x-x^\tau_k|^2}{2\tau}.
\end{equation}
We can forget now the convexity assumption on $F$, which are not necessary for this part of the analysis. Indeed, very mild assumptions on $F$ (l.s.c. and some lower bounds, for instance $F(x)\geq C_1-C_2|x|^2$) are sufficient to guarantee that these problems admit a solution for small $\tau$. The case where $F$ is $\lambda$-convex is covered by these assumptions, and also provides uniqueness of the minimizers. This is evident if $\lambda>0$ since we have strict convexity for every $\tau$, and if $\lambda$ is negative the sum will be strictly convex for small $\tau$.
 
We can interpret this sequence of points as the values of the curve $x(t)$ at times $t=0,\tau,2\tau,\dots, k\tau,\dots$. It happens that the optimality conditions of the recursive minimization exactly give a connection between these minimization problems and the equation, since we have
$$x^\tau_{k+1}\in \argmin F(x)+\frac{|x-x^\tau_k|^2}{2\tau}\quad\impl\quad \nabla F(x^\tau_{k+1})+\frac{x^\tau_{k+1}-x^\tau_k}{\tau}=0,$$
i.e.
$$ \frac{x^\tau_{k+1}-x^\tau_k}{\tau}=-  \nabla F(x^\tau_{k+1}).$$
This expression is exactly the discrete-time {\it implicit Euler scheme} for $x'=-\nabla F(x)$! (note that in the convex non-smooth case this becomes $ \frac{x^\tau_{k+1}-x^\tau_k}{\tau}\in-  \partial F(x^\tau_{k+1})$).

We recall that, given an ODE $x'(t)=\vv(x(t))$ (that we take autonomous for simplicity), with given initial datum $x(0)=x_0$, Euler schemes are time-discretization where derivatives are replaced by finite differences. We fix a time step $\tau>0$ and define a sequence $x^\tau_k$. The explicit scheme is given by
$$x^\tau_{k+1}=x^\tau_k+\tau  \vv(x^\tau_{k}),\quad x^\tau_0=x_0,$$
while the implicit scheme is given by 
$$x^\tau_{k+1}=x^\tau_k+\tau  \vv(x^\tau_{k+1}),\quad x^\tau_0=x_0.$$
This means that $x^\tau_{k+1}$ is selected as a solution of an equation involving $x^\tau_k$, instead of being explicitly computable from $x^\tau_k$. The explicit scheme is obviously easier to implement, but enjoys less stability and qualitative properties than the implicit one. Suppose for instance $ \vv=-\nabla F$: then the quantity $F(x(t))$ decreases in $t$ in the continuous solution, which is also the case for the implicit scheme, but not for the explicit one (which represents the iteration of the gradient method for the minimization of $F$). 
Note that the same can be done for evolution PDEs, and that solving the Heat equation $\partial_t\varrho=\Delta\varrho_t$
by an explicit scheme is very dangerous: at every step, $\varrho^\tau_{k+1}$ would have two degrees of regularity less than $\varrho^\tau_{k}$, since it would be obtained through $\varrho^\tau_{k+1}=\varrho^\tau_{k}-\tau\Delta\varrho^\tau_{k}$. 

It is possible to prove that, for $\tau\to 0$, the sequence we found, suitably interpolated, converges to the solution of Problem \eqref{subnabF}.
We give here below the details of this argument, as it will be the basis of the argument that we will use in Section \ref{W2}.

First, we define two different interpolations of the points $x^\tau_k$. Let us define two curves $x^\tau,\tilde x^\tau:[0,T]\to \R^n$ as follows: first we define 
$$\vv^\tau_{k+1}:= \frac{x^\tau_{k+1}-x^\tau_k}{\tau},$$
then we set
$$x^\tau(t)=x^\tau_{k+1} \qquad \tilde x^\tau(t)=x^\tau_{k}+(t-k\tau)\vv^\tau_{k+1} \,\mbox{ for }t\in ]k\tau,(k+1)\tau].$$
Also set 
$$\vv^\tau(t)=\vv^\tau_{k+1} \,\mbox{ for }t\in ]k\tau,(k+1)\tau].$$
It is easy to see that $\tilde x^\tau$ is a continuous curve, piecewise affine (hence absolutely continuous), satisfying $(\tilde x^\tau)'=\vv^\tau$. On the contrary, $x^\tau$ is not continuous, but satisfies by construction $\vv^\tau(t)\in -\partial F(x^\tau(t))$.

The iterated minimization scheme defining $x^\tau_{k+1}$ provides the estimate
 \begin{equation}\label{estimation compac-eucl}
F(x^\tau_{k+1})+\frac{|x^\tau_{k+1}-x^\tau_k|^2}{2\tau}\leq F(x^\tau_{k}),
\end{equation}
obtained comparing the optimal point $x^\tau_{k+1}$  to the previous one.
If $F(x_0)<+\infty$ and $\inf F>-\infty$, summing over $k$ we get 
\begin{equation}\label{1sttimeH1}
\sum_{k=0}^\ell \frac{|x^\tau_{k+1}-x^\tau_k|^2}{2\tau}\leq \left(F(x^\tau_{0})-F(x^\tau_{\ell+1})\right)\leq C.
\end{equation}
This is valid for every $\ell$, and we can arrive up to $\ell=\lfloor T/\tau\rfloor$. Now, note that 
$$\frac{|x^\tau_{k+1}-x^\tau_k|^2}{2\tau}=\tau\left(\frac{|x^\tau_{k+1}-x^\tau_k|}{2\tau}\right)^2=\tau |\vv^\tau_k|^2=\int_{k\tau}^{(k+1)\tau}|(\tilde x^\tau)'(t)|^2dt.$$
This means that we have 
\begin{equation}\label{H1compEucl}
\int_0^T\frac 12|(\tilde x^\tau)'(t)|^2dt\leq C
\end{equation}
and hence $\tilde x^\tau$ is bounded in $H^1$ and $\vv^\tau$ in $L^2$. The injection $H^1\subset C^{0,1/2}$ provides an equicontinuity bound on $\tilde x^\tau$ of the form 
\begin{equation}\label{H1compEucl2}
|\tilde x^\tau(t)-\tilde x^\tau(s)|\leq C|t-s|^{1/2}.
\end{equation}

This also implies 

\begin{equation}\label{H1compEucl3}
|\tilde x^\tau(t)- x^\tau(t)|\leq C\tau^{1/2},
\end{equation}

since $ x^\tau(t)=\tilde x^\tau(s)$ for a certain $s=k\tau$ with $|s-t|\leq \tau$.

This provides the necessary compactness to prove the following.
\begin{proposition}
Let $\tilde x^\tau$, $x^\tau$ and $\vv^\tau$ be constructed as above using the minimizing movement scheme. Suppose $F(x_0)<+\infty$ and $\inf F>-\infty$. Then, up to a subsequence $\tau_j\to 0$ (still denoted by $\tau$), both $\tilde x^\tau$ and $x^\tau$ converge uniformly to a same curve $x\in H^1$, and $\vv^\tau$ weakly converges in $L^2$ to a vector function $\vv$, such that $x'=\vv$ and
\begin{enumerate}
\item if $F$ is $\lambda$-convex, we have  $\vv(t)\in -\partial F(x(t))$ for a.e. $t$, i.e. $x$ is a solution of \eqref{subnabF};
\item if $F$ is $C^1$, we have $\vv(t)= -\nabla F(x(t))$ for all $t$, i.e. $x$ is a solution of \eqref{nabF}.
\end{enumerate}
\end{proposition}
\begin{proof} Thanks to the estimates \eqref{H1compEucl} and \eqref{H1compEucl2} and the fact that the initial point $\tilde x^\tau(0)$ is fixed, we can apply Ascoli-Arzel\`a's theorem to $\tilde x^\tau$ and get a uniformly converging subsequence. The estimate\eqref{H1compEucl3} implies that, on the same subsequence, $x^\tau$ also converges uniformly to the same limit, that we will call $x=[0,T]\to\R^n$. Then, $\vv^\tau=(\tilde x^\tau)'$ and \eqref{H1compEucl} allow to guarantee, up to an extra subsequence extraction, the weak convergence $\vv^\tau\deb \vv$ in $L^2$. The condition $x'=\vv$ is automatical as a consequence of distributional convergence.

To prove 1), we will fix a point $y\in\R^n$ and write
$$F(y)\geq F(x^\tau(t))+\vv^\tau(t)\cdot (y-x^\tau(t))+\frac\lambda 2|y-x^\tau(t)|^2.$$
We then multiply by a positive measurable function $a:[0,T]\to \R_+$ and integrate:
$$\int_0^T a(t)\left(F(y)- F(x^\tau(t))-\vv^\tau(t)\cdot (y-x^\tau(t))+\frac\lambda 2|y-x^\tau(t)|^2\right)\dd t\geq 0.$$
We can pass to the limit as $\tau\to 0$, using the uniform (hence $L^2$ strong) convergence $x^\tau\to x$ and the weak convergence $\vv^\tau\deb \vv$. In terms of $F$, we just need its lower semi-continuity. This provides
$$\int_0^T a(t)\left(F(y)- F(x(t))-\vv(t)\cdot (y-x(t))+\frac\lambda 2|y-x(t)|^2\right)\dd t\geq 0.$$
From the arbitrariness of $a$, the inequality 
$$F(y)\geq F(x(t))+\vv(t)\cdot (y-x(t))-\frac\lambda 2|y-x(t)|^2$$
is true for a.e. $t$ (for fixed $y$). Using $y$ in a dense countable set in the interior of $\{F<+\infty\}$ (where $F$ is continuous), we get $\vv(t)\in \partial F(x(t))$.

To prove 2), the situation is easier. Indeed we have
$$-\nabla F(x^\tau(t))=\vv^\tau(t)=(\tilde x^\tau)'(t).$$
The first term in the equality uniformly converges, as a function of $t$, (since $\nabla F$ is continuous and $x^\tau$ lives in a compact set) to $-\nabla F(x)$, the second weakly converges to $\vv$ and the third to $x'$. This proves the claim, and the equality is now true for every $t$ as the function $t\mapsto -\nabla F(x(t))$ is uniformly continuous.
\end{proof}

In the above result, we only proved convergence of the curves $x^\tau$ to the limit curve $x$, solution of $x'=-\nabla F(x)$ (or $-x'\in \partial F(x)$), but we gave no quantitative order of convergence, and we will not study such an issue in the rest of the survey neither. On the contrary, the book \cite{AmbGigSav} which will be the basis for the metric case, also provides explicit estimates; these estimates are usually of order $\tau$. An interesting observation, in the Euclidean case, is that if the sequence $x^\tau_k$ is defined by
$$x^\tau_{k+1} \in \argmin_x \quad 2F\left(\frac{x+x^\tau_k}{2}\right)+\frac{|x-x^\tau_k|^2}{2\tau},$$
then we have
$$\frac{x^\tau_{k+1}-x^\tau_{k}}{\tau}=-\nabla F \left(\frac{x+x^\tau_k}{2}\right),$$
and the convergence is of order $\tau^2$. This has been used in the Wasserstein case\footnote{The attentive reader can observe that, setting $y:=(x+x^\tau_k)/2$, this minimization problem becomes $\min_y 2F(y)+2|y-x^\tau_k|^2/\tau$. Yet, when acting on a metric space, or  simply on a manifold or a bounded domain, there is an extra constraint on $y$: the point $y$ must be the middle point of a geodesic between $x^\tau_k$ and a point $x$ (on a sphere, for instance, this means that if $x^\tau_k$ is the North Pole, then $y$ must lie in the northern emisphere).} (see Section \ref{W2} and in particular Section \ref{W2num}) in \cite{LegTur}.

\subsection{An introduction to the metric setting}

The iterated minimization scheme that we introduced above has another interesting feature: it even suggests how to define solutions for functions $F$ which are only l.s.c., with no gradient at all! 

Even more, a huge advantage of this discretized formulation is also that it can easily be adapted to metric spaces. Actually, if one has a metric space $(X,d)$ and a l.s.c. function $F:X\to\R\cup\{+\infty\}$ (under suitable compactness assumptions to guarantee existence of the minimum), one can define 
\begin{equation}\label{iterated min}
x^\tau_{k+1}\in\argmin_x F(x)+\frac{d(x,x^\tau_k)^2}{2\tau}
\end{equation}
and study the limit as $\tau\to 0$. Then, we use the piecewise constant interpolation
\begin{equation}\label{interp const}
x^\tau(t):=x^\tau_k\quad\mbox{ for every }t\in](k-1)\tau,k\tau]
\end{equation}
and study the limit of $x^\tau$ as $\tau\to 0$.

De Giorgi, in \cite{DeG}, defined what he called {\it Generalized Minimizing Movements}\footnote{We prefer not to make any distinction here between Generalized Minimizing Movements and Minimizing Movements.}:
\begin{definition}
A curve $x:[0,T]\to X$ is called Generalized Minimizing Movements (GMM) if there exists a sequence of time steps $\tau_j\to 0$ such that the sequence of curves $x^{\tau_j}$ defined in \eqref{interp const} using the iterated solutions of \eqref{iterated min} uniformly converges to $x$ in $[0,T]$.
\end{definition}
The compactness results in the space of curves guaranteeing the existence of GMM are also a consequence of an H\"older estimate that we already saw in the Euclidean case. Yet, in the general case, some arrangements are needed, as we cannot use the piecewise affine interpolation. We will see later that, in case the segments may be replaced by geodesics, a similar estimate can be obtained. Yet, we can also obtain a H\"older-like estimate from the piecewise constant interpolation. 

We start from
 \begin{equation}\label{estimation compac}
F(x^\tau_{k+1})+\frac{d(x^\tau_{k+1},x^\tau_k)^2}{2\tau}\leq F(x^\tau_{k}),
\end{equation}
and
$$\sum_{k=0}^l d(x^\tau_{k+1},x^\tau_k)^2\leq 2\tau\left(F(x^\tau_{0})-F(x^\tau_{l+1})\right)\leq C\tau.$$
The Cauchy-Schwartz inequality gives, for $t<s$, $t\in[k\tau,(k+1)\tau[$ and $s\in[l\tau,(l+1)\tau[$ (which implies $|l-k|\leq \frac{|t-s|}{\tau}+1$),
$$d(x^\tau(t),x^\tau(s))\leq\! \sum_{k=0}^l d(x^\tau_{k+1},x^\tau_k)\leq\left(\sum_{k=0}^l d(x^\tau_{k+1},x^\tau_k)^2\!\right)^{\!1/2}\!\!\left( \frac{|t-s|}{\tau}\!+\!1\!\right)^{\!1/2}\!\leq C\left(|t-s|^{1/2}\!\!+\!\sqrt{\tau}\right).$$
This means that the curves $x^\tau$ - if we forget that they are discontinuous - are morally equi-h\"older continuous with exponent $1/2$ (up to a negligible error of order $\sqrt\tau$), and allows to extract a converging subsequence.

Anyway, if we add some structure to the metric space $(X,d)$, a more similar analysis to the Euclidean case can be performed. This is what happens when we suppose that $(X,d)$ is a {\it geodesic space}. This requires a short discussion about curves and geodesics in metric spaces.\medskip

{\bf Curves and geodesics in metric spaces.}
We recall that a curve $\omega$ is a continuous function defined on a interval, say $[0,1]$ and valued in a metric space $(X,d)$. As it is a map between metric spaces, it is meaningful to say whether it is Lipschitz or not, but its speed $\omega'(t)$ has no meaning, unless $X$ is a vector space. Surprisingly, it is possible to give a meaning to the modulus of the velocity, $|\omega'|(t)$.
\begin{definition} If $\omega:[0,1]\to X$ is a curve valued in the metric space $(X,d)$ we define the metric derivative of $\omega$ at time $t$, denoted by $|\omega'|(t)$ through
$$ |\omega'|(t):=\lim_{h\to 0}\frac{d(\omega(t+h),\omega(t))}{|h|},$$
provided this limit exists.
\end{definition}

In the spirit of Rademacher Theorem, it is possible to prove (see \cite{AmbTil}) that, if $\omega:[0,1]\to X$ is Lipschitz continuous, then the metric derivative $|\omega'|(t)$ exists for a.e. $t$. Moreover we have, for $t_0<t_1$,
$$d(\omega(t_0),\omega(t_1))\leq \int_{t_0}^{t_1} |\omega'|(s)\,\dd s.$$
The same is also true for more general curves, not only Lipschitz continuous. 

\begin{definition} A curve $\omega:[0,1]\to X$ is said to be {\it absolutely continuous} whenever there exists $g\in L^1([0,1])$ such that $d(\omega(t_0),\omega(t_1))\leq \int_{t_0}^{t_1}g(s)\dd s$ for every $t_0<t_1$.  The set of absolutely continuous curves defined on $[0,1]$ and valued in $X$ is denoted by $\mathrm{AC}(X)$.
\end{definition}

It is well-known that every absolutely continuous curve can be reparametrized in time (through a monotone-increasing reparametrization) and become Lipschitz continuous, and 
%
the existence of the metric derivative for a.e. $t$ is also true for $\omega\in\mathrm{AC}(X)$, via this reparametrization.

Given a continuous curve, we can also define its length, and the notion of geodesic curves.

\begin{definition} 
For a curve $\omega:[0,1]\to X$, let us define
$$\len(\omega):=\sup\left\{\sum_{k=0}^{n-1}d(\omega(t_k),\omega(t_{k+1}))\,:\,n\geq 1,\, 0=t_0<t_1<\dots<t_n=1\right\}.$$
\end{definition}
It is easy to see that all curves $\omega\in \mathrm{AC}(X)$ satisfy $\len(\omega)\leq\int_0^1 g(t)\dd t<+\infty$.
Also, we can prove that, for any curve $\omega\in\mathrm{AC}( X)$, we have
$$\len(\omega)=\int_0^1 |\omega'|(t)\dd t.$$

We collect now some more definitions.

\begin{definition} A curve $\omega:[0,1]\to X$ is said to be a geodesic between $x_0$ and $x_1\in X$ if $\omega(0)=x_0$, $\omega(1)=x_1$ and $\len(\omega)=\min\{\len(\tilde\omega)\,:\,\tilde\omega(0)=x_0,\,\tilde\omega(1)=x_1\}$. 

A space $(X,d)$ is said to be a {\it length space} if for every $x$ and $y$ we have
$$d(x,y)=\inf\{\len(\omega)\,:\,\omega\in\mathrm{AC}(X),\,\omega(0)=x,\,\omega(1)=y\}.$$

A space $(X,d)$ is said to be a {\it geodesic space} if for every $x$ and $y$ we have
$$d(x,y)=\min\{\len(\omega)\,:\,\omega\in\mathrm{AC}(X),\,\omega(0)=x,\,\omega(1)=y\},$$
i.e. if it is a length space and there exist geodesics between arbitrary points.
\end{definition}
 In a length space, a curve $\omega:[t_0,t_1]\to X$ is said to be a {\it constant-speed geodesic} between $\omega(0)$ and $\omega(1)\in X$ if it satisfies
$$d(\omega(t),\omega(s))=\frac{|t-s|}{t_1-t_0}d(\omega(t_0),\omega(t_1))\quad\mbox{ for all }t,s\in [t_0,t_1].$$
It is easy to check that a curve with this property is automatically a geodesic, and that the following three facts are equivalent (for arbitrary $ p>1$)
\begin{enumerate}
\item $\omega$ is a constant-speed geodesic defined on $ [t_0,t_1]$ and joining $x_0$ and $x_1$,
\item $\omega\in\mathrm{AC}(X)$ and $|\omega'|(t)=\frac{d(\omega(t_0),\omega(t_1))}{t_1-t_0}$ a.e., 
\item $\omega$ solves $\min\left\{\int_{t_0}^{t_1} |\omega'|(t)^p\dd t\,:\,\omega(t_0)=x_0,\omega(t_1)=x_1\right\}$. 
\end{enumerate}
\bigskip

We can now come back to the interpolation of the points obained through the Minimizing Movement scheme \eqref{iterated min} and note that, if $(X,d)$ is a geodesic space, then the piecewise affine interpolation that we used in the Euclidean space may be helpfully replaced via a piecewise geodesic interpolation. This means defining a curve $x^\tau:[0,T]\to X$ such that $x^\tau(k\tau)=x^\tau_k$ and such that $x^\tau$ restricted to any interval $[k\tau,(k+1)\tau]$ is a constant-speed geodesic with speed equal to $d(x^\tau_k,x^\tau_{k+1})/\tau$. Then, the same computations as in the Euclidean case allow to prove an $H^1$ bound on the curves $x^\tau$ (i.e. an $L^2$ bound on the metric derivatives $|(x^\tau)'|$) and prove equicontinuity.

The next question is how to characterize the limit curve obtained when $\tau\to 0$, and in particular how to express the fact that it is a gradient flow of the function $F$. Of course, one cannot try to prove the equality $x'=-\nabla F(x)$, just because neither the left-hand side nor the right-hand side have a meaning in a metric space!

If the space $X$, the distance $d$, and the functional $F$ are explicitly known, in some cases it is possible to pass to the limit the optimality conditions of each optimization problem in the discretized scheme, and characterize the limit curves (or the limit curve) $x(t)$. It will be possible to do so in the framework of probability measures, as it will be discussed in Section \ref{W2}, but not in general. Indeed, without a little bit of (differential) structure on the space $X$, it is essentially impossible to do so. Hence, if we want to develop a general theory for gradient flows in metric spaces, finer tools are needed. In particular, we need to characterize the solutions of $x'=-\nabla F(x)$ (or $x'\in -\partial F(x)$) by only using metric quantities (in particular, avoiding derivatives, gradients, and more generally vectors). The book by Ambrosio-Gigli-Savar\'e \cite{AmbGigSav}, and in particular its first part (the second being devoted to the space of probability measures) exactly aims at doing so. 

Hence, what we do here is to present alternative characterizations of gradient flows in the smooth Euclidean case, which can be used as a definition of gradient flow in the metric case, since all the quantities which are involved have their metric counterpart. 

The first observation is the following: thanks to the Cauchy-Schwartz inequality, for every curve we have
\begin{eqnarray*}
F(x(s))-F(x(t))=\int_s^t -\nabla F(x(r))\cdot x'(r)\;\dd r&\leq& \int_s^t |\nabla F(x(r))|| x'(r)|\;\dd r\\
 &\leq& \int_s^t  \left(\frac12 |x'(r)|^2 + \frac 12 |\nabla F(x(r))|^2\right)\dd r.
 \end{eqnarray*}
Here, the first inequality is an equality if and only if $x'(r)$ and $\nabla F(x(r))$ are vectors with opposite directions for a.e. $r$, and the second is an equality if and only if their norms are equal. Hence, the condition, called EDE ({\it Energy Dissipation Equality})\index{EDE condition}
$$F(x(s))-F(x(t))=\int_s^t  \left(\frac12 |x'(r)|^2 + \frac 12 |\nabla F(x(r))|^2\right)\dd r,\quad\mbox{ for all } s<t$$
(or even the simple inequality $F(x(s))-F(x(t))\geq \int_s^t  \left(\frac12 |x'(r)|^2 + \frac 12 |\nabla F(x(r))|^2\right)\dd r$) is equivalent to $x'=-\nabla F(x)$ a.e., and could be taken as a definition of gradient flow.

In the general theory of Gradient Flows (\cite{AmbGigSav}) in metric spaces, another characterization, different from the EDE, is proposed in order to cope with uniqueness and stability results. It is based on the following observation: if $F:\R^d\to\R$ is convex, then the inequality
$$F(y)\geq F(x)+p\cdot (y-x)\quad\mbox{ for all }y\in\R^d$$
characterizes (by definition) the vectors $p\in\partial F(x)$ and, if $F\in C^1$, it is only satisfied for $p=\nabla F(x)$. Analogously, if $F$ is $\lambda$-convex, the inequality that characterizes the gradient is
$$F(y)\geq F(x)+\frac{\lambda}{2}|x-y|^2+p\cdot (y-x)\quad\mbox{ for all }y\in\R^d.$$
We can pick a curve $x(t)$ and a point $y$ and compute
$$\frac{d}{dt}\frac 12|x(t)-y|^2=(y-x(t))\cdot (-x'(t)).$$
Consequently, imposing
$$\frac{d}{dt}\frac 12|x(t)-y|^2\leq F(y)-F(x(t))-\frac{\lambda}{2}|x(t)-y|^2,$$
for all $y$, will be equivalent to $-x'(t)\in -\partial F(x(t))$. This will provide a second characterization (called EVI, {\it Evolution Variational Inequality}) of gradient flows in a metric environment. Indeed, all the terms appearing in the above inequality have a metric counterpart (only squared distances and derivatives w.r.t. time appear). Even if we often forget the dependance on $\lambda$, it should be noted that the condition EVI should actually be written as EVI$_\lambda$, since it involves a parameter $\lambda$, which is a priori arbitrary. Actually, 
$\lambda$-convexity of $F$ is not necessary to define the EVI$_\lambda$ property, but it will be necessary in order to guarantee the existence of curves which satisfy such a condition. The notion of $\lambda$-convexity\index{$\lambda$-convexity} will hence be crucial also in metric spaces, where it will be rather ``$\lambda$-geodesic-convexity''.

The role of the EVI condition in the uniqueness and stability of gradient flows is quite easy to guess. Take two curves, that we call $x(t)$ and $y(s)$, and compute 
\begin{gather}\frac{d}{dt}\frac 12d(x(t),y(s))^2\leq F(y(s))-F(x(t))-\frac{\lambda}{2}d(x(t),y(s))^2,\\
			\frac{d}{ds}\frac 12d(x(t),y(s))^2\leq F(x(t))-F(y(s))-\frac{\lambda}{2}d(x(t),y(s))^2.\end{gather}
If one wants to estimate $E(t)=\frac 12d(x(t),y(t))^2$, summing up the two above inequalities, after a chain-rule argument for the composition of the function of two variables $(t,s)\mapsto \frac 12d(x(t),y(s))^2$ and of the curve $t\mapsto (t,t)$, gives
$$\frac{d}{dt} E(t)\leq -2\lambda E(t).$$
By Gronwall Lemma, this provides uniqueness (when $x(0)=y(0)$) and stability.

\section{The general theory in metric spaces}\label{metricth}

\subsection{Preliminaries}

In order to sketch the general theory in metric spaces, first we need to give (or recall) general definitions for the three main objects that we need in the  EDE and EVI properties, characterizing gradient flows:  the notion of speed of a curve, that of slope of a function (somehow the modulus of its gradient) and that of (geodesic) convexity. 

{\bf Metric derivative.} We already introduced in the previous section the notion of metric derivative: given a curve $x:[0,T]\to X$ valued in a metric space, we can define, instead of the velocity $x'(t)$ as a vector (i.e, with its direction, as we would do in a vector space), the speed (i.e. the modulus, or norm, of $x'(t)$) as follows:
$$|x'|(t):=\lim_{h\to 0}\frac{d(x(t),x(t+h))}{|h|},$$
provided the limit exists.This is the notion  of speed that we will use in metric spaces.

{\bf Slope and modulus of the gradient.} Many definitions of the modulus of the gradient of a function $F$ defined over a metric space are possible. First, we call {\it upper gradient} every function $g:X\to \R$ such that, for every Lipschitz curve $x$, we have
$$|F(x(0))-F(x(1))|\leq \int_0^1 g(x(t))|x'|(t)\dd t.$$
If $F$ is Lipschitz continuous, a possible choice is the {\it local Lipschitz constant}
\begin{equation}\label{const lip loc}
|\nabla F|(x):=\limsup_{y\to x}\frac{|F(x)-F(y)|}{d(x,y)};
\end{equation}
another is the {\it descending slope} (we will often say just {\it slope}), which is a notion more adapted to the minimization of a function than to its maximization, and hence reasonable for lower semi-continuous functions:
$$|\nabla^- F|(x):=\limsup_{y\to x}\frac{[F(x)-F(y)]_+}{d(x,y)}$$
(note that the slope vanishes at every local minimum point). In general, it is not true that the slope is an upper gradient, but we will give conditions to guarantee that it is. Later on (Section \ref{heat}) we will see how to define a Sobolev space $H^1$ on a (measure) metric space, by using suitable relaxations of the modulus of the gradient of $F$.

{\bf Geodesic convexity.} The third notion to be dealt with is that of convexity. This can only be done in a geodesic metric space. On such a space, we can say that a function is {\it geodesically convex} whenever it is convex along geodesics. More precisely, we require that for every pair $(x(0),x(1))$ there exists\footnote{Warning: this definition is not equivalent to true convexity along the geodesic, since we only compare intermediate instants $t$ to $0$ and $1$, and not to other interemediate instants; however, in case of uniqueness of goedesics, or if we required the same condition to be true for all geodesics, then we would recover the same condition. Also, let us note that we will only need the existence of geodesics connecting pairs of points where $F<+\infty$.} a geodesic $x$ with constant speed connecting these two points and such that 
$$F(x(t))\leq (1-t)F(x(0))+tF(x(1)).$$
We can also define $\lambda$-convex functions as those which satisfy a modified version of the above inequality: 
\begin{equation}\label{defilambda}
F(x(t))\leq (1-t)F(x(0))+tF(x(1))-\lambda\frac{t(1-t)}{2}d^2(x(0),x(1)).
\end{equation}

\subsection{Existence of a gradient flow}\label{3.2}
Once fixed these basic ingredients, we can now move on to the notion of gradient flow.
A starting approach is, again, the sequential minimization along a discrete scheme, for a fixed time step $\tau>0$, and then pass to the limit. First, we would like to see in which framework this procedure is well-posed. Let us suppose that the space $X$ and the function $F$ are such that every sub-level set $\{ F\leq c\}$ is compact in $X$, either for the topology induced by the distance $d$, or for a weaker topology, such that $d$ is lower semi-continuous w.r.t. it;  $F$ is required to be l.s.c. in the same topology. This is the minimal framework to guarantee existence of the minimizers at each step, and to get estimates as in \eqref{estimation compac} providing the existence of a limit curve. It is by the way a quite general situation, as we can see in the case where $X$ is a reflexive Banach space and the distance $d$ is the one induced by the norm: in this case there is no need to restrict to the (very severe) assumption that $F$ is strongly l.s.c., but the weak topology allows to deal with a much wider situation. 

We can easily understand that, even if the estimate \eqref{estimation compac} is enough to provide compactness, and thus the existence of a GMM, it will never be enough to characterize the limit curve (indeed, it is satisfied by any discrete evolution where $x^\tau_{k+1}$ gives a better value than  $x^\tau_k$, without any need for optimality). Hence, we will never obtain either of the two formulations - EDE or EVI - of metric gradient flows.

In order to improve the result, we should exploit how much $x^\tau_{k+1}$ is better than $x^\tau_k$. An idea due to De Giorgi allows to obtain the desired result, via a ``variational interpolation'' between the points $x^\tau_k$ and $x^\tau_{k+1}$. In order to do so, once we fix $x^\tau_k$, for every $\theta\in]0,1]$, we consider the problem
$$\min_x\quad F(x)+\frac{d^2(x,x^\tau_k)}{2\theta\tau}$$
and call $x(\theta)$ any minimizer for this problem, and $\varphi(\theta)$ the minimal value. It's clear that, for $\theta\to 0^+$, we have $x(\theta)\to x^\tau_k$ and $\varphi(\theta)\to F(x^\tau_k)$, and that, for $\theta=1$, we get back to the original problem with minimizer $x^\tau_{k+1}$. Moreover, the function $\varphi$ is non-increasing and hence a.e. differentiable (actually, we can even prove that it is locally semiconcave). Its derivative $\varphi'(\theta)$ is given by the derivative of the function $\theta\mapsto F(x)+\frac{d^2(x,x^\tau_k)}{2\theta\tau}$, computed at the optimal point $x=x(\theta)$ (the existence of $\varphi'(\theta)$ implies that this derivative is the same at every minimal point $x(\theta)$). Hence we have
$$\varphi'(\theta)=-\frac{d^2(x(\theta),x^\tau_k)}{2\theta^2\tau},$$
which means, by the way, that $d(x(\theta),x^\tau_k)^2$ does not depend on the minimizer $x(\theta)$ for all $\theta$ such that $\varphi'(\theta)$ exists.
Moreover, the optimality condition for the minimization problem with $\theta>0$ easily show that 
$$|\nabla^- F|(x(\theta))\leq \frac{d(x(\theta),x^\tau_k)}{\theta\tau}.$$ 
This can be seen if we consider the minimization of an arbitrary function $x\mapsto F(x)+cd^2(x,\bar x)$, for fixed $c>0$ and $\bar x$, and we consider a competitor $y$. If $x$ is optimal we have $F(y)+cd^2(y,\bar x)\geq F(x)+cd^2(x,\bar x)$, which implies
$$F(x)-F(y)\leq c\left(d^2(y,\bar x)-d^2(x,\bar x)\right)=c\left(d(y,\bar x)+d(x,\bar x)\right)\left(d(y,\bar x)-d(x,\bar x)\right)
\leq c\left(d(y,\bar x)+d(x,\bar x)\right)d(y,x).
$$
We divide by $d(y,x)$, take the positive part and then the $\limsup$ as $y\to x$, and we get $|\nabla^- F|(x)\leq 2cd(x,\bar x)$.

We now come back to the function $\varphi$ and use 
$$\varphi(0)-\varphi(1)\geq -\int_0^1 \varphi'(\theta)\,\dd\theta$$ 
(the inequality is due to the possible singular part of the derivative for monotone functions; actually, we can prove that it is an equality by using the local semiconcave behavior, but this is not needed in the following), together with the inequality
$$-\varphi'(\theta)=\frac{d(x(\theta),x^\tau_k)^2}{2\theta^2\tau}\geq \frac{\tau}{2}|\nabla^- F(x(\theta))|^2$$
that we just proved. Hence, we get an improved version of \eqref{estimation compac}:
$$
F(x^\tau_{k+1})+\frac{d(x^\tau_{k+1},x^\tau_k)^2}{2\tau}\leq F(x^\tau_{k})-\frac \tau 2\int_0^1|\nabla^- F(x(\theta))|^2\dd\theta.$$
If we sum up for $k=0,1,2,\dots$ and then take the limit $\tau\to 0$, we can prove, for every GMM $x$, the inequality
\begin{equation}\label{EDE ottentua}
F(x(t))+\frac 12\int_0^t |x'|(r)^2\dd r+\frac 12\int_0^t |\nabla^- F(x(r))|^2\dd r\leq F(x(0)),
\end{equation}
under some suitable assumptions that we must select. In particular, we need lower-semicontinuity of $F$ in order to handle the term $F(x^\tau_{k+1})$ (which will become $F(x(t))$ at the limit), but we also need  lower-semicontinuity of the slope $|\nabla^- F|$  in order to handle the corresponding term.

This inequality does not exactly correspond to EDE: on the one hand we have an inequality, and on the other we just compare instants $t$ and $0$ instead of $t$ and $s$. If we want equality for every pair $(t,s)$, we need to require the slope to be an upper gradient. Indeed, in this case, we have the inequality $F(x(0))-F(x(t))\leq \int_0^t |\nabla^- F(x(r))||x'|(r)\dd r$ and, starting from the usual inequalities, we find that \eqref{EDE ottentua} is actually an equality. This allows to subtract the equalities for $s$ and $t$, and get, for $s<t$:
$$F(x(t))+\frac 12\int_s^t |x'|(r)^2\dd r+\frac 12\int_s^t |\nabla^- F(x(r))|^2\dd r= F(x(s)).$$

Magically, it happens that the assumption that $F$ is $\lambda$-geodesically convex simplifies everything. Indeed, we have two good points: the slope is automatically l.s.c., and it is automatically an upper gradient. These results are proven in \cite{AmbGigSav,usersguide}. We just give here the main idea to prove both. This idea is based on a pointwise representation of the slope as a $\sup$ instead of a $\limsup$: if $F$ is $\lambda$-geodesically convex, then we have
\begin{equation}\label{point sup lamcon}
|\nabla^- F|(x)=\sup_{y\neq x}\left[\frac{F(x)-F(y)}{d(x,y)}+\frac \lambda 2 d(x,y)\right]_+.
\end{equation}
In order to check this, we just need to add a term $\frac \lambda 2 d(x,y)$ inside the positive part of the definition of $|\nabla^- F|(x)$, which does not affect the limit as $y\to x$ and shows that $|\nabla^- F|(x)$ is smaller than this $\sup$. The opposite inequality is proven by fixing a point $y$, connecting it to $x$ through a geodesic $x(t)$, and computing the limit along this curve.

This representation as a $\sup$ allows to prove semicontinuity of the slope \footnote{Warning: we get here semi-continuity w.r.t. the topology induced by the distance $d$, which only allows to handle the case where the set $\{F\leq c\}$ are $d$-compacts.}. It is also possible (see \cite{usersguide}, for instance) to prove that the slope is an upper gradient.

Let us insist anyway on the fact that the $\lambda$-convexity assumption is not natural nor crucial to prove the existence of a gradient flow. On the one hand, functions smooth enough could satisfy the assumptions on the semi-continuity of $F$ and of $|\nabla^-F|$ and the fact that $|\nabla^-F|$ is an upper gradient independently of convexity; on the other hand the discrete scheme already provides a method, well-posed under much weaker assumptions, to find a limit curve. If the space and the functional allow for it (as it will be the case in the next section), we can hope to characterize this limit curve as the solution of an equation (it will be a PDE in Section \ref{W2}), without passing through the general theory and the EDE condition.

\subsection{Uniqueness and contractivity}
On the contrary, if we think at the uniqueness proof that we gave in the Euclidean case, it seems that some sort of convexity should be the good assumption in order to prove uniqueness. Here we will only give the main lines of the uniqueness theory in the metric framework: the key point is to use the EVI condition instead of the EDE. 

The situation concerning these two different notions of gradient flows (EVI and EDE) in abstract metric spaces has been clarified by Savar\'e (in an unpublished note, but the proof can also be found in \cite{usersguide}), who showed that
\begin{itemize}
\item All curves which are gradient flows in the EVI sense also satisfy the EDE condition.
\item The EDE condition is not in general enough to guarantee uniqueness of the gradient flow. A simple example: take $X=\R^2$ with the $\ell^\infty$ distance 
$$d((x_1,x_2),(y_1,y_2))=|x_1-y_1|\vee|x_2-y_2|,$$ 
and take $F(x_1,x_2)=x_1$; we can check that any curve $(x_1(t),x_2(t))$ with $x_1'(t)=-1$ and $|x_2'(t)|\leq 1$ satisfies EDE. 
\item On the other hand, existence of a gradient flow in the EDE sense is quite easy to get, and provable under very mild assumption, as we sketched in Section \ref{3.2}.
\item The EVI condition is in general too strong in order to get existence (in the example above of the $\ell^\infty$ norm, no EVI gradient flow would exist), but always guarantees uniqueness and stability (w.r.t. initial data).
\end{itemize}

Also, the existence of EVI gradient flows is itself very restricting on the function $F$: indeed, it is proven in \cite{DanSav} that, if $F$ is such that from every starting point $x_0$ there exists an EVI$_\lambda$ gradient flow, then $F$ is necessarily $\lambda$-geodesically-convex.

We provide here an idea of the proof of the contractivity (and hence of the uniqueness) of the EVI gradient flows.

\begin{proposition}\label{uniqEVI}
 If two curves $x,y:[0,T]\to X$ satisfy the EVI condition, then we have
 $$\frac{d}{dt}d(x(t),y(t))^2\leq -2\lambda d(x(t),y(t))^2$$
 and $d(x(t),y(t))\leq e^{-\lambda t}d(x(0),y(0))$.
 \end{proposition}
The second part of the statement is an easy consequence of the first one, by Gronwall Lemma. The first is (formally) obtained by differentiating $t\mapsto d(x(t),y(t_0))^2$ at $t=t_0$, then $s\mapsto d(x(t_0),y(s))^2$ at $s=t_0$.The EVI condition allows to write
\begin{gather*}
\frac{d}{dt}d(x(t),y(t_0))^2_{|t=t_0}\leq-\lambda  d(x(t_0),y(t_0))^2+2F(y(t_0))-2F(x(t_0))\\
\frac{d}{ds}d(x(t_0),y(s))^2_{|s=t_0}\leq - \lambda  d(x(t_0),y(t_0))^2+2F(x(t_0))-2F(y(t_0))
\end{gather*}
and hence, summing up, and playing with the chain rule for derivatives, we get
$$\frac{d}{dt}d(x(t),y(t))^2\leq -2\lambda d(x(t),y(t))^2.$$
 
If we want a satisfying theory for gradient flows which includes uniqueness, we just need to prove the existence of curves which satisfy the EVI condition, accepting that this will probably require additional assumptions. This can still be done via the discrete scheme, adding a compatibility hypothesis between the function $F$ and the distance $d$, a condition which involves some sort of convexity. We do not enter the details of the proof, for which we refer to \cite{AmbGigSav}, where the convergence to an EVI gradient flow is proven, with explicit error estimates. These a priori estimates allow to prove that we have a Cauchy sequence, and then allow to get rid of the compactness part of the proof (by the way, we could even avoid using compactness so as to prove existence of a minimizer at every time step, using almost-minimizers and the in the Ekeland's variational principle \cite{Eke-VarPrin}). Here, we will just present this extra convexity assumption, needed for the existence of EVI gradient flows developed in \cite{AmbGigSav}.

This assumption, that we will call C$^2$G$^2$ ({\it Compatible Convexity along Generalized Geodesics}) is the following: suppose that, for every pair $(x_0,x_1)$ and every $y\in X$, there is a curve $x(t)$ connecting $x(0)=x_0$ to $x(1)=x_1$, such that
 \begin{gather*}
F(x(t))\leq (1-t)F(x_0)+tF(x_1)-\lambda\frac{t(1-t)}{2}d^2(x_0,x_1),\\
d^2(x(t),y)\leq (1-t)d^2(x_0,y)+td^2(x_1,y)-t(1-t)d^2(x_0,x_1).
\end{gather*}

In other words, we require $\lambda$-convexity of the function $F$, but also the $2$-convexity of the function $x\mapsto d^2(x,y)$, along a same curve which is not necessarily the geodesic. This second condition is automatically satisfied, by using the geodesic itself, in the Euclidean space (and in every Hilbert space), since the function $x\mapsto |x-y|^2$ is quadratic, and its Hessian matrix is $2I$ at every point. We can also see that it is satisfied in a normed space if and only if the norm is induced by a scalar product. It has been recently pointed out by Gigli that the sharp condition on the space $X$ in order to guarantee existence of EVI gradient flows is that $X$ should be {\it infinitesimally Hilbertian} (this will be made precise in Section \ref{heat}).

Here, we just observe that C$^2$G$^2$ implies $(\lambda+\frac 1\tau)$-convexity, along those curves, sometimes called {\it generalized geodesics} (consider that these curves also depend on a third point, sort of a base point, typically different from the two points that should be connected), of the functional that we minimize at each time step in the minimizing movement scheme. This provides uniqueness of the minimizer as soon as $\tau$ is small enough, and allows to perform the desired estimates. 

Also, the choice of this C$^2$G$^2$ condition, which is a technical condition whose role is only to prove existence of an EVI gradient flow, has been done in view of the applications to the case of the Wasserstein spaces, that wil be the object of the next section. Indeed, in these spaces the squared distance is not in general $2$-convex along geodesics, but we can find some adapted curves on which it is $2$-convex, and many functionals $F$ stay convex on these curves.

We finish this section by mentioning a recent extension to some non-$\lambda$-convex functionals. The starting point is the fact that the very use of Gronwall's lemma to prove uniqueness can be modified by allowing for a weaker condition. Indeed, it is well-known that, whenever a function $\omega$ satisfies an Osgood condition $\int_0^1\frac{1}{\omega(s)}\dd s=+\infty$, then $E'\leq \omega(E)$ together with $E(0)=0$ implies $E(t)=0$ for $t>0$. This suggests that one could define a variant of the EVI definition for functions which are not $\lambda-$convex, but almost, and this is the theory developed in \cite{KatyCraig}. Such a paper studies the case where $F$ satisfies some sort of $\omega$-convexity for a ``modulus of convexity'' $\omega$. More precisely, this means
$$F(x_t)\leq (1-t)F(_0)+tF(x_1)-\frac{|\lambda|}{ 2} [(1-t)\omega(t^2d(x_0,x_1)^2)+t\omega((1-t)^2d(x_0,x_1)^2)],$$
on generalized geodesics $x_t$ (note that in the case $\omega(s)=s$ we come back to \eqref{defilambda}). The function $\omega$ is required to satisfy an Osgood condition (and some other technical conditions). Then, the EVI condition is replaced by
$$\frac{d}{dt}\frac 12d(x(t),y)^2\leq F(y)-F(x(t))+\frac{|\lambda|}{2}\omega(d(x(t),y)^2),$$
and this allows to produce a theory with existence and uniqueness results (via a variant of Proposition \ref{uniqEVI}). In the Wasserstein spaces (see next section), a typical case of functionals which can fit this theory are functionals involving singular interaction kernels (or solutions to elliptic PDEs, as in the Keller-Segel case) under $L^\infty$ constraints on  the density (using the fact that the gradient $\nabla u$ of the solution of $-\Delta u=\varrho$ is not Lipschitz when $\varrho\in L^\infty$, but is at least log-lipschitz).

\section{Gradient flows in the Wasserstein space}\label{W2}

One of the most exciting applications (and maybe the only one\footnote{This is for sure exaggerated, as we could think for instance at the theory of geometrical evolutions of shapes and sets, even if it seems that this metric approach has not yet been generalized in this framework.}, in what concerns applied mathematics) of the theory of gradient flows in metric spaces is for sure that of evolution PDEs in the space of measures. This topic is inspired from the work by Jordan, Kinderlehrer and Otto (\cite{JKO}), who had the intuition that the Heat and the Fokker-Planck equations have a common variational structure in terms of a particular distance on the probability measures, the so-called Wasserstein distance. Yet, the theory has only become formal and general with the work by Ambrosio, Gigli and Savar\'e (which does not mean that proofs in \cite{JKO} were not rigorous, but the intutition on the general structure still needed to be better understood). 

The main idea is to endow the space $\pical(\Omega)$ of probability measures on a domain $\Omega\subset\R^d$ with a distance, and then deal with gradient flows of suitable functionals on such a metric space.
Such a distance arises from optimal transport theory. More details about optimal transport can be found in the books by C. Villani (\cite{villani, villani06}) and in the book on gradient flows by Ambrosio, Gigli and Savar\'e \cite{AmbGigSav}\footnote{Lighter versions exist, such as \cite{AmbSav}, or the recent {\it User's Guide to Optimal Transport} (\cite{usersguide}), which is a good reference for many topics in this survey, as it deals for one half with optimal transport (even if the title suggests that this is the only topic of the guide), then for one sixth with the general theory of gradient flows (as in our Section \ref{metricth}), and finally for one third with metric spaces with curvature bounds (that we will briefly sketch in Section \ref{heat}).}; a recent book by the author of this survey is also available \cite{OTAM}.

\subsection{Preliminaries on Optimal Transport}

The motivation for the whole subject is the following problem proposed by Monge in 1781 (\cite{Monge}): given two densities of mass $f,\,g\geq 0$ on $\R^d$, with $\int f=\int g =1$, find a map $T:\R^d\to\R^d$ pushing the first one onto the other, i.e.
such that
\begin{equation}\label{push-forward}
\int_{A}g(x)\dd x=\int_{T^{-1}(A)}f(y)\dd y\quad \mbox{ for any Borel subset }A\subset\R^d
\end{equation}
and minimizing the quantity
$$\int_{\R^d}|T(x)-x|f(x)\,\dd x$$
among all the maps satisfying this condition. This means that we have a collection of particles, distributed with density $f$ on $\R^d$, that have to be moved, so that they arrange according to a new distribution, whose density is prescribed and is $g$. The movement has to be chosen so as to minimize the average displacement. The map $T$ describes the movement, and $T(x)$ represents the destination of the particle originally located at $x$. The constraint on $T$ precisely accounts for the fact that we need to reconstruct the density $g$. In  the sequel, we will always define, similarly to \eqref{push-forward}, the image measure of a measure $\mu$ on $X$ (measures will indeed replace the densities $f$ and $g$ in the most general formulation of the problem) through a measurable map $T:X\to Y$: it is the measure denoted by $T_\#\mu$ on $Y$ and characterized by
\begin{gather*}
(T_\#\mu)(A)=\mu(T^{-1}(A)) \quad\mbox{for every measurable set } A,\\
\mbox{ or }\int_Y\phi \,\dd\left(T_\#\mu\right)=\int_X \phi\circ T\,\,\dd \mu\quad\mbox{for every measurable function } \phi.
\end{gather*}
The problem of Monge has stayed with no solution (does a minimizer exist? how to characterize it?\dots)  till the progress made in the 1940s with the work by Kantorovich (\cite{Kantorovich}). In the Kantorovich's framework, the problem has been widely generalized, with very general cost functions $c(x,y)$ instead of the Euclidean distance $|x-y|$ and more general measures and spaces. 

Let us start from the general picture. Consider a metric space $X$, that we suppose compact for simplicity\footnote{Most of the results that we present stay true without this assumptions, anyway, and we refer in particular to \cite{AmbGigSav} or \cite{villani} for details, since a large part of the analysis of \cite{OTAM} is performed under this simplfying assumption.} and a cost function $c:X\times X\to [0,+\infty]$. For simplicity of the exposition, we will suppose that $c$ is continuous and symmetric: $c(x,y)=c(y,x)$ (in particular, the target and source space will be the same space $X$).

The formulation proposed by Kantorovich of the problem raised by Monge is the following: given two probability measures $\mu,\nu\in\pical(X)$, consider the problem
\begin{equation}\label{kantorovich}
(KP)\quad\min\left\{\int_{X\times X}\!\!c\,\,\dd \gamma\;:\;\gamma\in\Pi(\mu,\nu)\right\},
\end{equation}
where
$\Pi(\mu,\nu)$ is the set of the so-called {\it transport plans}, i.e. 
$$\Pi(\mu,\nu)=\{\gamma\in\pical(X\times X):\,(\pi_0)_{\#}\gamma=\mu,\,(\pi_1)_{\#}\gamma=\nu,\}$$
where $\pi_0$ and $\pi_1$ are the two projections of $X\times X$ onto its factors. These probability measures over $X\times X$ are an alternative way to describe the displacement of the particles of $\mu$: instead of saying, for each $x$, which is the destination $T(x)$ of the particle originally located at $x$, we say for each pair $(x,y)$ how many particles go from $x$ to $y$. It is clear that this description allows for more general movements, since from a single point $x$ particles can a priori move to different destinations $y$. If multiple destinations really occur, then this movement cannot be described through a map $T$. 
It can be easily checked that if $(id, T)_{\#}\mu$ belongs to $\Pi(\mu,\nu)$ then $T$ pushes $\mu$ onto $\nu$ (i.e. $T_\#\mu=\nu$) and the functional takes the form $\int c(x,T(x))\dd\mu(x),$ thus generalizing Monge's problem.

The minimizers for this problem are called {\it
optimal transport plans} between $\mu$ and $\nu$. Should $\gamma$ be of
the form $(id, T)_{\#}\mu$ for a measurable map
$T:X\to X$ (i.e. when no splitting of the mass occurs), the map $T$ would be called {\it optimal
transport map} from $\mu$ to $\nu$.

This generalized problem by Kantorovich is much easier to handle than the original one proposed by Monge: for instance in the Monge case we would need existence of at least a map $T$ satisfying the constraints. This is not verified when $\mu=\delta_0$, if $\nu$ is not a single Dirac mass. On the contrary, there always exists at least a transport plan in $\Pi(\mu,\nu)$ (for instance we always have $\mu\otimes\nu\in\Pi(\mu,\nu)$). Moreover, one can state that $(KP)$ is the relaxation of the original problem by Monge: if one considers the problem in the same setting, where the competitors are transport plans, but sets the functional at $+\infty$ on all the plans that are not of the form $(id, T)_{\#}\mu$, then one has a functional on $\Pi(\mu,\nu)$ whose relaxation (in the sense of the largest lower-semicontinuous functional smaller than the given one) is the functional in $(KP)$ (see for instance Section 1.5 in \cite{OTAM}).

Anyway, it is important to note that an easy use of the Direct Method of Calculus of Variations (i.e. taking a minimizing sequence, saying that it is compact in some topology - here it is the weak convergence of probability measures - finding a limit, and proving semicontinuity, or continuity, of the functional we minimize, so that the limit is a minimizer) proves that a minimum does exist. 
As a consequence, if one is interested in the problem of Monge, the question may become``does this minimizer come from a transport map $T$?'' (note, on the contrary, that directly attacking by compactness and semicontinuity Monge's formulation is out of reach, because of the non-linearity of the constraint $T_\#\mu=\nu$, which is not closed under weak convergence).

Since the problem $(KP)$ is a linear optimization under linear constraints, an important tool will be duality theory, which is typically used for convex problems. We will find a dual problem $(DP)$ for $(KP)$ and exploit the relations between dual and primal.

The first thing we will do is finding a formal dual problem, by means of an inf-sup exchange.

First express the constraint  $\gamma\in\Pi(\mu,\nu)$ in the following way : notice that, if $\gamma$ is a non-negative measure on $X\times X$, then we have
$$\sup_{\phi,\,\psi}\int\phi \,\dd \mu +\int \psi \,\dd \nu - \int \left(\phi(x)+\psi(y)\right) \,\dd \gamma = \begin{cases}0&\mbox{ if }\gamma\in\Pi(\mu,\nu)\\
+\infty&\mbox{ otherwise }\end{cases}.$$

Hence, one can remove the constraints on $\gamma$ by adding the previous sup, since if they are satisfied nothing has been added and if they are not one gets $+\infty$ and this will be avoided by the minimization.
We may look at the problem we get and interchange the inf in $\gamma$ and the sup in $\phi,\psi$:
$$
\min_{\gamma\ge 0} \int c\,\,\dd \gamma +\sup_{\phi,\psi}\left(\!\!\int\phi \,\dd \mu +\int \psi \,\dd \nu - \int (\phi(x)+\psi(y)) \,\dd \gamma\!\right)
$$ becomes $$
\sup_{\phi,\psi}\int\phi \,\dd \mu +\int \psi \,\dd \nu +\inf_{\gamma\ge 0} \int\left(c(x,y)-(\phi(x)+\psi(y))\right) \,\dd \gamma.
$$
Obviously it is not always possible to exchange inf and sup, and the main tools to do this come from convex analysis. We refer to \cite{OTAM}, Section 1.6.3 for a simple proof of this fact, or to \cite{villani}, where the proof is based on Flenchel-Rockafellar duality (see, for instance, \cite{EkeTem}). Anyway, we insist that in this case it is true that $\inf\sup=\sup\inf$.

Afterwards, one can re-write the inf in $\gamma$ as a constraint on $\phi$ and $\psi$, since one has

$$\inf_{\gamma\geq 0} \int\left(c(x,y)-(\phi(x)+\psi(y))\right) \,\dd \gamma
= \begin{cases}0&\mbox{ if }\phi(x)+\psi(y)\leq c(x,y) \mbox{ for all }(x,y)\in X\times X\\
-\infty&\mbox{ otherwise }\end{cases}.$$

This leads to the following dual optimization problem: given the two probabilities $\mu$ and $\nu$ and the cost function $c:X\times X\to[0,+\infty]$ we consider the problem
\begin{equation}\label{dual}
(DP)\quad\max\left\{\int_{X}\!\phi \,\dd\mu+\!\!\int_{X}\!\psi \,\dd\nu\;:\;\phi,\psi\in C(X),\,\phi(x)\!+\!\psi(y)\leq c(x,y) \mbox{ for all }(x,y)\in X\times X\right\}.
\end{equation}

This problem does not admit a straightforward existence result, since the class of admissible functions lacks compactness. Yet, we can better understand this problem and find existence once we have introduced the notion of $c-$transform (a kind of generalization of the well-known Legendre transform).

\begin{definition}
Given a function $\chi:X\to\overline{\R}$ we define its {\it $c-$transform} (or $c-$conjugate function) by
$$\chi^c(y)=\inf_{x\in X}c(x,y)-\chi(x).$$
Moreover, we say that a function $\psi$ is {\it $c-$concave} if there exists $\chi$ such that $\psi=\chi^c$ and we denote by $\Psi_c(X)$ the set of $c-$concave functions.
\end{definition}

It is quite easy to realize that, given a pair $(\phi,\psi)$ in the maximization problem (DP), one can always replace it with $(\phi,\phi^c)$, and then with $(\phi^{cc},\phi^c)$, and the constraints are preserved and the integrals increased. Actually one could go on but it is possible to prove that $\phi^{ccc}=\phi^c$ for any function $\phi$. This is the same as saying that $\psi^{cc}=\psi$ for any $c-$concave function $\psi$, and this perfectly recalls what happens for the Legendre transform of convex funtions (which corresponds to the particular case $c(x,y)=-x\cdot y$). 

A consequence of these considerations is the following well-known result:
\begin{proposition}\label{prop-duality formula}
We have
\begin{equation}\label{duality formula}
\min\,(KP)=\max_{\phi\in\Psi_c(X)}\;\int_{X}\phi\,\,\dd \mu+\int_{X}\phi^c\,\,\dd \nu,
\end{equation}
where the max on the right hand side is realized. In particular the minimum value of $(KP)$ is a convex function of $(\mu,\nu)$, as it is a supremum of linear functionals.
\end{proposition}

\begin{definition}
The functions $\phi$ realizing the maximum in \eqref{duality formula} are called {\it Kantorovich potentials} for the transport from $\mu$ to $\nu$ (and will be often denoted by the symbol $\varphi$ instead of $\phi$).
\end{definition}

Notice that any $c-$concave function shares the same modulus of continuity of the cost $c$. Hence, if $c$ is uniformly continuous (which is always the case whenever $c$ is continuous and $X$ is compact), one can get a uniform modulus of continuity for the functions in $\Psi_c(X)$. This is the reason why one can prove existence for $(DP)$ (which is the same of the right hand side problem in Proposition \ref{prop-duality formula}), by applying Ascoli-Arzel\`a's Theorem.

We look at two interesting cases. When $c(x,y)$ is equal to the distance $d(x,y)$ on the metric space $X$, then we can easily see that 
\begin{equation}\label{|x-y|1}
\phi\in\Psi_{c}(X) \Longleftrightarrow \phi\mbox{  is a $1$-Lipschitz function}
\end{equation}
and 
\begin{equation}\label{|x-y|2}
\phi\in\Psi_{c}(X)\impl \phi^c=-\phi.
\end{equation}
Another interesting case is the case where $X=\Omega\subset\R^d$ and $c(x,y)=\frac 12|x-y|^2$. In this case we have
$$\phi\in\Psi_{c}(X) \Longrightarrow \;x\mapsto \frac{x^2}{2}-\phi(x)\text{  is a convex function.}
$$
Moreover, if $X=\R^d$ this is an equivalence.

A consequence of \eqref{|x-y|1} and \eqref{|x-y|2} is that, in the case where $c=d$, Formula \ref{duality formula} may be re-written as
\begin{equation}\label{|x-y|3}
\min\,(KP)=\max\,(DP)=\max_{\phi\in\Lip_1}\int_{X}\phi\,\dd(\mu-\nu).
\end{equation}

We now concentrate on the quadratic case when $X$ is a domain $\Omega\subset \R^d$, and look at the existence of optimal transport maps $T$. From now on, we will use the word {\it domain} to denote the closure of a bounded and connected open set.

The main tool is the duality result. If we have equality between the minimum of $(KP)$ and the maximum of $(DP)$ and both extremal values are realized, one can consider an optimal transport plan $\gamma$ and a Kantorovich potential $\varphi$ and write
$$\varphi(x)+\varphi^c(y)\leq c(x,y) \mbox{ on }X\times X \;\mbox{ and }\;\varphi(x)+\varphi^c(y)= c(x,y) \mbox{ on }\spt\gamma.$$
The equality on $\spt\gamma$ is a consequence of the inequality which is valid everywhere and of 
$$\int c\,\,\dd \gamma=\int\varphi \,\,\dd \mu+\int\varphi^c \,\,\dd \nu=\int(\varphi(x)+\varphi^c(y))\,\dd \gamma,$$
which implies equality $\gamma-$a.e. These functions being continuous, the equality passes to the support of the measure.
Once we have that, let us fix a point $(x_0,y_0)\in\spt\gamma$. One may deduce from the previous computations that
$$x\mapsto \varphi(x)-\frac12|x-y_0|^2\quad\mbox{ is maximal at }x=x_0$$
and, if $\varphi$ is differentiable at $x_0$, one gets $\nabla\varphi(x_0)=x_0-y_0,$ i.e. $y_0=x_0-\nabla\varphi(x_0)$. This shows that only one unique point $y_0$ can be such that $(x_0,y_0)\in \spt\gamma$, which means that $\gamma$ is concentrated on a graph. The map $T$ providing this graph is of the form $x\mapsto x-\nabla\varphi(x)=\nabla u(x)$ (where $u(x):= \frac{x^2}{2}-\varphi(x)$ is a convex function). This shows the following well-known theorem, due to Brenier (\cite{Brenier polar,Brenier91}). Note that this also gives uniqueness of the optimal transport plan and of the gradient of the Kantorovich potential. The only technical point to let this strategy work is the $\mu$-a.e. differentiability of the potential $\varphi$. Since $\varphi$ has the same regularity of a convex function, and convex function are locally Lipschitz, it is differentiable Lebesgue-a.e., which allows to prove the following:

\begin{theorem}\label{maintransport}
Given $\mu$ and $\nu$ probability measures on a domain $\Omega\subset\R^d$ there exists an optimal transport plan $\gamma$ for the quadratic cost $\frac 12|x-y|^2$. It is unique and of the form $(id, T)_{\#}\mu$, provided $\mu$ is absolutely continuous. Moreover there exists also at least a Kantorovich potential $\varphi$, and the gradient $\nabla\varphi$ is uniquely determined $\mu-$a.e. (in particular $\varphi$ is unique up to additive constants if the density of $\mu$ is positive a.e. on $\Omega$). The optimal transport map $T$ and the potential $\varphi$ are linked by $T(x)=x-\nabla\varphi(x)$. Moreover, the optimal map $T$ is equal a.e. to the gradient of a convex function $u$, given by $u(x):= \frac{x^2}{2}-\varphi(x)$.\end{theorem}

Actually, the existence of an optimal transport map is true under weaker assumptions: we can replace the condition of being absolutely continuous with the condition ``$\mu(A)=0$ for any $A\subset\R^d$ such that $\haus^{d-1}(A)<+\infty$'' or with any condition which ensures that the non-differentiability set of $\varphi$ is negligible (and convex function are more regular than locally Lipschitz functions).

In Theorem \ref{maintransport} only the part concerning the optimal map $T$ is not symmetric in $\mu$ and $\nu$: hence the uniqueness of the Kantorovich potential is true even if it $\nu$ (and not $\mu$) has positive density a.e. (since one can retrieve $\varphi$ from $\varphi^c$ and viceversa).

We stress that Theorem \ref{maintransport} admits a converse implication and that any gradient of a convex function is indeed optimal between $\mu$ and its image measure. Moreover, Theorem \ref{maintransport} can be translated in very easy terms in the one-dimensional case $d=1$: given a non-atomic measure $\mu$ and another measure $\nu$, both in $\pical(\R)$, there exists a unique monotone increasing transport map $T$ such that $T_\#\mu=\nu$, and it is optimal for the quadratic cost.

Finally, the same kind of arguments could be adapted to prove existence and uniqueness of an optimal map for other costs, in particular to costs of the form $c(x,y)=h(x-y)$ for a stricty convex function $h:\R^d\to \R$, which includes all the costs of the form $c(x,y)=|x-y|^p$ with $p>1$. In the one-dimensional case it even happens that the same monotone increasing map is optimal for every $p\geq 1$ (and it is the unique optimal map for $p>1$)!

\subsection{The Wasserstein distances}

Starting from the values of the problem $(KP)$ in \eqref{kantorovich} we can define a set of distances over $\pical(X)$. 

We mainly consider costs of the form $c(x,y)=|x-y|^p$ in $X=\Omega\subset\R^d$, but the analysis can be adapted to a power of the distance in a more general metric space $X$. The exponent $p$ will always be taken in $[1,+\infty[$ (we will not discuss the case $p=\infty$) in order to take advantage of the properties of the $L^p$ norms.  When $\Omega$ is unbounded we need to restrict our analysis to the following set of probabilities
$$\pp:=\left\{\mu\in\pical(\Omega)\,:\,\int_\Omega |x|^p\,\dd\mu(x)<+\infty\right\}.$$
In a metric space, we fix an arbitrary point $x_0\in X$, and set 
$$\pical_p(X):=\left\{\mu\in\pical(X)\,:\,\int_X d(x,x_0)^p\,\dd\mu(x)<+\infty\right\}$$ 
(the finiteness of this integral does not depend on the choice of $x_0$).

The distances that we want to consider are defined in the following way: for any $p\in [ 1,+\infty[$ set
$$W_p(\mu,\nu)=\big(\min\,(KP)\text{ with }c(x,y)=|x-y|^p\big)^{1/p}.$$

The quantities that we obtain in this way are called Wasserstein distances\footnote{They are named after L. Vaserstein (whose name is sometimes spelled Wasserstein), but this choice is highly debated, in particular in Russia, as the role he played in these distances is very marginal. However, this is now the standard name used in Western countries, probably due to the terminology used in \cite{JKO,Otto porous}, even if other names have been suggested, such as Monge-Kantorovich distances, Kantorovich-Rubinstein\dots\ and we will stick to this terminology.}. They are very important in many fields of applications and they seem a natural way to describe distances between equal amounts of mass distributed on a same space. 

It is interesting to compare these distances to $L^p$ distances between densities (a comparison which is meaningful when we consider absolutely continous measures on $\R^d$, for instance). A first observation is the very different behavior of these two classes of distances. We could say that, if $L^p$ distances can be considered ``vertical'', Wasserstein distances are instead ``horizontal''. This consideration is very informal, but is quite natural if one associates with every absolutely continuous measure the graph of its density. To compute $||f-g||_{L^p}$ we need to look, for every point $x$, at the distance between $f(x)$ and $g(x)$, which corresponds to a vertical displacement between the two graphs, and then integrate this quantity. On the contrary, to compute $W_p(f,g)$ we need to consider the distance between a point $x$ and a point $T(x)$ (i.e. an horizontal displacement on the graph) and then to integrate this, for a particular pairing between $x$ and $T(x)$ which makes a coupling between $f$ and $g$.

 A first example where we can see the very different behavior of these two ways of computing distances is the following: take two densities $f$ and $g$ supported on $[0,1]$, and define $g_h$ as $g_h(x)=g(x-h)$. As soon as $|h|>1$, the $L^p$ distance between $f$ and $g_h$ equals $(||f||_{L^p}^p+||g||_{L^p}^p)^{1/p}$, and does not depend on the ``spatial'' information consisting in $|h|$. On the contrary, the $W_p$ distance  between $f$ and $g_h$ is of the order of $|h|$ (for $h\to\infty$) and depends much more on the displacement than on the shapes of $f$ and $g$.

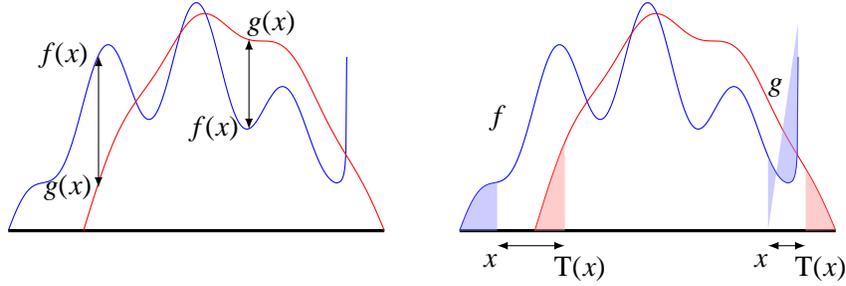
\begin{figure}\label{LpWp}
\begin{center}
\begin{tikzpicture}
\begin{scope}
\draw  [very thick] (0,0)--(5,0);
\draw [blue,domaine={0}{4.5},samples=400] plot (\x,{(\x)*(4.5-\x)*(1+cos(5*\x r))*0.15+(\x)*(4.5-\x)*0.27+0.06*(\x)*(4.5-\x)^0.3});
\draw [red,domaine={1}{5},samples=150] plot (\x,{(\x-1)*(5-\x)*(1.8+cos(5*\x r))*0.04+(\x-1)*(5-\x)*0.5+0.06*(\x-1)*(5-\x)^2});
\draw [<->,>=latex](1.2,0.57) --(1.2,2.33)  node[left]{$f(x)$};
\draw [<->,>=latex](3.2,1.35) node[left]{$f(x)$}--(3.2,2.55)  ;
\draw (3.5,2.75)  node{$g(x)$};
\draw (1.25,0.57) node[left]{$g(x)$};
\end{scope}
\begin{scope}[shift={(6,0)}]

\draw [very thick] (0,0)--(5,0);
\draw [blue,domaine={0}{4.5},samples=400] plot (\x,{(\x)*(4.5-\x)*(1+cos(5*\x r))*0.15+(\x)*(4.5-\x)*0.27+0.06*(\x)*(4.5-\x)^0.3});
\draw [red,domaine={1}{5},samples=150] plot (\x,{(\x-1)*(5-\x)*(1.8+cos(5*\x r))*0.04+(\x-1)*(5-\x)*0.5+0.06*(\x-1)*(5-\x)^2});
\fill[blue!40, opacity=0.5]  (0.5,0)--plot[domaine={0}{0.5},samples=300]  (\x,{(\x)*(4.5-\x)*(1+cos(5*\x r))*0.15+(\x)*(4.5-\x)*0.27+0.06*(\x)*(4.5-\x)^0.3})--cycle;
\fill[red!40, opacity=0.5]  (1.4,0)-- plot [red,domaine={1}{1.4},samples=150] (\x,{(\x-1)*(5-\x)*(1.8+cos(5*\x r))*0.04+(\x-1)*(5-\x)*0.5+0.06*(\x-1)*(5-\x)^2})--cycle;
\fill[blue!40, opacity=0.5]  (4.1,0)--plot[domaine={4.1}{4.5},samples=300]  (\x,{(\x)*(4.5-\x)*(1+cos(5*\x r))*0.15+(\x)*(4.5-\x)*0.27+0.06*(\x)*(4.5-\x)^0.3})--cycle;
\fill[red!40, opacity=0.5]  (4.6,0)-- plot [red,domaine={4.6}{5},samples=150] (\x,{(\x-1)*(5-\x)*(1.8+cos(5*\x r))*0.04+(\x-1)*(5-\x)*0.5+0.06*(\x-1)*(5-\x)^2})--cycle;
\draw [<->,>=latex] (0.5,-0.2)--(1.4,-0.2);
\draw (0.4,-0.4) node{$x$};
\draw (1.6,-0.5) node{$\T(x)$};

\draw [<->,>=latex] (4.1,-0.2)--(4.6,-0.2);
\draw (4,-0.4) node{$x$};
\draw (4.8,-0.5) node{$\T(x)$};

\draw (0.5,1.5) node{$f$};
\draw (4.2,1.9) node{$g$};

\end{scope}
\end{tikzpicture}
\caption{``Vertical'' vs ``horizontal'' distances (the transport $\T$ is computed in the picture on the right using 1D considerations, imposing equality between the blue and red areas under the graphs of $f$ and $g$).}
\end{center}
\end{figure}

We now analyze some properties of these distances. Most proofs can be found in \cite{OTAM}, chapter 5, or in \cite{villani} or \cite{AmbGigSav}.

First we underline that, as a consequence of H\"older (or Jensen) inequalities, the Wasserstein distances are always ordered, i.e. $W_{p_1}\leq W_{p_2}$ if $p_1\leq p_2$. Reversed inequalities are possible only if $\Omega$ is bounded, and in this case we have, if set $D=\diam(\Omega)$, for $p_1\leq p_2$,
$$W_{p_1}\leq W_{p_2}\leq D^{1-p_1/p_2}W_{p_1}^{p_1/p_2}.$$

This automatically guarantees that, if the quantities $W_p$ are distances, then they induce the same topology, at least when $\Omega$ is bounded. But first we should check that they are distances\dots\ 

\begin{proposition} The quantity $W_p$ defined above is a distance over $\pp$.\end{proposition}
\begin{proof}
First, let us note that $W_p\geq 0$. Then, we also remark that $W_p(\mu,\nu)=0$ implies that there exists $\gamma\in\Pi(\mu,\nu)$ such that $\int|x-y|^p\,\dd\gamma=0$. Such a $\gamma\in\Pi(\mu,\nu)$ is concentrated on $\{x=y\}$. This implies $\mu=\nu$ since, for any test function $\phi$, we have
$$\int \phi\,\,\dd\mu=\int \phi(x)\,\dd\gamma=\int \phi(y)\,\dd\gamma=\int \phi\,\,\dd\nu.$$

We need now to prove the triangle inequality. We only give a proof in the case $p>1$, with absolutely continuous measures. 

Take three measures $\mu,\varrho$ and $\nu$, and suppose that $\mu$ and $\varrho$ are absolutely continuous. Let $T$ be the optimal transport from $\mu$ to $\varrho$ and $S$ the optimal one from $\varrho$ to $\nu$. Then $S\circ T$ is an admissible transport from $\mu$ to $\nu$, since $(S\circ T)_\#\mu=S_\#(T_\#\mu)=S_\#\varrho=\nu$. We have
\begin{eqnarray*}
W_p(\mu,\nu)\leq \left(\int |S(T(x))-x|^p\,\dd\mu(x)\right)^{\frac 1p}&=&||S\circ T-\id||_{L^p(\mu)}\\
&\leq& ||S\circ T-T||_{L^p(\mu)}+||T-\id||_{L^p(\mu)}.
\end{eqnarray*}
Moreover,
$$||S\circ T-T||_{L^p(\mu)}=\left(\int |S(T(x))-T(x)|^p\,\dd\mu(x)\right)^{\frac 1p}=\left(\int |S(y)-y|^p\dd\varrho(y)\right)^{\frac 1p}$$
and this last quantity equals $W_p(\varrho,\nu)$. Moreover, $||T-\id||_{L^p(\mu)}=W_p(\mu,\varrho)$, whence
$$W_p(\mu,\nu)\leq W_p(\mu,\varrho)+W_p(\varrho,\nu).$$
This gives the proof when $\mu,\varrho\ll\lcal^d$ and $p>1$. For the general case, an approximation is needed (but other arguments can also apply, see, for instance, \cite{OTAM}, Section 5.1).
\end{proof}

We now give, without proofs, two results on the topology induced by $W_p$ on a general metric space $X$.
\begin{theorem}
If $X$ is compact, for any $p\geq 1$ the function $W_p$ is a distance over $\pical(X)$ and the convergence with respect to this distance is equivalent to the weak convergence of probability measures.\end{theorem}

To prove that the convergence according to $W_p$ is equivalent to weak convergence one first establish this result for $p=1$, through the use of the duality formula in the form \eqref{|x-y|3}. Then it is possible to use the inequalities between the distances $W_p$ (see above) to extend the result to a general $p$.

The case of a noncompact space $X$ is a little more difficult. As we said, the distance is only defined on a subset of the whole space of probability measures, to avoid infinite values. Set, for a fixed reference point $x_0$ which can be chosen to be $0$ in the Euclidean space,
$$m_p(\mu):=\int_{X}d(x,x_0)^p\dd\mu(x).$$
In this case, the distance $W_p$ is only defined on $\pical_p(X):=\left\{\mu\in\pical(X)\,:\,m_p(\mu)<+\infty\right\}$. We have
\begin{theorem}\label{conWp}
For any $p\geq 1$ the function $W_p$ is a distance over $\mathcal{P}_p(X)$ and, given a measure $\mu$ and a sequence $(\mu_n)_n$ in $\mathbb W_p(X)$, the following are equivalent:
\begin{itemize}
\item $\mu_n\to\mu$ according to $W_p$;
\item $\mu_n\deb\mu$ and $m_p(\mu_n)\to m_p(\mu)$;
\item $\int_{X}\phi\,\,\dd \mu_n\to\int_{X}\phi\,\,\dd \mu$ for any $\phi\in C^0(X)$ whose growth is at most of order $p$ (i.e. there exist constants $A$ and $B$ depending on $\phi$ such that $|\phi(x)|\leq A+Bd(x,x_0)^p$ for any $x$).
\end{itemize}
\end{theorem}

After this short introduction to the metric space $\mathbb W_p:=(\pical_p(X),W_p)$ and its topology, we will focus on the Euclidean case, i.e. where the metric space $X$ is a domain $\Omega\subset\R^d$, and study the curves valued in $\mathbb W_p(\Omega)$ in connections with PDEs.

The main point is to identify the absolutely continuous curves in the space $\WW_p(\Omega)$ with solutions of the continuity equation $\partial_t\mu_t+\nabla\cdot(\vv_t\mu_t)=0$ with $L^p$ vector fields $\vv_t$. Moreover, we want to connect the $L^p$ norm of $\vv_t$ with the metric derivative $|\mu'|(t)$.

We recall that standard considerations from fluid mechanics tell us that the continuity equation above may be interpreted as the equation ruling the evolution of the density $\mu_t$ of a family of particles initially distributed according to $\mu_0$ and each of which follows the flow 
$$\begin{cases}y_x'(t)=\vv_t(y_x(t))\\y_x(0)=x.\end{cases}$$
The main theorem in this setting (originally proven in \cite{AmbGigSav}) relates absolutely continuous curves in $\WW_p$ with solutions of the continuity equation:
\begin{theorem}\label{curves and PDE}
Let $(\mu_t)_{t\in[0,1]}$ be an absolutely continuous curve in $\WW_p(\Omega)$ (for $p>1$ and $\Omega\subset\R^d$ an open domain). Then for a.e. $t\in [0,1]$ there exists a vector field $\vv_t\in L^p(\mu_t;\R^d)$ such that
\begin{itemize}
\item the continuity equation $\partial_t\mu_t+\nabla\cdot(\vv_t\mu_t)=0$ is satisfied in the sense of distributions,
\item for a.e. $t$ we have $||\vv_t||_{L^p(\mu_t)}\leq |\mu'|(t)$ (where $|\mu'|(t)$ denotes the metric derivative at time $t$ of the curve $t\mapsto\mu_t$, w.r.t. the distance $W_p$);
\end{itemize}

Conversely, if $(\mu_t)_{t\in[0,1]}$ is a family of measures in $\pp$ and for each $t$ we have a vector field $\vv_t\in L^p(\mu_t;\R^d)$ with $\int_0^1||\vv_t||_{L^p(\mu_t)}\,\dd t<+\infty$ solving  $\partial_t\mu_t+\nabla\cdot(\vv_t\mu_t)=0$, then $(\mu_t)_t$ is absolutely continuous in $\WW_p(\Omega)$ and for a.e. $t$ we have $ |\mu'|(t)\leq ||\vv_t||_{L^p(\mu_t)}$. \end{theorem}

Note that, as a consequence of the second part of the statement, the vector field $\vv_t$ introduced in the first part must a posteriori satisfy $||\vv_t||_{L^p(\mu_t)}= |\mu'|(t)$.

We will not give the proof of this theorem, which is quite involved. The main reference is \cite{AmbGigSav} (but the reader can also find alternative proofs in Chapter 5 of \cite{OTAM}, in the case where $\Omega$ is compact). 
Yet, if the reader wants an idea of the reason for this theorem to be true, it is possible to start from the case of two time steps: there are two measures $\mu_t$ and $\mu_{t+h}$ and there are several ways for moving the particles so as to reconstruct the latter from the former. It is exactly as when we look for a transport. One of these transports is optimal in the sense that it minimizes $\int |T(x)-x|^p\dd\mu_t(x)$ and the value of this integral equals $W_p^p(\mu_t,\mu_{t+h})$. If we call $\vv_t(x)$ the ``discrete velocity of the particle located at $x$ at time $t$, i.e. $\vv_t(x)=(T(x)-x)/h$, one has $||\vv_t||_{L^p(\mu_t)}=\frac 1 h W_p(\mu_t,\mu_{t+h})$.  We can easily guess that, at least formally, the result of the previous theorem can be obtained as a limit as $h\to 0$.

Once we know about curves in their generality, it is interesting to think about geodesics. The following result is a characterization of geodesics in $W_p(\Omega)$ when $\Omega$ is a convex domain in $\R^d$. This procedure is also known as {\it McCann's displacement interpolation}.

\begin{theorem}\label{geodesics in Wp}
If $\Omega\subset\R^d$ is convex, then all the spaces $\WW_p(\Omega)$ are length spaces and if $\mu$ and $\nu$ belong to $\WW_p(\Omega)$, and $\gamma$ is an optimal transport plan from $\mu$ to $\nu$ for the cost $c_p(x,y)=|x-y|^p$, then the curve
$$\mu^{\gamma}(t)=(\pi_t)_{\#}\gamma$$
where $\pi_t:\Omega\times \Omega\to\Omega$ is given by $\pi_t(x,y)=(1-t)x+ty$, is a constant-speed geodesic  from $\mu$ to $\nu$. In the case $p>1$ all the constant-speed geodesics are of this form, and, if $\mu$ is absolutely continuous, then there is only one geodesic and it has the form
$$
\mu_t= [T_t]_{\#} \mu,\qquad\mbox{where }T_t:=(1 - t)id + t T
$$
where $T$ is the optimal transport map from $\mu$ to $\nu$. In this case, the velocity field $\vv_t$ of the geodesic $\mu_t$ is given by $\vv_t=(T-id)\circ(T_t)^{-1}$. In particular, for $t=0$ we have $\vv_0=-\nabla\varphi$ and for $t=1$ we have $\vv_1=\nabla\psi$, where $\varphi$ is the Kantorovich potential in the transport from $\mu$ to $\nu$ and $\psi=\varphi^c$.
\end{theorem}

The above theorem may be adapted to the case where the Euclidean domain $\Omega$ is replaced by a Riemannian manifold, and in this case the map $T_t$ must be defined by using geodesics instead of segments: the point $T_t(x)$ will be the value at time $t$ of a constant-speed geodesic, parametrized on the interval $[0,1]$, connecting $x$ to $T(x)$. For the theory of optimal transport on manifolds, we refer to \cite{McC manif}.

Using the characterization of constant-speed geodesics as minimizers of a strictly convex kinetic energy, we can also infer the following interesting information.

\begin{itemize}
\item Looking for an optimal transport for the cost $c(x,y)=|x-y|^p$ is equivalent to looking for constant-speed geodesic in $\WW_p$ because from optimal plans we can reconstruct geodesics and from geodesics (via their velocity field) it is possible to reconstruct the optimal transport;
\item constant-speed geodesics may be found by minimizing $\int_0^1|\mu'|(t)^p\dd t$ ;\index{geodesics!in $\WW_p$}
\item in the case of the Wasserstein spaces, we have $|\mu'|(t)^p=\int_\Omega |\vv_t|^p\,\dd\mu_t$, where $\vv$ is a velocity field solving the continuity equation together with $\mu$ (this field is not unique, but the metric derivative $|\mu'|(t)$ equals the minimal value of the $L^p$ norm of all possible fields).
\end{itemize}

As a consequence of these last considerations, for $p>1$, solving the kinetic energy minimization problem\index{kinetic energy}
$$\min\left\{\int_0^1\int_\Omega |\vv_t|^p\dd\varrho_t\,\dd t\;:\;\partial_t\varrho_t+\nabla\cdot(\vv_t\varrho_t)=0,\;\varrho_0=\mu,\,\varrho_1=\nu\right\}$$
selects constant-speed geodesics connecting $\mu$ to $\nu$, and hence allows to find the optimal transport between these two measures.
This is what is usually called {\it Benamou-Brenier formula} (\cite{BenBre}).

On the other hand, this minimization problem in the variables $(\varrho_t,\vv_t)$ has non-linear constraints (due to the product $\vv_t\varrho_t$) and the functional is non-convex (since $(t,x)\mapsto t|x|^p$ is not convex). Yet, it is possible to transform it into a convex problem.
For this, it is sufficient to switch variables, from $(\varrho_t,\vv_t)$ into $(\varrho_t,E_t)$ where $E_t=\vv_t\varrho_t$, thus obtaining the following minimization problem
\begin{equation}\label{BBp}
\min\left\{\int_0^1\int_\Omega \frac{|E_t|^p}{\varrho_t^{p-1}}\,\dd x\dd t\;:\;\partial_t\varrho_t+\nabla\cdot E_t=0,\;\varrho_0=\mu,\,\varrho_1=\nu\right\}.
\end{equation}
We need to use the properties of the function $f_p:\R\times\R^d\to\R\cup\{+\infty\}$, defined through 
$$
f_p(t,x):=\sup_{(a,b)\in K_q}\,(at+b\cdot x)=
\begin{cases}\frac 1p\frac{|x|^p}{t^{p-1}}&\mbox{ if }t>0,\\
			0&\mbox{ if }t=0,x=0\\
			+\infty&\mbox{ if }t=0,x\neq 0,\;\mbox{ or }t<0,\end{cases}
			$$
where $K_q:=\{(a,b)\in\R\times\R^d\,:\,a+\frac 1q |b|^q\leq 0\}$ and $q=p/(p-1)$ is the conjugate exponent of $p$. In particular, $f_p$ is convex, which makes the above minimization problem convex, and also allows to write what we formally wrote as $\int_0^1\int_\Omega \frac{|E_t|^p}{\varrho_t^{p-1}}\,\dd x\dd t$ (an expression which implicitly assumes $\varrho_t,E_t\ll\lcal^d$) into the form 

$$\mathcal B_p(\varrho,E):=\sup\left\{\int a\,\dd\varrho+\int b\cdot\dd E\;:\;(a,b)\in C(\Omega\times[0,1];K_q)\right\}.$$
Both the convexity and this last expression will be useful for numerical methods (as it was first done in \cite{BenBre}).

\subsection{Minimizing movement schemes in the Wasserstein space and evolution PDEs}

Thanks to all the theory which has been described so far, it is natural to study gradient flows in the space $\mathbb W_2(\Omega)$ (the reason for choosing the exponent $p=2$ will be clear in a while) and to connect them to PDEs of the form of a continuity equation. The most convenient way to study this is to start from the time-discretized problem, i.e. to consider a sequence of iterated minimization problems:
\begin{equation}\label{FW2}
\varrho^\tau_{k+1}\in\argmin_\varrho F(\varrho)+\frac{W_2^2(\varrho,\varrho^{\tau}_k)}{2\tau}.
\end{equation}
Note that we denote now the measures on $\Omega$ by the letter $\varrho$ instead of $\mu$ or $\nu$ because we expect them to be absolutely continuous measures with nice (smooth) densities, and we want to study the PDE they solve. The reason to focus on the case $p=2$ can also be easily understood. Indeed, from the very beginning, i.e. from Section \ref{2}, we saw that the equation $x'=-\nabla F(x)$ corresponds to a sequence of minimization problems with the squared distance $|x-x^\tau_k|^2$ (if we change the exponent here we can consider
$$\min_x\; F(x)+\frac 1p \cdot\frac{|x-x^\tau_k|^p}{\tau^{p-1}},$$
but this gives raise to the equation $x'=-|\nabla F(x)|^{q-2}\nabla F(x),$ where $q=p/(p-1)$ is the conjugate exponent of $p$), and in the Wasserstein space $\mathbb W_p$ the distance is defined as the power $1/p$ of a transport cost; only in the case $p=2$ the exponent goes away and we are lead to a minimization problem involving $F(\varrho)$ and a transport cost of the form
$$\mathcal T_c(\varrho,\nu):=\min\left\{\int c(x,y)\,\dd\gamma\,:\,\gamma\in\Pi(\varrho,\nu)\right\},$$
for $\nu=\varrho^\tau_k$.

In the particular case of the space $\mathbb W_2(\Omega)$, which has some additional structure, if compared to arbitrary metric spaces, we would like to give a PDE description of the curves that we obtain as gradient flows, and this will pass through the optimality conditions of the minimization problem \eqref{FW2}. In order to study these optimality conditions, we introduce the notion of first variation of a functional. This will be done in a very sketchy and formal way (we refer to Chapter 7 in \cite{OTAM} for more details).

Given a functional $G:\pical(\Omega)\to\R$ we call $\frac{\delta G}{\delta\varrho}(\varrho)$, if it exists, the unique (up to additive constants) function such that $\frac{d}{d\ve}G(\varrho+\ve\chi)_{|\ve=0}=\int \frac{\delta G}{\delta\varrho}(\varrho) d\chi$ for every perturbation $\chi$ such that, at least for $\ve\in[0,\ve_0],$ the measure $\varrho+\ve\chi$ belongs to $\pical(\Omega)$. The function $\frac{\delta G}{\delta\varrho}(\varrho)$ is called {\it first variation} of the functional $G$ at $\varrho$. In order to understand this notion, the easiest possibility is to analyze some examples.

The three main classes of examples are the following functionals\footnote{Note that in some cases the functionals that we use are actually valued in $\R\cup\{+\infty\}$, and we restrict to a suitable class of perturbations $\chi$ which make the corresponding functional finite.}
$$\mathcal F(\varrho)=\int f(\varrho(x))\dd x, \quad \mathcal V(\varrho)=\int V(x)\dd\varrho,\quad \mathcal W(\varrho)=\frac12\int\!\! \int \! W(x-y)\dd\varrho(x)\dd\varrho(y),$$
where $f:\R\to\R$ is a convex superlinear function (and the functional $F$ is set to $+\infty$ if $\varrho$ is not absolutely continuous w.r.t. the Lebesgue measure) and $V:\Omega\to\R$ and $W:\R^d\to\R$ are regular enough (and $W$ is taken symmetric, i.e. $W(z)=W(-z)$, for simplicity). In this case it is quite easy to realize that we have 
$$\frac{\delta \mathcal F}{\delta \varrho}(\varrho)=f'(\varrho), \quad \frac{\delta \mathcal V}{\delta \varrho}(\varrho)=V,\quad \frac{\delta \mathcal W}{\delta \varrho}(\varrho)= W\!*\!\varrho.$$ 

It is clear that the first variation of a functional is a crucial tool to write optimality conditions for variational problems involving such a functional. In order to study the problem \eqref{FW2}, we need to complete the picture by undestanding the first variation of functionals of the form $\varrho\mapsto \mathcal T_c(\varrho,\nu)$. The result is the following:

\begin{proposition}\label{conv_subdiff_first var}\index{Kantorovich potentials} 
Let $c:\Omega\times\Omega\to\R$ be a continuous cost function. Then the functional $\varrho\mapsto\Wc(\varrho,\nu)$ is convex, and its subdifferential at $\varrho_0$ coincides with the set of Kantorovich potentials $\{\varphi\in C^0(\Omega)\,:\,\int\varphi\,\,\dd\varrho_0+\int\varphi^c\,\,\dd\nu=\Wc(\varrho,\nu)\}$. Moreover, if there is a unique $c$-concave Kantorovich potential $\varphi$ from $\varrho_0$ to $\nu$ up to additive constants, then we also have $\frac{\delta \Wc(\cdot,\nu)}{\delta\varrho}(\varrho_0)=\varphi$. \end{proposition}

Even if a complete proof of the above proposition is not totally trivial (and Chapter 7 in \cite{OTAM} only provides it in the case where $\Omega$ is compact), one can guess why this is true from the following considerations. Start from Propositon \ref{prop-duality formula}, which provides
$$\Wc(\varrho,\nu)=\max_{\phi\in\Psi_c(X)}\;\int_{\Omega}\phi\,\dd\varrho+\int_{\Omega}\phi^c\,\,\dd \nu.$$
This expresses $\Wc$ as a supremum of linear functionals in $\varrho$ and shows convexity. Standard considerations from convex analysis allow to identify the subdifferential as the set of functions $\varphi$ attaining the maximum. An alternative point of view is to consider the functional $\varrho\mapsto \int\phi\,\dd\varrho+\int\phi^c\,\,\dd \nu$ for fixed $\phi$, in which case the first variation is of course $\phi$; then it is easy to see that the first variation of the supremum may be obtained (in case of uniqueness) just by selecting the optimal $\phi$.

Once we know how to compute first variations, we come back to the optimality conditions for the minimization problem \eqref{FW2}. Which are these optimality conditions? roughly speaking, we should have 
$$\frac{\delta F}{\delta \varrho}(\varrho^\tau_{k+1})+\frac{\varphi}{\tau}=const$$ 
(where the reasons for having a constant instead of $0$ is the fact that, in the space of probability measures, only zero-mean densities are considered as admissible perturbations, and the first variations are always defined up to additive constants). Note that here $\varphi$ is the Kantorovich potential associated with the transport from $\varrho_{k+1}^\tau$ to $\varrho_k^\tau$ (and not viceversa).

More precise statements and proofs of this optimality conditions will be presented in the next section. Here we look at the consequences we can get. Actually, if we combine the fact that the above sum is constant, and that we have $T(x)=x-\nabla\varphi(x)$ for the optimal $T$, we get
\begin{equation}\label{eq grad}
-\vv(x):=\frac{T(x)-x}{\tau}=-\frac{\nabla\varphi(x)}{\tau}=\nabla\big(\frac{\delta F}{\delta \varrho}(\varrho)\big)(x).
\end{equation}
We will denote by $-\vv$ the ratio $\frac{T(x)-x}{\tau}$. Why? because, as a ratio between a displacement and a time step, it has the meaning of a velocity, but since it is the displacement associated to the transport from $\varrho^\tau_{k+1}$ to $\varrho^\tau_{k}$, it is better to view it rather as a backward velocity (which justifies the minus sign).

Since here we have $\vv=-\nabla\big(\frac{\delta F}{\delta \varrho}(\varrho))$, this suggests that at the limit $\tau\to 0$ we will find a solution of 
\begin{equation}\label{PDEgradflow}
\partial_t\varrho-\nabla\cdot\left(\varrho\,\nabla\left(\frac{\delta F}{\delta \varrho}(\varrho)\right)\right)=0.
\end{equation}

This is a PDE where the velocity field in the continuity equation depends on the density $\varrho$ itself. 
An interesting example is the case where we use $F(\varrho)=\int f(\varrho(x))\dd x,$ with $f(t)=t\log t$. In such a case we have $f'(t)=\log t +1$ and $\nabla(f'(\varrho))=\frac{\nabla \varrho}{\varrho}$: this means that the gradient flow equation associated to the functional $F$ would be the {\it Heat Equation} $\partial_\varrho-\Delta\varrho=0$. Using $F(\varrho)=\int f(\varrho(x))\dd x+\int V(x)\dd\varrho(x)$, we would have the {\it Fokker-Planck Equation} $\partial_\varrho-\Delta\varrho-\nabla\cdot(\varrho\nabla V)=0$. We will see later which other interesting PDEs can be obtained in this way.

Many possible proofs can be built for the convergence of the above iterated minimization scheme. In particular, one could follow the general theory developed in \cite{AmbGigSav}, i.e. checking all the assumptions to prove existence and uniqueness of an EVI gradient flow for the functional $F$ in the space $\mathbb W_2(\Omega)$, and then characterizing the velocity field that Theorem \ref{curves and PDE} associates with the curve obtained as a gradient flow. In \cite{AmbGigSav}, it is proven, under suitable conditions, that such a vector field $\vv_t$ must belong to what is defined as the {\it Wasserstein sub-differential} of the functional $F$, provided in particular that $F$ is $\lambda$-convex. Then, the {\it Wasserstein sub-differential}  is proven to be of the desired form (i.e. composed only of the gradient of the first variation of $F$, when $F$ admits a first variation). 

This approach has the advantage to use a general theory and to adapt it to the scopes of this particular setting. On the other hand, some considerations seem necessary:

\begin{itemize}
\item the important point when studying these PDEs is that the curves $(\varrho_t)_t$ obtained as a limit are true weak solutions of the continuity equations; from this point of view, the notions of EDE and EVI solutions and the formalism developed in the first part of the book \cite{AmbGigSav} (devoted to the general metric case) are not necessary; if the second part of \cite{AmbGigSav} is exactly concerned with Wasserstein spaces and with the caracterization of the limit as $\tau\to 0$ as the solution of a PDE, we can say that the whole formalism is sometimes too heavy.
\item after using optimal transport thery to select a suitable distance in the discrete scheme above and a suitable interpolation, the passage to the limit can be done by classical compactness techniques in the space of measures and in functional analysis; on the other hand, there are often some difficulties in handling some non-linear terms, which are not always seen when using the theory of \cite{AmbGigSav} (which is an advantage of this general theory).
\item the $\lambda$-convexity assumption is in general not crucial in what concerns existence (but the general theory in \cite{AmbGigSav} has been built under this assumption, as we saw in Section \ref{metricth}).
\item as far as uniqueness of the solution is concerned, the natural point of view in the applications would be to prove uniqueness of the weak solution of the equation (or, possibly, to define a more restrictive notion of solution of the PDE for which one can prove uniqueness), and this is a priori very different from the EDE or EVI notions. To this aim, the use of the Wasserstein distance can be very useful, as one can often prove uniqueness by differentiating in time the distance between two solutions, and maybe apply a Gronwall lemma (and this can be done independently of the EVI notion; see for instance the end of section \ref{geodconvW2}). On the other hand, uniqueness results are almost never possible without some kind of $\lambda$-convexity (or weaker versions of it, as in \cite{KatyCraig}) of the functional.
\end{itemize}

For the reader who wants an idea of how to prove convergence of the scheme to a solution of the PDE independently of the EDE/EVI theory, here are some sketchy ideas. Everything will be developed in details in Section \ref{secFP} in a particular case.

The main point lies in the interpolation of the sequence $\varrho^\tau_k$ (and of the corresponding velocities $\vv^\tau_k$). Indeed, two possible  interpolations turn out to be useful:  on the one hand, we can define an interpolation $(\varrho^\tau,\vv^\tau)$ which is piecewise constant, as in  \eqref{interp const} (to define  $\vv^\tau$ we use $\nabla\varphi/\tau$); on the other hand, we can connect each measure $\varrho^\tau_k$ to $\varrho^\tau_{k+1}$ by using a piecewise geodesic curve $\tilde\varrho^\tau$, where geodesics are chosen according to the Wasserstein metric, using Theorem \eqref{geodesics in Wp} to get an explicit expression. This second interpolation is a Lipschitz curve in $\mathbb W_2(\Omega)$, and has an explicit velocity field, that we know thanks to Theorem \eqref{geodesics in Wp}: we call it $\tilde \vv^\tau $ and it is related to $\vv^\tau$. The advantage of the second interpolation is that $(\tilde\varrho^\tau,\tilde \vv^\tau)$ satisfies the continuity equation. On the contrary, the first interpolation is not continuous, the continuity equation is not satisfied, but the optimality conditions at each time step provide a connection between $\vv^\tau$ and $\varrho^\tau$ ($\vv^\tau=-\nabla\big(\frac{\delta F}{\delta \varrho}(\varrho^\tau))$),which is not satisfied with $\tilde\varrho^\tau$ and $\tilde \vv^\tau$. It is possible to prove that the two interpolations converge to the same limit as $\tau\to 0$, and that the limit will satisfy a continuity equation with a velocity vector field $\vv=-\nabla\big(\frac{\delta F}{\delta \varrho}(\varrho))$, which allows to conclude.

\subsection{Geodesic convexity in $\WW_2$}\label{geodconvW2}

Even if we insisted on the fact that the most natural approach to gradient flows in $\mathbb W_2$ relies on the notion of weak solutions to some PDEs and not on the EDE/EVI formulations, for sure it could be interesting to check whether the general theory of Section \ref{metricth} could be applied to some model functionals on the space of probabilities, such as $\mathcal F$, $\mathcal V$ or $\mathcal W$. This requires to discuss their $\lambda$-convexity, which is also useful because it can provide uniqueness results. As we now know the geodesics in the Wasserstein space, the question of which functionals are geodesically convex is easy to tackle. The notion of geodesic convexity in the $\mathbb W_2$ space, aslo called {\it displacement convexity}, has first been introduced by McCann in \cite{MC}.

\paragraph{{\bf Displacement convexity of $\mathcal F$, $\mathcal V$ and $\mathcal W$.}}

It is not difficult to check that the convexity of $V$ is enough to guarantee geodesic convexity of $\mathcal V$, since
$$\mathcal V(\mu_t)=\int V\,\dd\big((1-t)id+tT\big)_\#\mu=\int V\big((1-t)x+tT(x)\big)\dd\mu,$$
as well as the convexity of $W$ guarantees that of $\mathcal W$:
\begin{eqnarray*}
\mathcal W(\mu_t)&=&\int W(x-y)\,\dd\left(\big((1-t)id+tT\big)_\#\mu\big)\otimes\big(\big((1-t)id+tT\big)_\#\mu\right)(x,y)\\&=&\int W\big((1-t)x+tT(x),(1-t)y+tT(y)\big)\dd(\mu\otimes \mu).
\end{eqnarray*}
Similarly, if $V$ or $W$ are $\lambda$-convex we get $\lambda$-geodesical convexity. Note that in the case of $\mathcal V$ it is easy to see that the $\lambda$-convexity of $V$ is also necessary for the $\lambda$-geodesical convexity of $\mathcal V$, while the same is not true for $W$ and $\mathcal W$.

  The most interesting displacement convexity result is the one for functionals depending on the density. 
  To consider these functionals, we need some technical facts.
  
  The starting point is the computation of the density of an image measure, via standard change-of-variable techniques: if $T:\Omega\to \Omega$ is a map smooth enough\footnote{We need at least $T$ to be countably Lipschitz, i.e. $\Omega$ may be written as a countable union of measurable sets $(\Omega_i)_{i\geq 0}$ with $\Omega_0$ negligible and $T\res \Omega_i$ Lipschitz continuous for every $i\geq 1$.} and injective, and $\det(DT(x))\neq 0$ for a.e. $x\in\Omega$, if we set $\varrho_T:=T_\#\varrho$, then $\varrho_T$ is absolutely continuous with density given by
  $$\varrho_T=\frac{\varrho}{\det(DT)}\circ T^{-1}.$$
  
 Then, we underline an interesting computation, which can be proven as an exercice.
 \begin{lemma}\label{concagt}
 Let $A$ be a $d\times d$ matrix such that its eigenvalues $\lambda_i$  are all real and larger than $-1$ (for instance this is the case when $A$ is symmetric and $\mathrm{I}+A\geq 0$). Then $[0,1]\ni t\mapsto g(t):= \det(\mathrm{I}+tA)^{1/d}$ is concave.
 \end{lemma}

We can now state the main theorem.
\begin{theorem}\label{convMC}
Suppose that $f(0)=0$ and that $s\mapsto s^{d}f(s^{-d})$ is convex and decreasing. Suppose that $\Omega$ is convex and take $1<p<\infty$. Then $\mathcal F$ is geodesically convex in $\WW_2$. \end{theorem}
\begin{proof}
Let us consider two measures $\mu_0$, $\mu_1$ with $\mathcal F(\mu_0),\mathcal F(\mu_1)<+\infty$. They are absolutely continuous and hence there is a unique constant-speed geodesic $\mu_t$ between them, which has the form $\mu_t=(T_t)_\#\mu_0$, where $T_t=\id+t(T-\id)$. Note that we have $T_t(x)=x-t\nabla\varphi(x),$ where $\varphi$ is such that $\frac{x^2}{2}-\varphi$ is convex. This implies that $D^2\varphi\leq I$ and $T_t$ is, for $t<1$, the gradient of a strictly convex function, hence it is injective. Moreover $\nabla\varphi$ is countably Lipschitz, and so is $T_t$. From the formula for the density of the image measure,  we know that $\mu_t$ is absolutely continuous and we can write its density $\varrho_t$ as $\varrho_t(y)=\varrho(T_t^{-1}(y))/\det( \mathrm{I}+tA(T_t^{-1}(y)))$, where $A=-D^2\varphi$ and $\varrho$ is the density of $\mu$, and hence
$$\mathcal F(\mu_t)=\int f\left(\frac{\varrho(T_t^{-1}(y)))}{\det( \mathrm{I}+tA(T_t^{-1}(y)))}\right)\dd y=
\int f\left(\frac{\varrho(x)}{\det( \mathrm{I}+tA(x)}\right)\det (\mathrm{I}+tA(x))\,\dd x,$$
where we used the change of variables $y=T_t(x)$ and $\dd y=\det DT_t(x)\,\dd x=\det(\mathrm{I}+tA(x))\,\dd x$.

From Lemma \ref{concagt} we know that $\det (\mathrm{I}+tA(x))=g(t,x)^d$ for a function $g:[0,1]\times \Omega$ which is concave in $t$. It is a general fact that the composition of a convex and decreasing function with a concave one gives a convex function. This implies that 
$$t\mapsto f\left(\frac{\varrho(x)}{g(t,x)^d}\right)g(t,x)^d$$
is convex (if $\varrho(x)\neq 0$ this uses the assumption on $f$ and the fact that $t\mapsto g(t,x)/\varrho(x)^{\frac1d}$ is concave; if $\varrho(x)=0$ then this function is simply zero). Finally, we proved convexity of $t\mapsto \mathcal F(\mu_t)$. 
\end{proof}
\begin{remark}
Note that the assumption that $s\mapsto s^{d}f(s^{-d})$ is convex and decreasing implies that $f$ itself is convex (the reader can check it as an exercise), a property which can be useful to establish, for instance, lower semicontinuity of $\mathcal F$.
\end{remark}

Let us see some easy examples of convex functions satisfying the assumptions of Theorem \ref{convMC}:
\begin{itemize}
\item for any $q>1$, the function $f(t)=t^q$ satisfies these assumptions, since $s^df(s^{-d})=s^{-d(q-1)}$ is convex and decreasing;
\item the entropy function $f(t)=t\log t$ also satisfies the assumptions since $s^df(s^{-d})=-d\log s$ is convex and decreasing;\index{entropy}
\item if $1-\frac 1d\leq m<1$ the function $f(t)=-t^m$ is convex, and if we compute $s^df(s^{-d})=-t^{m(1-d)}$ we get a convex and decreasing function since $m(1-d)<1$. Note that in this case $f$ is not superlinear, which requires some attention for the semicontinuity of $\mathcal F$. 
\end{itemize}

\paragraph{{\bf Convexity on generalized geodesics.}}

It is quite disappointing to note that the functional $\mu\mapsto W_2^2(\mu,\nu)$ is not, in general, displacement convex. This seems contrary to the intuition because usually squared distances are nice convex functions\footnote{Actually, this is true in normed spaces, but not even true in Riemannian manifolds, as it depends on curvature properties.}. However, we can see that this fails from the following easy example. Take $\nu=\frac 12\delta_{(1,0)}+\frac 12\delta_{(-1,0)}$ and $\mu_t=\frac 12\delta_{(t,a)}+\frac 12\delta_{(-t,-a)}$. If $a>1$, the curve $\mu_t$ is the geodesic between $\mu_{-1}$ and $\mu_{1}$ (because the optimal transport between these measures sends $(a,-1)$ to $(a,1)$ and $(-a,1)$ to $(-a,-1)$. Yet, if we compute $W_2^2(\mu_t,\nu)$ we have
$$W_2^2(\mu_t,\nu)= a^2+\min\{(1-t)^2,(1+t)^2\}.$$
But this function is not convex! (see Figure \ref{gen conv})

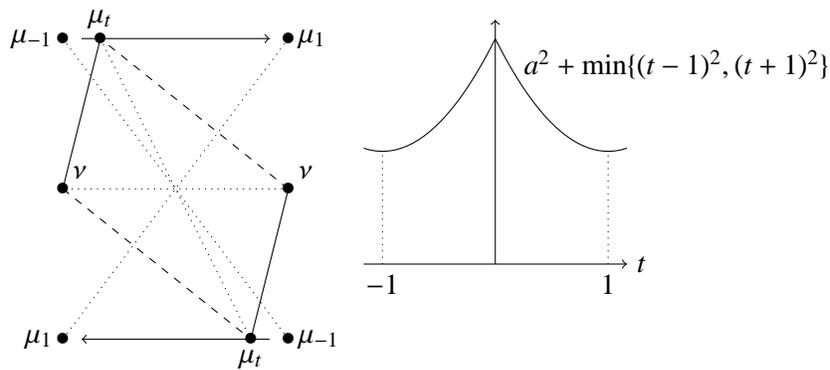
\begin{figure}[h]
\begin{center}
\begin{tikzpicture}[scale=0.5]

\draw (3,0) node{$\bullet$} node[above right]{$\nu$};
\draw (-3,0) node{$\bullet$} node[above right]{$\nu$};
\draw[dotted] (3,0)--(-3,0);
\draw (3,4) node{$\bullet$} node[right]{$\mu_1$};
\draw (-3,-4) node{$\bullet$} node[left]{$\mu_1$};
\draw[dotted] (3,4)--(-3,-4); 

\draw (-3,4) node{$\bullet$} node[left]{$\mu_{-1}$};
\draw (3,-4) node{$\bullet$} node[right]{$\mu_{-1}$};
\draw[dotted] (-3,4)--(3,-4);

\draw (-2,4) node{$\bullet$} node[above]{$\mu_{t}$};
\draw (2,-4) node{$\bullet$} node[below]{$\mu_t$};
\draw[dotted] (-2,4)--(2,-4);

\draw[->] (-2.5,4)--(2.5,4);
\draw[->] (2.5,-4)--(-2.5,-4);

\draw (3,0)--(2,-4);
\draw[dashed] (3,0)--(-2,4);
\draw (-3,0)--(-2,4);
\draw[dashed] (-3,0)--(2,-4);

\draw[->] (5,-2)--(12,-2);
\draw[->] (8.5,-2)--(8.5,4.5);
\draw plot[domain=5:8.5] (\x,{1+3*(0.333*(\x-5.5))^2});
\draw plot[domain=8.5:12] (\x,{1+3*(0.333*(\x-11.5))^2});
\draw[dotted] (5.5,-2)--(5.5,1);
\draw[dotted] (11.5,-2)--(11.5,1);
\draw (5.5,-2) node[below]{$-1$};
\draw (11.5,-2) node[below]{$1$};
\draw (12,-2) node[right]{$t$};
\draw (9,3.3) node[right]{$a^2+\min\{(t-1)^2,(t+1)^2\}$};
\end{tikzpicture}
\caption{The value of $W_2^2(\mu_t,\nu)$.}\label{gen conv}
\end{center}
\end{figure}

The lack of geodesic convexity of this easy functional\footnote{By the way, this functional can even be proven to be somehow geodetically concave, as it is shown in \cite{AmbGigSav}, Theorem 7.3.2.} is a problem for many issues, and in particular for the C$^2$G$^2$ condition, and an alternate notion has been proposed, namely that of convexity along generalized geodesics. 

\begin{definition}
If we fix an absolutely continuous probability $\varrho\in\pical(\Omega)$,  for every pair $\mu_0,\mu_1\in\pical(\Omega)$ we call {\it generalized geodesic} between $\mu_0$ and $\mu_1$ with base $\varrho$ in $\WW_2(\Omega)$  the curve $\mu_t=((1-t)T_0+tT_1)_\#\varrho$, where $T_0$ is the optimal transport map (for the cost $|x-y|^2$) from $\varrho$ to $\mu_0$, and $T_1$ from $\varrho$ to $\mu_1$.
\end{definition}

It is clear that $t\mapsto W_2^2(\mu_t,\varrho)$ satisfies 
\begin{eqnarray*}
W_2^2(\mu_t,\varrho)&\leq& \int |((1-t)T_0(x)+tT_1(x))-x|^2\dd\varrho(x)\\&\leq& (1-t)\int |T_0(x)-x|^2\dd\varrho(x)+t\int |T_1(x)-x|^2\dd\varrho(x)
=(1-t)W_2^2(\mu_0,\varrho)+tW_2^2(\mu_1,\varrho)
\end{eqnarray*}
and hence we have the desired convexity along this curve. Moreover, similar considerations to those we developed in this section show that all the functionals that we proved to be geodesically convex are also convex along generalized geodesics. For the case of functionals $\mathcal V$ and $\mathcal W$ this is easy, while for the case of the functional $\mathcal F$, Lemma \ref{concagt} has to be changed into ``$t\mapsto \det((1-t)A+tB)^{1/d}$ is concave, whenever $A$ and $B$ are symmetric and positive-definite'' (the proof is similar). We do not develop these proofs here, and we refer to \cite{AmbGigSav} or Chapter 7 in \cite{OTAM} for more details.

Of course, we could wonder whether the assumption C$^2$G$^2$ is satisfied in the Wasserstein space $\mathbb W_2(\Omega)$ for these functionals $\mathcal F, \mathcal V$ and $\mathcal W$: actually, if one looks closer at this questions, it is possible to see that the very definition of C$^2$G$^2$ has been given on purpose in order to face the case of Wasserstein spaces. Indeed, if we fix $\nu\in\pical(\Omega)$ and take $\mu_0,\mu_1$ two other probabilities, with $T_0,T_1$ the optimal transports from $\nu$ to $\mu_0$ and $\mu_1$, respectively, then the curve
\begin{equation}\label{geod gen}
\mu_t:=((1-t)T_0+tT_1)_\#\nu
\end{equation}
connects $\mu_0$ to $\mu_1$ and can be used in C$^2$G$^2$.

\paragraph{{\bf Displacement convexity and curvature conditions.}}

In the proof of the geodesic convexity of the functional $\mathcal F$ we strongly used the Euclidean structure of the geodesics in the Wasserstein space. The key point was form of the intermediate map $T_t$: a linear interpolation between the identity map $id$ and the optimal map $T$, together with the convexity properties of $t\mapsto\det(I+tA)^{1/d}$. On a Riemannian manifold, this would completely change as the geodesic curves between $x$ and $T(x)$, which are no more segments, could concentrate more or less than what happens in the Euclidean case, depending on the curvature of the manifold (think at the geodesics on a sphere connecting points close to the North Pole to points close to the South Pole: they are much farther from each other at their middle points than at their endpoints). It has been found (see \cite{vReStu}) that the condition of $\lambda$-geodesic convexity of the  Entropy functional $\varrho\mapsto \int\varrho\log(\varrho)\dd\mathrm{Vol}$ (where $\varrho$ is absolutely continuous w.r.t. the volume meaure $\mathrm{Vol}$ and densities are computed accordingly) on a manifold characterizes a lower bound on its Ricci curvature:
\begin{proposition}\label{SLV}
Let $M$ be a compact manifold of dimension $d$ and $\mathrm{Vol}$ its volume measure. Let $\mathcal E$ be the entropy functional defined via $\mathcal E(\varrho)=\int \varrho\log \varrho \,\dd\mathrm{Vol}$ for all measures $\varrho\ll\mathrm{Vol}$ (set to $+\infty$ on non-absolutely continuous measures). Then $\mathcal E$ is $\lambda$-geodesically convex in the Wasserstein space $\mathbb W_2(M)$ if and only if the Ricci curvature $\mathrm{Ric}_M$ satisfies $\mathrm{Ric}_M\geq \lambda$.
In the case $\lambda=0$, the same equivalence is true if one replaces the entropy function $f(t)=t\log t$ with the function $f_N(t)=-t^{1-1/N}$ with $N\geq d$.
\end{proposition}

This fact will be the basis (we will see it in Section \ref{heat}) of a definition based on optimal transport of the notion of Ricci curvature bounds in more abstract spaces, a definition independently proposed in two celebrated papers by Sturm, \cite{sturm} and Lott and Villani, \cite{LotVil}.

\begin{remark}
The reader who followed the first proofs of this section has for sure observed that it is easy, in the space $\mathbb W_2(\Omega)$, to produce $\lambda$-geodesically convex functionals which are not geodesically convex (with $\lambda<0$, of course), of the form $\mathcal V$ (just take a $\lambda$-convex function $V$ which is not convex), but that Theorem \ref{convMC} only provides geodesic convexity (never provides $\lambda$-convexity without convexity) for functionals of the form $\mathcal F$: this is indeed specific to the Euclidean case, where the optimal transport has the form $T(x)=x-\nabla\varphi(x)$; in Riemannian manifolds or other metric measure spaces, this can be different!
\end{remark}

\paragraph{{\bf Geodesic convexity and uniqueness of gradient flows.}} The fact that $\lambda$-convexity is a crucial tool to establish uniqueness and stability results in gradient flows is not only true in the abstract theory of Section \ref{metricth}, where we saw that the EVI condition (intimately linked to $\lambda$-convexity) provides stability. Indeed, it can also be observed in the concrete case of gradient flows in $\mathbb W_2(\Omega)$, which take the form of the PDE \eqref{PDEgradflow}. We will see this fact via some formal considerations, starting from an interesting lemma:

\begin{lemma}\label{geodconvimpliesgronw}
Suppose that $F:\mathbb W_2(\Omega)\to\R\cup\{+\infty\}$ is $\lambda$-geodesically convex. Then, for every $\varrho_0,\varrho_1\in\pical(\Omega)$ for which the integrals below are well-defined, we have
$$\int \nabla\varphi\cdot\nabla\left(\frac{\delta F}{\delta\varrho}(\varrho_0)\right)\dd \varrho_0+\int \nabla\psi\cdot\nabla\left(\frac{\delta F}{\delta\varrho}(\varrho_1)\right)\dd \varrho_1\geq \lambda W_2^2(\varrho_0,\varrho_1),$$
where $\varphi$ is the Kantorovich potential in the transport from $\varrho_0$ to $\varrho_1$ and $\psi=\varphi^c$ is the Kantorovich potential from $\varrho_1$ to $\varrho_0$. 
\end{lemma}
\begin{proof}
Let $\varrho_t$ be the the geodesic curve in $\mathbb W_2(\Omega)$ connecting $\varrho_0$ to $\varrho_1$ and $g(t):=F(\varrho_t)$. The assumption of $\lambda$-convexity means (see the definition in \eqref{defilambda})
$$g(t)\leq (1-t)g(0)+tg(1)-\frac\lambda 2 t(1-t) W_2^2(\varrho_0,\varrho_1).$$
Since the above inequality is an equality for $t=0,1$, we can differentiate it at these two points thus obtaining
$$g'(0)\leq g(1)-g(0)-\frac\lambda 2  W_2^2(\varrho_0,\varrho_1),\quad g'(1)\geq  g(1)-g(0)+\frac\lambda 2  W_2^2(\varrho_0,\varrho_1),$$
which implies 
$$g'(0)-g'(1)\leq -\lambda W_2^2(\varrho_0,\varrho_1).$$
Then, we compute the derivative of $g$, formally obtaining
$$g'(t)=\int \frac{\delta F}{\delta\varrho}(\varrho_t)\partial_t\varrho_t=-\int \frac{\delta F}{\delta\varrho}(\varrho_t)\nabla\cdot(\vv_t \varrho_t)=\int \nabla\left(\frac{\delta F}{\delta\varrho}(\varrho_t)\right)\cdot\vv_t\,\dd \varrho_t,$$
and, using $\vv_0=-\nabla\varphi$ and $\vv_1=\nabla\psi$, we obtain the claim.
\end{proof}

With the above lemma in mind, we consider two curves $\varrho^0_t$ and $\varrho^1_t$, with 
$$\partial_t \varrho^i_t+\nabla\cdot(\varrho^i_t \vv^i_t)=0\qquad\mbox{ for }i=0,1,$$
where the vector fields $\vv^i_t$ are their respective velocty fields provided by Theorem \ref{curves and PDE}. Setting $d(t):=\frac 12 W_2^2(\varrho^0_t,\varrho^1_t)$, it is natural to guess that we have
$$d'(t)= \int \nabla\varphi\cdot\vv^0_t\dd\varrho^0_t+\int \nabla\psi\cdot\vv^1_t\dd\varrho^1_t,$$
where $\varphi$ is the Kantorovich potential in the transport from $\varrho^0_t$ to $\varrho^1_t$ and $\psi=\varphi^c$. Indeed, a rigorous proof is provided in \cite{AmbGigSav} or in Section 5.3.5 of \cite{OTAM}, but one can guess it from the duality formula
$$d(t)=\max\left\{\int\!\phi \,\dd\varrho^0_t+\!\!\int\!\psi \,\dd\varrho^1_t\;:\;\phi(x)\!+\!\psi(y)\leq \frac12|x-y|^2)\right\}.$$
As usual, differentiating an expression written as a max involves the optimal functions $\phi,\psi$ in such a max, and the terms $\partial_t \varrho^i_t$ have been replaced by $-\nabla\cdot(\varrho^i_t \vv^i_t)$ as in the proof Lemma \ref{geodconvimpliesgronw}.

When the two curves $\varrho^0_t$ and $\varrho^1_t$ are solutions of \eqref{PDEgradflow} we have $\vv^i_t=-\nabla\left(\frac{\delta F}{\delta\varrho}(\varrho^i_t)\right)$, and Lemma \ref{geodconvimpliesgronw} allows to obtain the following:
\begin{proposition}
Suppose that $F:\mathbb W_2(\Omega)\to\R\cup\{+\infty\}$ is $\lambda$-geodesically convex and that the  two curves $\varrho^0_t$ and $\varrho^1_t$ are solutions of \eqref{PDEgradflow}. Then, setting $d(t):=\frac 12W_2^2(\varrho^0_t,\varrho^1_t)$, we have
$$d'(t)\leq -\lambda d(t).$$
This implies uniqueness of the solution of \eqref{PDEgradflow} for fixed initial datum, stability, and exponential convergence as $t\to+\infty$ if $\lambda>0$.\end{proposition}

\subsection{Analysis of the Fokker-Planck equation as a gradient flow in $\mathbb W_2$}\label{secFP}

This section will be a technical parenthesis providing proof details in a particular case study, that of the Fokker-Planck equation, which is  the gradient flow of the functional \index{entropy}\index{potential energy}
$$J(\varrho)=\int_\Omega \varrho\log \varrho + \int_\Omega V\dd\varrho,$$
 where $V$ is a $C^1$ function\footnote{The equation would also be meaningful for $V$ only Lipschitz continuous, but we prefer to stick to the $C^1$ case for simplicity.} on a domain $\Omega$, that we will choose compact for simplicity. The initial measure $\varrho_0\in\pical(\Omega)$ is taken such that $J(\varrho_0)<+\infty$. 
 
 We stress that the first term of the functional is defined as
 $$\mathcal E(\varrho):=\begin{cases}\int_\Omega \varrho(x)\log \varrho(x)\,\dd x&\mbox{ if }\varrho\ll\lcal^d,\\
					+\infty &\mbox{ otherwise,}\end{cases}$$
where we identify the measure $\varrho$ with its density, when it is absolutely continuous. This functional is l.s.c. for the weak topology of measures (for general references on the semicontinuity of convex functionals on the space of measures we refer to Chapter 7 in \cite{OTAM} or to \cite{But-lsc}), which is equivalent, on the compact domain $\Omega$, to the $W_2$ convergence.
Semi-continuity allows to establish the following:
\begin{proposition}
The functional $J$ has a unique minimum over $\pical(\Omega)$. In particular $J$ is bounded from below. Moreover, for each $\tau>0$ the following sequence of optimization problems recursively defined is well-posed
\begin{equation}\label{rhok+}
\varrho^\tau_{(k+1)}\in \argmin_\varrho\quad J(\varrho)+\frac{W_2^2(\varrho,\varrho^\tau_{(k)})}{2\tau},
\end{equation}
which means that there is a minimizer at every step, and this minimizer is unique.
\end{proposition}
\begin{proof}
Just apply the direct method, noting that $\pical(\Omega)$ is compact for the weak convergence, which is the same as the convergence for the $W_2$ distance (again, because $\Omega$ is compact), and for this convergence $\mathcal F$ is l.s.c. and the other terms are continuous. This gives at the same time the existence of a minimizer for $J$ and of a solution to each of the above minimization problems \eqref{rhok+}. Uniqueness comes from the fact that all the functionals are convex (in the usual sense) and $\mathcal F$ is strictly convex.
\end{proof}

\paragraph{{\bf Optimality conditions at each time step.}}

A first preliminary result we need is the following :
\begin{lemma}\label{pos and log}
Any minimizer $\widehat\varrho$ in \eqref{rhok+} must satisfy $\widehat\varrho>0$ a.e.
\end{lemma}
\begin{proof}
Consider the measure $\widetilde\varrho$ with constant positive density $c$ in $\Omega$ (i.e. $c=|\Omega|^{-1}$). Let us define $\varrho_\ve$ as $(1-\ve)\widehat\varrho+\ve\widetilde\varrho$ and compare $\widehat\varrho$ to $\varrho_\ve$. 

By optimality of $\widehat\varrho$, we may write
\begin{equation}\label{a estimer with C}
\mathcal E(\widehat\varrho)-\mathcal E(\varrho_\ve)\leq  \int_\Omega V\dd\varrho_\ve - \int_\Omega V\dd\widehat\varrho+\frac{W_2^2(\varrho_\ve,\varrho^\tau_{(k)})}{2\tau}-\frac{W_2^2(\widehat\varrho,\varrho^\tau_{(k)})}{2\tau}.
\end{equation}
The two differences in the right hand side may be easily estimated (by convexity, for instance) so that we get (for a constant $C$ depending on $\tau$, of course):
$$\int_\Omega f(\widehat\varrho)-f(\varrho_\ve)\leq C\ve$$
where $f(t)=t\log t$ (set to $0$ in $t=0$). Write 
$$A=\{x\in\Omega\,:\,\widehat\varrho(x)>0\},\quad B=\{x\in\Omega\,:\,\widehat\varrho(x)=0\}.$$
Since $f$ is convex we write, for $x\in A$, $ f(\widehat\varrho(x))-f(\varrho_\ve(x))\geq (\widehat\varrho(x)-\varrho_\ve(x))f'(\varrho_\ve(x))=\ve(\widehat\varrho(x)-\widetilde\varrho(x))(1+\log\varrho_\ve(x))$. For $x\in B$ we simply write $ f(\widehat\varrho(x))-f(\varrho_\ve(x))=-\ve c \log(\ve c)$. This allows to write
\begin{equation*}
-\ve c \log(\ve c)|B|+\ve \int_A (\widehat\varrho(x)-c)(1+\log\varrho_\ve(x))\,\dd x\leq C\ve
\end{equation*}
and, dividing by $\ve$, 
\begin{equation}\label{avant d'estimer nicolas}
-c \log(\ve c)|B|+ \int_A (\widehat\varrho(x)-c)(1+\log\varrho_\ve(x))\,\dd x\leq C
\end{equation}
Note that the always have 
$$(\widehat\varrho(x)-c)(1+\log\varrho_\ve(x))\geq (\widehat\varrho(x)-c)(1+\log c)$$ 
(just distinguish between the case $\widehat\varrho(x)\geq c$ and $\widehat\varrho(x)\leq c)$. Thus,  we may write
$$- c \log(\ve c)|B|+\int_{A} (\widehat\varrho(x)-c)(1+\log c)\,\dd x\leq C.$$
Letting $\ve\to 0$ provides a contradiction, unless $|B|=0$.
\end{proof}
\begin{remark}
If, after proving $|B|=0$, we go on with the computations (as it is actually done in Chapter 8 in \cite{OTAM}), we can also obtain $\log\widehat\varrho\in L^1$, which is a stronger condition than just $\widehat\varrho>0$ a.e.
\end{remark}
%

We can now compute the first variation and give optimality conditions on the optimal $\varrho^\tau_{(k+1)}$.

\begin{proposition}\label{opticond}
The optimal measure $\varrho^\tau_{(k+1)}$ in \eqref{rhok+} satisfies
\begin{equation}\label{inphik}
\log(\varrho^\tau_{(k+1)})+V+\frac{\varphi}{\tau}=\mbox{const. a.e.}
\end{equation}
where $\varphi$ is the (unique) Kantorovich potential from $\varrho^\tau_{(k+1)}$ to $\varrho^\tau_{(k)}$.
In particular, $\log\varrho^\tau_{(k+1)}$ is Lipschitz continuous. If $T^\tau_k$ is the optimal transport  from $\varrho^\tau_{(k+1)}$ to $\varrho^\tau_{(k)}$, then it satisfies
\begin{equation}\label{vk}
\vv^\tau_{(k)}:=\frac{\id-T^\tau_k}{\tau}=-\nabla\left(\log(\varrho^\tau_{(k+1)})+V\right)\;\mbox{a.e.}
\end{equation}
\end{proposition}
\begin{proof}
Take the optimal measure $\widehat\varrho:=\varrho^\tau_{(k+1)}$. We can check that all the functionals involved in the minimization admit a first variation at $\widehat\varrho$. For the linear term it is straightforward, and for the Wasserstein term we can apply Proposition \ref{conv_subdiff_first var}. The uniqueness of the Kantorovich potential is guaranteed by the fact that $\widehat\varrho$ is strictly positive a.e. on the domain $\Omega$ (Lemma \ref{pos and log}). For the entropy term, the integrability of $\log\widehat\varrho$ provided in the previous remark allows (computations are left to the reader) to differentiate under the integral sign for smooth perturbations.  

The first variation of $J$ is hence 
$$\frac{\delta J}{\delta\varrho}=f'(\varrho)+V=\log(\varrho)+1+V.$$
Using carefully (details are in Chapter 8 of \cite{OTAM}) the optimality conditions, we obtain Equation \eqref{inphik}, which is valid a.e. since $\widehat\varrho>0$. In particular, this implies that $\varrho^\tau_{(k+1)}$ is Lipschitz continuous, since we have
$$  \varrho^\tau_{(k+1)}(x)=\exp\left(C-V(x)-\frac{\varphi(x)}{\tau}\right).$$ 
Then, one differentiates and gets the equality 
$$\nabla\varphi=\frac{\id-T^\tau_k}{\tau}=-\nabla\left(\log(\varrho^\tau_{(k+1)})+V\right)\;\mbox{a.e.}$$
and this allows to conclude.
\end{proof}

\paragraph{{\bf Interpolation between time steps and uniform estimates}}

Let us collect some other tools
\begin{proposition}
For any $\tau>0$, the sequence of minimizers satisfies 
$$\sum_k\frac{W_2^2(\varrho^\tau_{(k+1)},\varrho^\tau_{(k)})}{\tau}\leq C:=2(J(\varrho_0)-\inf J).$$
\end{proposition}
\begin{proof}This is just the standard estimate in the minimizing movement scheme, corresponding to what we presented in \eqref{1sttimeH1} in the Euclidean case.\end{proof}

Let us define two interpolations between the measures $\varrho^\tau_{(k)}$.

With this time-discretized method, we have obtained, for each $\tau>0$, a sequence $(\varrho^\tau_{(k)})_k$. We can use it to build at least two interesting curves in the space of measures: 
\begin{itemize}
\item first we can define some piecewise constant curves, i.e. $\varrho^\tau_t:=\varrho^\tau_{(k+1)}$ for $t\in ]k\tau,(k+1)\tau]$; associated to this curve we also define the velocities $\vv^\tau_t=\vv^\tau_{(k+1)}$ for $t\in ]k\tau,(k+1)\tau]$, where $\vv^\tau_{(k)}$ is defined as in \eqref{vk}: $\vv^\tau_{(k)}=(\id-T^\tau_k)/\tau$, taking as $T^\tau_k$ the optimal transport from $\varrho^\tau_{(k+1)}$ to $\varrho^\tau_{(k)}$; we also define the momentum variable $E^\tau=\varrho^\tau \vv^\tau$;
\item then, we can also consider  the densities $\widetilde \varrho^\tau_t$ that interpolate the discrete values $(\varrho^\tau_{(k)})_k$ along geodesics:\index{geodesics!in $\WW_p$}
\begin{equation}
\widetilde \varrho^\tau_t = \left( \dfrac{k\tau-t}{\tau} \vv^\tau_{(k)} +\id \right)_{\#} \varrho^\tau_{(k)},\;\mbox{ for }t\in](k-1)\tau,k\tau[;
\end{equation}
the velocities $\widetilde \vv^{\tau}_t$ are defined so that $(\widetilde \varrho^\tau, \widetilde \vv^{\tau})$ satisfy the continuity equation and $||\widetilde \vv^{\tau}_t||_{L^2(\widetilde \varrho^\tau_t)}=|(\widetilde\varrho^\tau)'|(t)$. To do so, we take
$$\widetilde \vv^{\tau}_t=\vv^\tau_t\circ\left( (k\tau-t) \vv^\tau_{(k)} +\id \right)^{-1};$$ 
as before, we define a momentum variable: $\widetilde E_{\tau}= \widetilde \varrho^\tau \widetilde \vv^{\tau}$.
\end{itemize}

After these definitions we consider some a priori bounds on the curves and the velocities that we defined. We start from some estimates which are standard in the framework of Minimizing Movements.

Note that the velocity (i.e. metric derivative\index{metric derivative}) of $\widetilde\varrho^\tau$ is constant on each interval $]k\tau,(k+1)\tau[$ and equal to 
$$\frac{W_2(\varrho^\tau_{(k+1)},\varrho^\tau_{(k)})}{\tau}=\frac 1\tau\left(\int |\id-T^\tau_k|^2\dd\varrho^\tau_{(k+1)}\right)^{1/2}=||\vv_{k+1}^\tau||_{L^2(\varrho^\tau_{(k+1)})},$$ 
which gives
$$||\widetilde \vv^\tau_t||_{L^2(\widetilde\varrho^\tau_t)}=|(\widetilde\varrho^\tau)'|(t)=\frac{W_2(\varrho^\tau_{(k+1)},\varrho^\tau_{(k)})}{\tau}=|| \vv^\tau_t||_{L^2(\varrho^\tau_t)},$$
where we used the fact that the velocity field $\widetilde \vv^\tau$ has been chosen so that its $L^2$ norm equals the metric derivative of the curve $\widetilde\varrho^\tau$.

In particular we can obtain
\begin{eqnarray*}
|E^\tau|([0,T]\times\Omega)&=&\int_0^T \dd t \int_\Omega |\vv^\tau_t|\dd\varrho^\tau_t=\int_0^T ||\vv^\tau_t||_{L^1(\varrho^\tau_t)}\dd t\leq \int_0^T ||\vv^\tau_t||_{L^2(\varrho^\tau_t)}\dd t\\
&\leq& T^{1/2}\int_0^T ||\vv^\tau_t||^2_{L^2(\varrho^\tau_t)}\dd t=T^{1/2}\sum_k \tau \left(\frac{W_2(\varrho^\tau_{(k+1)},\varrho^\tau_{(k)})}{\tau}\right)^2\leq C.
\end{eqnarray*}
The estimate on $\widetilde E^\tau$ is completely analogous
$$|\widetilde E^\tau|([0,T]\times\Omega)=\int_0^T \dd t \int_\Omega |\widetilde \vv^\tau_t|\dd\widetilde\varrho^\tau_t\leq T^{1/2}\int_0^T ||\widetilde \vv^\tau_t||^2_{L^2(\widetilde\varrho^\tau_t)}\dd t
=T^{1/2}\sum_k \tau \left(\frac{W_2(\varrho^\tau_{(k+1)},\varrho^\tau_{(k)})}{\tau}\right)^2\leq C.$$
This gives compactness of $E^\tau$ and $\widetilde E^\tau$ in the space of vector measures on space-time, for the weak convergence. As far as  $\widetilde\varrho^\tau$ is concerned, we can obtain more than that. Consider the following estimate, for $s<t$
$$W_2(\widetilde\varrho^\tau_t,\widetilde\varrho^\tau_s)\leq \int_s^t |(\widetilde\varrho^\tau)'|(r)\dd r\leq (t-s)^{1/2}\left(\int_s^t  |(\widetilde\varrho^\tau)'|(r)^2\dd r\right)^{1/2}.$$
From the previous computations, we have again 
$$\int_0^T |(\widetilde\varrho^\tau)'|(r)^2\dd r=\sum_k \tau \left(\frac{W_2(\varrho^\tau_{(k+1)},\varrho^\tau_{(k)})}{\tau}\right)^2\leq C,$$
and this implies 
\begin{equation}\label{holder}
W_2(\widetilde\varrho^\tau_t,\widetilde\varrho^\tau_s)\leq C (t-s)^{1/2},
\end{equation}
which means that the curves $\widetilde\varrho^\tau$ are uniformly H\"older continuous. Since they are defined on $[0,T]$ and valued in $\WW_2(\Omega)$ which is compact, we can apply the Ascoli Arzel\`a Theorem. This implies that, up to subsequences, we have
$$
E^\tau\deb E\,\mbox{ in }\mathcal{M}^d([0,T]\times\Omega),\;\widetilde E^\tau\deb \widetilde E\,\mbox{ in }\mathcal{M}^d([0,T]\times\Omega);\quad
\widetilde\varrho^\tau\to\varrho\mbox{ uniformly for the $W_2$ distance}.
$$
The limit curve $\varrho$, from the uniform bounds on $\widetilde\varrho^\tau$, is both $\frac 12$-H\"older continuous and absolutely continuous in $\mathbb W_2$.
As far as the curves $\varrho^\tau$ are concerned, they also converge uniformly to the same curve $\varrho$, since $W_2(\varrho^\tau_t,\widetilde\varrho^\tau_t)\leq C\sqrt{\tau}$ (a consequence of \eqref{holder}, of the fact that $\widetilde\varrho^\tau=\varrho^\tau$ on the points of the form $k\tau$ and of the fact that $\varrho^\tau$ is constant on each interval $]k\tau,(k+1)\tau]$).

Let us now prove that $\widetilde E = E$.

\begin{lemma}
Suppose that we have two families of vector measures $E^\tau$ and $\widetilde E^\tau$ such that 
\begin{itemize}
\item $\widetilde E^\tau= \widetilde\varrho^\tau \widetilde \vv^\tau$; $E^\tau=\varrho^\tau \vv^\tau$;
\item $\widetilde \vv^{\tau}_t=\vv^\tau_t\circ\left( (k\tau-t) \vv^\tau_{(k)} +\id \right)^{-1}$; $\widetilde\varrho^\tau=\left( (k\tau-t) \vv^\tau_{(k)} +\id \right)_\#\varrho^\tau$;
\item $\int\!\!\int |\vv^\tau|^2\dd\varrho^\tau\leq C$ (with $C$ independent of $\tau$);
\item $E^\tau\deb E$ and $\widetilde E^\tau\deb \widetilde E$ as $\tau\to 0$
\end{itemize}
Then  $\widetilde E = E$.
\end{lemma}
\begin{proof}
It is sufficient to fix a Lipschitz function $f:[0,T]\times\Omega\to\R^d$ and to prove $\int f\cdot \dd E=\int f\cdot \dd\widetilde E$. To do that, we write 
$$\int f\cdot \dd\widetilde E^\tau=\int_0^T\dd t\int_\Omega f\cdot \widetilde \vv^{\tau}_t \dd\widetilde\varrho^\tau=\int_0^T\dd t\int_\Omega f\circ\left( (k\tau-t) \vv^\tau +\id \right)\cdot \vv^{\tau}_t \dd\varrho^\tau,$$
which implies
$$
\left|\int\! f\cdot \dd\widetilde E^\tau-\!\int \!f\cdot \dd E^\tau\right|\leq \int_0^T\! \dd t\int_\Omega \left|f\circ\left( (k\tau-t) \vv^\tau\! \!+\!\id \right)-f\right| |\vv^{\tau}_t| \dd\varrho^\tau
\leq \Lip(f)\tau\int_0^T\int_\Omega |\vv^{\tau}_t|^2 \dd\varrho^\tau\leq C\tau.$$
This estimate proves that the limit of $\int f\cdot \dd \widetilde E^\tau$ and $\int f\cdot \dd E^\tau$ is the same, i.e. $E=\widetilde E$.
\end{proof}

\paragraph{{\bf Relation between $\varrho$ and $E$}}

We can obtain the following:
\begin{proposition}
The pair $(\varrho,E)$ satisfies, in distributional sense 
$$\partial_t\varrho+\nabla\cdot E=0, \quad E=-\nabla\varrho-\varrho\nabla V,$$
with no-flux boundary conditions on $\partial\Omega$.
In particular we have found a solution to\index{heat equation}
$$\begin{cases}\partial_t\varrho-\Delta\varrho-\nabla\cdot (\varrho\nabla V)=0,\\
(\nabla\varrho+\varrho\nabla V)\cdot \mathbf{n}=0,\\
			\varrho(0)=\varrho_0,\end{cases}$$		
where the initial datum is to be intended in the following sense: the curve $t\mapsto \varrho_t$ is (absolutely) continuous in $\WW_2$, and its initial value is $\varrho_0$.		
\end{proposition}
\begin{proof}
First, consider the weak convergence $(\widetilde\varrho^\tau,\widetilde E^\tau)\deb (\varrho,E)$ (which is a consequence of $\widetilde E=E$). The continuity equation $\partial_t\widetilde\varrho^\tau+\nabla\cdot \widetilde E^\tau=0$ is satisfied by construction in the sense of distributions, and this passes to the limit. 
Hence, $\partial_t\varrho+\nabla\cdot E=0$.
The continuity in $\WW_2$ and the initial datum pass to the limit because of the uniform $C^{0,1/2}$ bound in \eqref{holder}. 

Then, use the convergence $(\varrho^\tau,E^\tau)\deb (\varrho,E)$. Actually, using the optimality conditions of  Proposition \ref{opticond} and the definition of $E^\tau=\vv^\tau\varrho^\tau$, we have, for each $\tau>0$, $E^\tau=-\nabla\varrho^\tau-\varrho^\tau \nabla V$ (note that $\varrho^\tau$ is Lipschitz continuous for fixed $\tau>0$, which allows to write this equality, exploiting in particular $\nabla \varrho^\tau= \varrho^\tau\nabla (\log(\varrho^\tau))$, which would have no meaning for less regular measures; this regularity could be, of course, lost in the limit $\tau\to 0$). It is not difficult to pass this condition to the limit either (but now $\nabla\varrho$ will have to be intended in the sense of distributions). Take $f\in C^1_c(]0,T[\times \Omega;\R^d)$ and test:
$$\int f\cdot \dd E^\tau=-\int f\cdot \nabla\varrho^\tau-\int f\cdot  \nabla V\varrho^\tau=\int \nabla\cdot f \dd\varrho^\tau-\int f\cdot  \nabla V\varrho^\tau.$$
These terms pass to the limit as $\varrho^\tau\deb \varrho$, using the assumption $V\in C^1$, since all the test functions above are continuous. This gives $\int f\cdot \dd E=\int (\nabla\cdot f )\dd\varrho-\int f\cdot  \nabla V\,\dd\varrho,$ which implies $E=-\nabla\varrho-\varrho \nabla V$.
%
\end{proof}

\subsection{Other gradient-flow PDEs}\label{8.4.2}

We saw in the previous section the example, and a detailed analysis, of the Fokker-Planck equation. The main reason to choose such example is its simplicity, because it is a linear equation. This allows easily to pass to the limit all the terms in the relation between $E$ and $\varrho$. Yet, many other important equations can be obtained as gradient flows in $\WW_2$, choosing other functionals. We will see some of them here, without entering into details of the proof. We stress that handling non-linear terms is often difficult and requires ad-hoc estimates. We will discuss some of the ideas that one should apply. Note on the other hand that, if one uses the abstract theory of \cite{AmbGigSav}, there is no need to produce these ad-hoc estimates: after developing a general (and hard) theory for general metric spaces, the second part of  \cite{AmbGigSav} explains which are the curves that one finds as gradient flows in $\WW_2$, with the relation between the velocity field $\vv$ and the derivatives (with an ad-hoc notion of subdifferential in the Wasserstein space) of the functional $F$. This automatically gives the desired result as a part of a larger theory. 

We will discuss five classes of PDEs in this section: the {\it porous media equation}, the {\it Keller-Segel equation}, more general {\it diffusion, advection and aggregation} equations, a model for {\it crowd motion} with density constraints, and the {\it flow of the squared sliced Wasserstein distance $\mathrm{S}W_2^2$}.

\paragraph{{\bf Porous Media Equation}}\index{porous media}\index{fast diffusion}This equation models the diffusion of a substance into a material whose properties are different than the void, and which slows down the diffusion. If one considers the case of particles which are advected by a potential and subject to this kind of diffusion, the PDE reads
$$\partial_t \varrho-\Delta(\varrho^m)-\nabla\cdot(\varrho\nabla V)=0,$$
for an exponent $m>1$. One can formally check that this is the equation of the gradient flow of the energy
$$F(\varrho)=\frac{1}{m-1}\int \varrho^m(x)\,\dd x+\int V(x)\varrho(x)\,\dd x$$
(set to $+\infty$ for $\varrho\notin L^m$).
Indeed, the first variation of the first part of the functional is $\frac{m}{m-1}\varrho^{m-1}$, and $\varrho\nabla\left(\frac{m}{m-1}\varrho^{m-1}\right)=m\varrho\cdot\varrho^{m-2}\nabla\varrho=\nabla(\varrho^m)$. 

Note that, in the discrete step $\min_\varrho\; F(\varrho)+\frac{W_2^2(\varrho,\varrho_0)}{2\tau}$, the solution $\varrho$ satisfies 
$$\begin{cases}\frac{m}{m-1}\varrho^{m-1}+V+\frac{\varphi}{\tau}=C&\varrho-\mbox{a.e.}\\ \frac{m}{m-1}\varrho^{m-1}+V+\frac{\varphi}{\tau}\geq C&\mbox{on }\{\varrho=0\}.\end{cases}$$
This allows to express $\varrho^{m-1}=\frac{m-1}{m}(C-V-\varphi/\tau)_+$ and implies that $\varrho$ is compactly supported if $\varrho_0$ is compactly supported, as soon as $V$ has some growth conditions. This fact contradicts the usual infinite propagation speed that one finds in linear diffusion models (Heat and Fokker-Planck equation).

The above analysis works in the case $m>1$: the fact that the usual Fokker-Planck equation can be obtained for $m\to 1$ can be seen in the following way: nothing changes if we define $F$ via $F(\varrho)=\frac{1}{m-1}\int (\varrho^m(x)-\varrho(x))\,\dd x+\int V(x)\varrho(x)\,\dd x,$ since the mass $\int \varrho(x)\,\dd x=1$ is a given constant. Yet, then it is is easy to guess the limit, since $$\lim_{m\to 1}\frac{\varrho^m-\varrho}{m-1}=\varrho\log\varrho,$$
which provides the entropy that we already used for the Heat equation.

It is also interesting to consider the case $m<1$: the function $\varrho^m-\varrho$ is no longer convex, but it is concave and the negative coefficient $1/(m-1)$ makes it a convex function. Unfortunately, it is not superlinear at infinity, which makes it more difficult to handle. But for $m\geq 1-1/d$ the functional $F$ is still displacement convex. The PDE that we get as a gradient flow is called {\it Fast diffusion} equation, and it has different (and opposite) properties in terms of diffusion rate than the porous media one.

From a technical point of view, proving compactness of the minimizing movement scheme for these equations is not very easy, since one needs to pass to the limit the non-linear term $\Delta(\varrho^m)$, which means proving strong convergence on $\varrho$ instead of weak convergence. The main ingredient is a sort of $H^1$ bound in space, which comes from the fact that we have 
\begin{equation}\label{H^1 bound porous}
\int_0^T\int_\Omega |\nabla (\varrho^{m-1/2})|^2\,\dd x\dd t\approx \int_0^T\int_\Omega \varrho\frac{|\nabla\varphi|^2}{\tau^2}\,\dd x\dd t=\int_0^T |\varrho'|(t)^2\dd t\leq C
\end{equation}
 (but one has to deal with the fact that this is not a full $H^1$ bound, and the behavior in time has to be controlled, via some variants of the classical Aubin-Lions Lemma, \cite{Aubin}).

\paragraph{{\bf Diffusion, advection and aggregation}} Consider a more general case where the movement is advected by a potential determined by the superposition of many potentials, each created by one particle. For instance, given a function $W:\R^d\to\R$, the particle located at $x$ produces a potential $W(\cdot-x)$ and, globally, the potential is given by $V(y)=\int W(y-x)\dd\varrho(x)$, i.e. $V=W*\varrho$. The equation, if every particle follows $-\nabla V$ is 
$$\partial_t\varrho-\nabla\cdot(\varrho \left((\nabla W)*\varrho\right))=0,$$
where we used $\nabla(W*\varrho)=(\nabla W)*\varrho$. If $W$ is even (i.e. the interaction between $x$ and $y$ is the same as between $y$ and $x$), then this is the gradient flow of the functional
$$F(\varrho)=\frac 12 \int\!\!\int W(x-y)\dd\varrho(x)\dd\varrho(y).$$

When $W$ is convex, for instance in the quadratic case $\int\!\!\int |x-y|^2\dd\varrho(x)\dd\varrho(y)$, this gives raise to a general aggregation behavior of the particles, and as $t\to\infty$ one expects $\varrho_t\deb\delta_{x_0}$ (the point $x_0$ depending on the initial datum $\varrho_0$: in the quadratic example above it is the barycenter of $\varrho_0$).  If $W$ is not smooth enough, the aggregation into a unique point can also occur in finite time, see \cite{figalli exo}.

Note that these equations are both non-linear (the term in the divergence is quadratic in $\varrho$) and non-local. It is rare to see these non-local aggregation terms alone in the equation, as they are often coupled with diffusion or other terms. This is why we do not provide specific references except \cite{figalli exo}. We also note that from the technical point of view this kind of nonlinearity is much more compact than the previous ones, since the convolution operator transforms weak convergence into strong one, provided $W$ is regular enough.

Most often, the  above aggregation energy is studied together with an internal energy and a confining potential energy, using the functional
$$F(\varrho)=\int f(\varrho(x))\,\dd x+\int V(x)\,\dd\varrho(x)+\frac 12 \int\!\!\int W(x-y)\dd\varrho(x)\dd\varrho(y).$$
This gives the equation
$$\partial_t\varrho-\nabla\cdot\left(\varrho \left[\nabla(f'(\varrho))+\nabla V+(\nabla W)*\varrho\right]\right)=0.$$
Among the mathematical interest for this family of equations, we stress that they are those where more results (in termes of stability, and convergence to equilibrium) can be proven, due to the fact that conditions to guarantee that $F$ is displacement convex are well-known (Section \ref{geodconvW2}). See in particular \cite{CarMcCvillani06003,CarMcCvillani06006} for physical considerations and convergence results on this equation.

\paragraph{{\bf Keller-Segel}}

An interesting model in mathematical biology (see \cite{KelSeg,KelSeg2} for the original modeling) is the following: a population $\varrho$ of bacteria evolves in time, following diffusion and advection by a potential. The potential is given by the concentration $u$ of a chemo-attractant nutrient substance, produced by the bacteria themselves. This kind of phenomenon is also known under the name of {\it chemotaxis}. More precisely, bacteria move (with diffusion) in the direction where they find more nutrient, i.e. in the direction of $\nabla u$, where the distribution of $u$ depends on their density $\varrho$. The easiest model uses linear diffusion and supposes that the distribution of $u$ is related to $\varrho$ by the condition
$-\Delta u =\varrho$, with Dirichlet boundary conditions $u=0$ on $\partial\Omega$.
This gives the system 
$$\begin{cases} \partial_t\varrho+\alpha\nabla\cdot(\varrho\nabla u)-\Delta\varrho=0,\\
			-\Delta u =\varrho,\\
			u=0\mbox{ on }\partial\Omega, \,\varrho(0,\cdot)=\varrho_0,\,\partial_\nn\varrho-\varrho\partial_\nn u=0\mbox{ on }\partial\Omega.\end{cases}$$
The parameter $\alpha$ stands for the attraction intensity of bacteria towards the chemo-attractant. By scaling, instead of using probability measures $\varrho\in\pical(\Omega)$ one can set $\alpha=1$ and play on the mass of $\varrho$ (indeed, the non-linearity is only in the term $\varrho\nabla u$, which is quadratic in $\varrho$).  

Alternative equations can be considered for $u$, such as $-\Delta u+u =\varrho$ with Neumann boundary conditions. On the contrary, the boundary conditions on $\varrho$ must be of no-flux type, to guarantee conservation of the mass. This system can also be set in the whole space, with suitable decay conditions at infinity. Note also that often the PDE condition defining $u$ as the solution of a Poisson equation is replaced, when $\Omega=\R^2$, by the explicit formula
\begin{equation}\label{KSlog}
u(x)=-\frac{1}{2\pi}\int_{\R^2} \log(|x-y|)\varrho(y)\,\dd y.
\end{equation}
There is some confusion in higher dimension, as the very same formula does not hold for the Poisson equation (the logarithmic kernel should indeed be replaced by the corresponding Green function), and there two alternatives: either keep the fact that $u$ solves $-\Delta u=\varrho$, or the fact that it derives from $\varrho$ through \eqref{KSlog}. 

One can see that this equation is the gradient flow of the functional
$$F(\varrho)=\int_\Omega \varrho\log\varrho-\frac 12 \int_\Omega|\nabla u_\varrho|^2,\quad\mbox{where $u_\varrho\in H^1_0(\Omega)$ solves }-\Delta u_\varrho=\varrho.$$
Indeed, the only non-standard computation is that of the first variation of the Dirichlet term $- \frac 12 \int|\nabla u_\varrho|^2$. Suppose $\varrho_\ve=\varrho+\ve\chi$ and set $u_{\varrho+\ve\chi}=u_\varrho+\ve u_\chi$. Then
$$\frac{d}{d\ve}\left(- \frac 12 \int|\nabla u_{\varrho+\ve\chi}|^2\right)_{|\ve=0}=-\int\nabla u_\varrho\cdot\nabla u_\chi=\int u_\varrho\Delta u_\chi=-\int u_\varrho \chi.$$
It is interesting to note that this Dirichlet term is indeed (up to the coefficient $-1/2$) the square of the $H^{-1}$ norm of $\varrho$, since $||u||_{H^1_0}=||\nabla u||_{L^2}=||\varrho||_{H^{-1}}$. We will call it the $H^{-1}$ term.\index{$H^{-1}$ norm}

It is also possible to replace linear diffusion with non-linear diffusion of porous media type, replacing the entropy $\int\varrho\log\varrho$ with a power-like energy $\int \varrho^m$.

Note that the variational problem $\min F(\varrho)+\frac{W_2^2(\varrho,\varrho_0)}{2\tau}$ requires some assumption to admit existence of minimizers, as unfortunately the Dirichlet term has the wrong sign. In particular, it would be possible that the infimum is $-\infty$, or that the energy is not l.s.c. because of the negative sign.

When we use non-linear diffusion with $m>2$ the existence of a solution is quite easy. Sophisticated functional inequalities allow to handle smaller exponents, and even the linear diffusion case in dimension 2, provided $\alpha\leq 8\pi$. We refer to \cite{BlaCalCarr} and to the references therein for details on the analysis of this equation. Some of the technical difficulties are similar to those of the porous media equation, when passing to the limit non-linear terms. In  \cite{BlaCalCarr}, the $H^{-1}$ term is treated in terms of its logarithmic kernel, and ad-hoc variables symmetrization tricks are used. Note however that the nonlinear diffusion case is easier, as $L^m$ bounds on $\varrho$ translate into $W^{2,m}$ bounds on $u$, and hence strong compactness for $\nabla u$.

We also remark that the above model, coupling a parabolic equation on $\varrho$ and an elliptic one on $u$, implicitly assumes that the configuration of the chemo-attractant instantaneously follows that of $\varrho$. More sophisticated models can be expressed in terms of the so-called {\it parabolic-parabolic Keller-Segel} equation, in the form 
$$\begin{cases} \partial_t\varrho+\alpha\nabla\cdot(\varrho\nabla u)-\Delta\varrho=0,\\
			\partial_tu-\Delta u =\varrho,\\
			u=0\mbox{ on }\partial\Omega, \,\varrho(0,\cdot)=\varrho_0,\,\partial_\nn\varrho-\varrho\partial_\nn u=0\mbox{ on }\partial\Omega.\end{cases}$$
or other variants with different boundary conditions. This equation can also be studied as a gradient flow in two variables, using the distance $W_2$ on $\varrho$ and the $L^2$ distance on $u$; see \cite{BCKKLL}.

			\paragraph{{\bf Crowd motion}}
			
			 The theory of Wasserstein gradient flows has interestingly been applied to the study of a continuous model of crowd movement under density constraints.

Let us explain the modeling, starting from the discrete case. Suppose that we have a population of particles such that each of them, if alone, would follow its own velocity $\uu$ (which could a priori depend on time, position, on the particle itself\dots).  Yet, these particles are modeled by rigid disks that cannot overlap, hence, the actual velocity cannot always be $\uu$, in particular if $\uu$ tends to concentrate the masses. We will call $\vv$ the actual velocity of each particle, and the main assumption of the model is that $\vv=P_{\mathrm{adm}(q)}(\uu)$, where $q$ is the particle configuration, $\mathrm{adm}(q)$ is the set of velocities that do not induce (for an infinitesimal time) overlapping starting from the configuration $q$, and $P_{\mathrm{adm}(q)}$ is the projection on this set.

The simplest example is the one where every particle is a disk with the same radius $R$ and center located at $q_i$. In this case we define the admissible set of configurations $K$ through
$$ K:=\{q=(q_i)_i\in\Omega^N\;:\;|q_i-q_j|\geq 2R\;\mbox{ for all }i\neq j\}.$$
In this way the set of admissible velocities is easily seen to be
$$\mathrm{adm}(q)=\{\vv=(\vv_i)_i\;:\; (\vv_i-\vv_j)\cdot (q_i-q_j)\geq 0\,\mbox{ for all $(i,j)$ with }|q_i-q_j|=2R \}.$$
The evolution equation which has to be solved to follow the motion of $q$ is then 
\begin{equation}\label{q'}
q'(t)=P_{\mathrm{adm}(q(t))}\uu(t)
\end{equation}
(with $q(0)$ given).
Equation \eqref{q'}, not easy from a mathematical point of view, was studied by Maury and Venel in \cite{crowd1,crowd2}. 

We are now interested in the simplest continuous counterpart of this microscopic model (without pretending that it is any kind of homogenized limit of the discrete case, but only an easy re-formulation in a density setting). In this case the particles population will be described by a probability density $\varrho\in\pical(\Omega)$, the constraint becomes a density constraint $\varrho\leq 1$ (we define the set $K=\{\varrho\in\pical(\Omega)\,:\,\varrho\leq 1\}$\index{$L^\infty$ density constraints}), the set of admissible velocities will be described by the sign of the divergence on the saturated region $\{\varrho=1\}$: $\mathrm{adm}(\varrho)=\big\{ \vv:\Omega\to\R^d\;:\;\nabla\cdot \vv\geq 0\;\mbox{ on }\{\varrho=1\}\big\}$;
 we will consider a projection $P$, which will be either  the projection in $L^2(\lcal^d)$ or in $L^2(\varrho)$ (this will turn out to be the same, since the only relevant zone is $\{\varrho=1\}$). Finally, we
solve the equation 
  \begin{equation}\label{cont con proj}
\partial_t\varrho_t+\nabla\cdot\big( \varrho_t \,\big(P_{\mathrm{adm}(\varrho_t)} \uu_t\big)\big)=0.
  \end{equation}
    
  The main difficulty is the fact that the vector field $\vv=P_{\mathrm{adm}(\varrho_t)}\uu_t$ is neither regular (since it is obtained as an $L^2$ projection, and may only be expected to be $L^2$ a priori), nor it depends regularly on $\varrho$ (it is very sensitive to small changes in the values of $\varrho$:  passing from a density $1$ to a density $1-\ve$ completely modifies the saturated zone, and hence the admissible set of velocities and the projection onto it).
  
  In \cite{MauRouSan} these difficulties have been overpassed in the case $\uu=-\nabla D$ (where $D:\Omega\to\R$ is a given Lipschitz function) and the existence of a solution (with numerical simulations) is proven via a gradient flow method. Indeed, \eqref{cont con proj} turns out to be the gradient flow in $\WW_2$ of the energy
  $$F(\varrho)=\begin{cases}\int D\,\dd\varrho&\mbox{if }\varrho\in K;\\
  					+\infty&\mbox{if }\varrho\notin K.\end{cases}$$
  
We do not enter into the details of the study of this equation, but we just make a little bit more precise the definitions above. Actually, instead of considering the divergence of vector fields which are only supposed to be $L^2$, it is more convenient to give a better description of $\mathrm{adm}(\varrho)$ by duality: 
  $$\mathrm{adm}(\varrho)=\left\{ \vv\in L^2(\varrho)\;:\;\int \vv\cdot\nabla p\leq 0\quad\forall p\in H^1(\Omega)\,:\,p\geq 0,\, p(1-\varrho)=0\right\}.$$
  
 In this way we characterize $\vv=P_{\mathrm{adm}(\varrho)}(u)$ through 
   \begin{gather*}
  \uu=\vv+\nabla p,\quad \vv\in \mathrm{adm}(\varrho),\quad \int \vv\cdot\nabla p = 0,\\
  p\in \mathrm{press}(\varrho):=\{p\in H^1(\Omega),\, p\geq 0,\, p(1-\varrho)=0\},
  \end{gather*}
where $\mathrm{press}(\varrho)$ is the space of functions $p$ used as test functions in the dual definition of $\mathrm{adm}(\varrho)$, which play the role of a pressure affecting the movement. The two cones $\nabla \mathrm{press}(\varrho)$ (defined as the set of gradients of elements of $\mathrm{press}(\varrho)$) and $\mathrm{adm}(\varrho)$ are in duality for the $L^2$ scalar product (i.e. one is defined as the set of vectors which make a negative scalar product  with all the elements of the other). This allows for an orthogonal decomposition $\uu_t=\vv_t+\nabla p_t$, and gives the alternative expression of Equation \eqref{cont con proj}, i.e.\index{pressure}
  \begin{equation}\label{con press}
  \begin{cases}\partial_t\varrho_t+\nabla\cdot\big( \varrho_t (\uu_t-\nabla p_t)\big)=0,\\
  0\leq \varrho\leq 1,\,p\geq 0,\,p(1-\varrho)=0.\end{cases}
  \end{equation}
  
  More details can be found in \cite{MauRouSan,aude phd,MauRouSanVen}. In particular, in \cite{MauRouSan} it is explained how to handle the nonlinearities when passing to the limit. Two sources of nonlinearity are observed: the term $\varrho\nabla p$ is easy to consider, since it is actually equal to $\nabla p$ (as we have $p=0$ on $\{\varrho\neq 1\}$); on the other hand, we need to deal with the equality $p(1-\varrho)=0$ and pass it to the limit. This is done by obtaining strong compactness on $p$, from a bound on $\int_0^T\int_\Omega |\nabla p|^2$, similarly to \eqref{H^1 bound porous}. It is important to observe that transforming $\varrho\nabla p$ into $\nabla p$ is only possible in the case of one only phase $\varrho$, while the multi-phasic case presents extra-difficulties (see, for instance \cite{CanGalMon}). These difficulties can be overpassed in the presence of diffusion, which provides extra $H^1$ bounds and hence strong convergence for $\varrho$ (for crowd motion with diffusion, see \cite{MesSan}). Finally, the uniqueness question is a tricky one: if the potential $D$ is $\lambda$-convex then this is easy, but considerations from the DiPerna-Lions theory (\cite{DiPLio,Amb-BV}) suggest that uniqueness should be true under weaker assumptions; nevertheless this is not really proven. On the other hand, in \cite{DiMMes} uniqueness is proven under very mild assumptions on the drift (or, equivalently, on the potential), provided some diffusion is added.
  
  \paragraph{{\bf Sliced Wasserstein distance}}
   
 Inspired by a construction proposed in \cite{PitKokDah} for image impainting, M. Bernot proposed around 2008 an interesting scheme to build a transport map between two given measures which, if not optimal, at least had some monotoncity properties obtained in an isotropic way.

Consider two measures $\mu,\nu\in\pical(\R^d)$ and project them onto any one-dimen\-sional direction. For every $e\in \mathbb S^{d\!-\!1}$ (the unit sphere of $\R^d$), we take the map $\pi_e:\R^d\to\R$ given by $\pi_e(x)=x\cdot e$ and look at the image measures $(\pi_e)_\#\mu$ and $(\pi_e)_\#\nu$. They are measures on the real line, and we call $T_e:\R\to\R$ the monotone (optimal) transport between them. The idea is that, as far as the direction $e$ is concerned, every point $x$ of $\R^d$ should be displaced of a vector $\vv_e(x):=(T_e(\pi_e(x))-\pi_e(x))e$. To do a global displacement, consider $\vv(x)=\fint_{\mathbb S^{d\!-\!1}}\vv_e(x)\,\dd\haus^{d\!-\!1}(e)$, where $\haus^{d\!-\!1}$ is the uniform Hausdorff measure on the sphere.

There is no reason to think that $\id+\vv$ is a transport map from $\mu$ to $\nu$, and indeed in general it is not. But if one fixes a small time step $\tau>0$ and uses a displacement $\tau \vv$ getting a measure $\mu_1=(\id+\tau \vv)_\#\mu$, then it is possible to iterate the construction. One expects to build in this way a sequence of measures $\mu_n$ that converges to $\nu$, but this has not yet been rigorously proven. From the empirical point of view, the transport maps that are obtained in this way are quite satisfactory, and have been tested in particular in the discrete case (a finite number of Dirac masses with equal mass, i.e. the so-called assignment problem). 

For the above construction, discrete in time and used for applications to images, there are not many references (essentially, the only one is \cite{RabPeyDelBer}, see also \cite{Bonnotte} for a wider discussion). A natural continuous counterpart exists: simply consider, for every absolutely continuous measure $\varrho\in\pical(\R^d)$, the vector field $\vv=\vv_{(\varrho)}$ that we defined above as a velocity field depending on $\varrho$ (absolute continuity is just required to avoid atoms in the projections). Then, we solve the equation
  \begin{equation}\label{flowbernot}
  \partial_t\varrho_t+\nabla\cdot(\varrho_t \vv_{(\varrho_t)})=0.
  \end{equation}

It happens that this equation has a gradient flow structure: it is indeed the gradient flow in $\WW_2$ of a functional related to a new distance on probability measures, induced by this construction.

This distance can be defined as follows (see \cite{RabPeyDelBer}): given two measures $\mu,\nu\in\pical_2(\R^d)$, we define
$$\mathrm{S}W_2(\mu,\nu):=\left(\fint_{\mathbb S^{d\!-\!1}} W_2^2((\pi_e)_\#\mu,(\pi_e)_\#\nu)\,\dd\haus^{d\!-\!1}(e)\right)^{1/2}.$$
This quantity could have been called ``projected Wasserstein distance'' (as it is based on the behavior through projections), but since in \cite{RabPeyDelBer} it is rather called ``sliced Wasserstein distance'', we prefer to keep the same terminology.  

The fact that $\mathrm{S}W_2$ is a distance comes from $W_2$ being a distance. The triangle inequality may be proven using the triangle inequality for $W_2$ (see Section \ref{5.1}) and for the $L^2$ norm. Positivity and symmetry are evident. The equality $\mathrm{S}W_2(\mu,\nu)=0$ implies $W_2^2((\pi_e)_\#\mu,(\pi_e)_\#\nu)$ for all $e\in \mathbb S^{d\!-\!1}$. This means $(\pi_e)_\#\mu=(\pi_e)_\#\nu)$ for all $e$ and it suffices to prove $\mu=\nu$.

It is evident from its definition, and from the fact that the maps $\pi_e$ are $1$-Lipschitz (which implies $W_2^2((\pi_e)_\#\mu,(\pi_e)_\#\nu)\leq W_2^2(\mu,\nu)$) that we have $\mathrm{S}W_2(\mu,\nu)\leq W_2(\mu,\nu)$. Moreover, the two distances also induce the same topology, at least on compact sets. Indeed, the identity map from $\WW_2$ to $(\pical(\Omega),\mathrm{S}W_2)$ is continuous (because of $\mathrm{S}W_2\leq W_2$) and bijective. Since the space where it is defined is compact, it is also a homeomorphism. 
One can also prove more, i.e. an inequality of the form $W_2\leq C\mathrm{S}W_2^\beta$ for a suitable exponent $\beta\in]0,1[$. Chapter 5 in \cite{Bonnotte} proves this inequality with $\beta=(2(d+1))^{-1}$.

The interest in the use of this distance is the fact that one has a distance on $\pical(\Omega)$ with very similar qualitative properties as $W_2$, but much easier to compute, since it only depends on one-dimensional computations (obviously, the integral over $e\in \mathbb S^{d\!-\!1}$ is discretized in practice, and becomes an average over a large number of directions). We remark anyway an important difference between $W_2$ and $\mathrm{S}W_2$: the latter is not a geodesic distance. On the contrary, the geodesic distance associated to $\mathrm{S}W_2$ (i.e. the minimal lenght to connect two measures) is exactly $W_2$.
 
If we come back to gradient flows, it is not difficult to see that the equation \eqref{flowbernot} is the gradient flow of $F(\varrho):=\frac 12\mathrm{S}W_2^2(\varrho,\nu)$ and can be studied as such. Existence and estimates on the solution of this equation are proven in \cite{Bonnotte}, and the nonlinearity of $\vv_{(\varrho)}$ is quite easy to deal with. On the other hand, many, natural and useful, questions are still open: is it true that $\varrho_t\deb\nu$ as $t\to\infty$? Can we define (at least under regularity assumptions on the initial data) the flow of the vector field $\vv_{(\varrho_t)}$, and what is the limit of this flow as $t\to\infty$? The idea is that it should be a transport map between $\varrho_0$ and $\nu$ and, if not the optimal transport map, at least a ``good'' one, but most of the related questions are open.
 			
\subsection{Dirichlet boundary conditions}\label{8.4.3}\index{Dirichlet conditions}\index{heat equation}

For sure, the attentive reader has already noted that all the equations that have been identified as gradient flows for the distance $W_2$ on a bounded domain $\Omega$ are always accompanied by Neumann boundary conditions. This should not be surprising. Wasserstein distances express the movement of masses when passing from a configuration to another, and the equation represents the conservation of mass. It means that we are describing the movement of a collection $\varrho$ of particles, bound to stay inside a given domain $\Omega$, and selecting their individual velocity $\vv$ in way which is linked to the global value of a certain functional $F(\varrho)$. It is natural in this case to have boundary conditions which write down the fact that particles do not exit the domain, and the pointwise value of the density $\varrho$ on $\partial\Omega$ is not particularly relevant in this analysis. Note that ``do not exit'' does not mean ``those on the boundary stay on the boundary'', which is what happens when solutions are smooth and the velocity field $\vv$ satisfies $\vv\cdot\nn=0$. Yet, the correct Neumann condition here is rather $\varrho \vv \cdot\nn=0$ a.e., which means that particles could enter from $\partial\Omega$ into the interior, but immediately after it happens there will be (locally) no mass on the boundary, and the condition is not violated, hence. On the contrary, should some mass go from the interior to outside $\Omega$, then we would have a violation of the Neumann condition, since there would be (intuitively) some mass $\varrho>0$ on the boundary with velocity directed outwards.

Anyway, we see that Dirichlet conditions do not find their translation into $W_2$ gradient flows!

To cope with Dirichlet boundary conditions, Figalli and Gigli defined in \cite{FigGig calore} a sort of modified Wasserstein distance, with a special role played by the boundary $\partial\Omega$, in order to study the Heat equation $\partial_t\varrho=\Delta\varrho$ with Dirichlet b.c. $\varrho=1$ on $\partial\Omega$.
Their definition is as follows: given two finite positive measures $\mu,\nu\in\M_+(\ip\Omega)$ (not necessarily probabilities, not necessarily with the same mass), we define 
$$\Pi b(\mu,\nu)=\{\gamma\in\M_+(\overline\Omega\times\overline\Omega)\;:\;(\pi_x)_\#\gamma\res \ip\Omega=\mu,\,(\pi_y)_\#\gamma\res \ip\Omega=\nu\}.$$
Then, we set
$$Wb_2(\mu,\nu):=\sqrt{\,\inf \left\{\int_{\overline\Omega\times\overline\Omega} |x-y|^2\,\dd\gamma,\;\gamma\in \Pi b(\mu,\nu)\right\}}.$$
The index $b$ stands for the special role played by the boundary. Informally, this means that the transport from $\mu$ to $\nu$ may be done usually (with a part of $\gamma$ concentrated on $\ip\Omega\times\ip\Omega$), or by moving some mass from $\mu$ to $\partial\Omega$ (using $\gamma\res(\ip\Omega\times\partial\Omega)$), then moving from one point of the boundary to another point of the boundary (this should be done by using $\gamma\res(\partial\Omega\times\partial\Omega)$, but since this part of $\gamma$ does not appear in the constraints, then we can forget about it, and the transport is finally free on $\partial\Omega$), and finally from $\partial\Omega$ to $\nu$ (using $\gamma\res(\partial\Omega\times\ip\Omega)$). 

In \cite{FigGig calore} the authors prove that $Wb_2$ is a distance, that the space $\M_+(\ip\Omega)$ is always a geodesic space, independently of convexity or connectedness properties of $\Omega$ (differently from what happens with $\Omega$, since here the transport is allowed to ``teleport'' from one part of the boundary to another, either to pass from one connected component to another or to follow a shorter path going out of $\Omega$), and they study the gradient flow, for this distance, of the functional $F(\varrho)=\int (\varrho\log\varrho-\varrho)\,\dd x$. Note that in the usual study of the entropy on $\pical(\Omega)$ one can decide to forget the term $-\int\varrho$, which is anyway a constant because the total mass is fixed. Here this term becomes important (if the function $f(t)=t\log t-t$ is usually preferred to $t\log t$, it is because its derivative is simpler, $f'(t)=\log t$, without changing its main properties).

With this choice of the functional and of the distance, the gradient flow that Figalli and Gigli obtain is the Heat equation\index{heat equation} with the particular boundary condition  $\varrho=1$ on $\partial\Omega$. One could wonder where the constant $1$ comes from, and a reasonable explanation is the following: if the transport on the boundary is free of charge, then automatically the solution selects the value which is the most performant for the functional, i.e. the constant $t$ which minimizes $f(t)=t\log t-t$. In this way, changing the linear part and using $F(\varrho)=\int (\varrho\log\varrho-c\varrho)\,\dd x$ could change the constant on the boundary, but the constant $0$ is forbidden for the moment. It would be interesting to see how far one could go with this approach and which Dirichlet conditions and which equations could be studied in this way, but this does not seem to be done at the moment.

Moreover, the authors explain that, due to the lack of geodesic convexity of the entropy w.r.t. $Wb_2$, the standard abstract theory of gradient flows is not able to provide uniqueness results (the lack of convexity is due in some sense to the possible concentration of mass on the boundary, in a way similar to what happened in \cite{MauRouSan} when dealing with the door on $\partial\Omega$). On the other hand, standard Hilbertian results on the Heat equation can provide uniqueness for this equation, as the authors smartly remark in \cite{FigGig calore}.
 
We observe that this kind of distances with free transport on the boundary were already present in \cite{BouButSep,BouBut JEMS}, but in the case of the Wasserstein distance $W_1$, and the analysis in those papers was not made for applications to gradient flows, which are less natural to study with $p=1$. We can also point out a nice duality formula:
$$
Wb_1(\mu,\nu):=\min\left\{\int |x-y|\,\dd\gamma\,:\,\gamma\in\Pi b(\mu,\nu)\right\}
=\sup\left\{\int u\,\dd(\mu-\nu)\,:\, u\in \Lip_1(\Omega),\,u=0\,\mbox{ on }\partial\Omega\right\}. 
$$
In the special case $p=1$ and $\mu(\ip\Omega)=\nu(\ip\Omega)$, we also obtain
$$Wb_1(\mu,\nu)=\Wc(\mu,\nu),\quad\mbox{ for }c(x,y)=\min\{|x-y|,d(x,\partial\Omega)+d(y,\partial\Omega)\}.$$
The cost $c$ is a pseudo-distance on $\Omega$ where moving on the boundary is free. This kind of distance has also been used in \cite{ButOudSte} (inspired by \cite{BouBut JEMS}) to model free transport costs on other lower dimensional sets, and not only the boundary (with the goal to model, for instance, transportation networks, and optimize their shape). It is interesting to see the same kind of ideas appear for so different goals.\index{distance costs}

\subsection{Numerical methods from the JKO scheme}\label{W2num}

We present in this section two different numerical methods which have been recently proposed to tackle evolution PDEs which have the form of a gradient flow in $\mathbb W_2(\Omega)$ via their variational JKO scheme. We will only be concerned with discretization methods allowing the numerical treatment of one step of the JKO scheme, i.e. solving problems of the form
$$\min\left\{F(\varrho)+\frac12 W_2^2(\varrho,\nu)\;:\;\varrho\in\pical(\Omega)\right\},$$
for suitable $\nu$ (to be taken equal to $\varrho^\tau_k$) and suitable $F$ (including the $\tau$ factor). We will not consider the convergence as $\tau\to 0$ of the iterations to solutions of the evolution equation.

We will present two methods. One, essentially taken from \cite{BenCarLab}, is based on the Benamou-Brenier formula first introduced in \cite{BenBre} as a numerical tool for optimal transport. This method is well-suited for the case where the energy $F(\varrho)$ used in the gradient flow is a convex function of $\varrho$. For instance, it works for functionals of the form $F(\varrho)=\int f(\varrho(x))\dd x+\int V\dd \varrho$ and can be used for Fokker-Planck and porous medium equations. The second method is based on methods from semi-discrete optimal transport, essentially developed by Q. M\'erigot using computational geometry (see \cite{Mer, KitMerThi} and \cite{levy3D} for 3D implementation) and translates the problem into an optimization problem in the class of convex functions; it is well suited for the case where $F$ is geodesically convex, which means that the term $\int f(\varrho(x))\,\dd x$ is ony admissible if $f$ satisfies McCann's condition, the term $\int V\dd \varrho$ needs $V$ to be convex, but interaction terms such as $\int\int W(x-y)\,\dd \varrho(x)\,\dd \varrho(y)$ are also allowed, if $W$ is convex.

\paragraph{{\bf Augmented Lagrangian methods}}

Let us recall the basis of the Benamou-Brenier method. This amounts to solve the variational problem \eqref{BBp} which reads, in the quadratic case, as
$$\min\quad \sup_{(a,b)\in K_2} \int\int a\dd\varrho + \int\int b\cdot \dd E\;:\;\partial_t\varrho_t+\nabla\cdot E_t=0,\;\varrho_0=\mu,\,\varrho_1=\nu
$$
where  $K_2=\{(a,b)\in\R\times\R^d\;:\;a+\frac 12 |b|^2\leq 0\}$. We then use the fact that the continuity equation constraint can also be written as a sup penalization, by adding to the functional
$$\sup_{\phi\in C^1([0,1]\times \Omega)} -\int\int \partial_t\phi\dd\varrho- \int\int \nabla\phi\cdot \dd E+\int \phi_1\dd \nu-\int \phi_0\dd \mu,$$
which is $0$ in case the constraint is satisfied, and $+\infty$ if not.

It is more convenient to express everything in the space-time formalism, i.e. by writing $\nabla_{t,x}\phi$ for $( \partial_t\phi,\nabla\phi)$ and using the variable $\mm$ for $(\varrho,E)$ and $A$ for $(a,b)$. We also set $G(\phi):=\int \phi_1\dd \nu-\int \phi_0\dd \mu$. Then the problem becomes
$$\min_\mm \sup_{A,\phi} \quad \mm\cdot (A-\nabla_{t,x}\phi)-I_{K_2}(A)+G(\phi),$$
where the scalar product is here an $L^2$ scalar product, but becomes a standard Euclidean scalar product as soon as one discretizes (in time-space).  The function $I_{K_2}$ denotes the indicator function in the sense of convex analysis, i.e. $+\infty$ if the condition $A\in K_2$ is not satisfied, and $0$ otherwise.

The problem can now be seen as the search for a saddle-point of the Lagrangian 
$$L(\mm,(A,\phi)):=\mm\cdot (A-\nabla_{t,x}\phi)-I_{K_2}(A)+G(\phi),$$
 which means that we look for a pair $(\mm,(A,\phi)) $ (actually, a triple, but $A$ and $\phi$ play together the role of the second variable) where $\mm$ minimizes for fixed $(A,\phi)$ and $(A,\phi)$ maximizes for fixed $\mm$. This fits the following framework, where the variables are $X$ and $Y$ and the Lagrangian has the form $L(X,Y):=X\cdot \Lambda Y-H(Y)$. In this case one can use a very smart trick, based on the fact that the saddle points of this Lagrangan are the same of the {\it augmented Lagrangian} $\tilde L$ defined as $\tilde L(X,Y):=X\cdot \Lambda Y-H(Y)-\frac r2|\Lambda Y|^2$, whatever the value of the parameter $r>0$ is. Indeed, the saddle-point of $L$ are characterized by (we assume all the functions we minimize are convex and all the functions we maximize are concave)
$$\begin{cases} \Lambda Y=0&\mbox{(optimality of $X$)},\\
			\Lambda^t X-\nabla H(Y)=0 &\mbox{(optimality of $Y$)},\end{cases}$$
			while those of $\tilde L$ are characterized by
			
			$$\begin{cases} \Lambda Y=0&\mbox{(optimality of $X$)},\\
			\Lambda^t X-\nabla H(Y)-r\Lambda^t\Lambda Y=0 &\mbox{(optimality of $Y$)},\end{cases}$$
which is the same since the first equation implies that the extra term in the second vanishes.

In this case, we obtain a saddle point problem of the form 
$$\min_\mm \max_{A,\phi}\qquad \mm\cdot (A-\nabla_{t,x}\phi)-I_{K_2}(A)+G(\phi)-\frac{r}{2}||A-\nabla_{t,x}\phi||^2$$
(where the squared norm in the last term is an $L^2$ norm in time and space),
which is then solved by iteratively repeating three steps: for fixed $A$ and $\mm$, finding the optimal $\phi$ (which amounts to minimizing a quadratic functional in calculus of variations, i.e. solving a Poisson equation in the space-time $[0,1]\times\Omega$, with Neumann boundary conditions, homogeneous on $\partial \Omega$ and non-homogeneous on $t=0$ and $t=1$, due to the term $G$); then for fixed $\phi$ and $\mm$ find the optimal $A$ (which amounts to a pointwise minimization problem, in this case a projection on the convex set $K_2$); finally update $\mm$ by going in the direction of the gradient descent, i.e. replacing $\mm$ with $\mm-r(A-\nabla_{t,x}\phi)$ (it is convenient to choose the parameter of the gradient descent to be equal to that the Augmented Lagrangian).

This is what is done in the case where the initial and final measures are fixed. At every JKO step, one is fixed (say, $\mu$), but the other is not, and a penalization on the final $\varrho_1$ is added, of the form $\tau F(\varrho_1)$. Inspired from the considerations above, the saddle point below allows to treat the problem
$$\min_{\varrho_1}\quad \frac 12W_2^2(\varrho_1,\mu)+\int f(\varrho_1(x))\dd x+\int V\dd \varrho_1$$
by formulating it as
\begin{eqnarray*}
\min_{\mm,\varrho_1}\; \max_{A,\phi,\lambda}&\quad& \int\int \mm\cdot (A-\nabla_{t,x}\phi) + \int \varrho_1\cdot (\phi_1+\lambda+V)-\int\int I_{K_2}(A)-\int \phi_0\dd\mu-\int f^*(\lambda(x))\dd x\\&&-\frac{r}{2}\int\int|A-\nabla_{t,x}\phi|^2-\frac{r}{2}\int|\phi_1+\lambda+V|^2,
\end{eqnarray*}
where we re-inserted the integration signs to underline the difference between integrals in space-time (with $\mm, A$ and $\phi$) and in space only (with $\phi_0,\phi_1,\varrho_1,V$ and $\lambda$). The role of the variable $\lambda$ is to be dual to $\varrho_1$, which allows to express $f(\varrho_1)$ as $\sup_\lambda\, \varrho_1\lambda-f^*(\lambda)$.

To find a solution to this saddle-point problem, an iterative procedure is also used, as above. The last two steps are the update via a gradient descent of $\mm$ and $\varrho_1$, and do not require further explications. The first three steps consist in the optimization of $\phi$ (which requires the solution of a Poisson problem) and in two pointwise minimization problems in order to find $A$ (which requires a projection on $K_2$) and $\lambda$ (the minimization of $f^*(\lambda)+\frac{r}{2}|\phi_1(x)+\lambda+V(x)|^2-\varrho_1(x)\lambda$, for fixed $x$).

For the applications to gradient flows, a small time-step $\tau>0$ has to be fixed, and this scheme has to be done for each $k$, using $\mu=\varrho^\tau_k$ and setting $\varrho^\tau_{k+1}$ equal to the optimizer $\varrho_1$ and the functions $f$ and $V$ must include the scale factor $\tau$. The time-space $[0,1]\times \Omega$ has to be discretized but the evolution in time is infinitesimal (due to the small time scale $\tau$), which allows to choose a very rough time discretization. In practice, the interval $[0,1]$ is only discretized using less than 10 time steps for each $k$\dots

The interested reader can consult \cite{BenCarLab} for more details, examples and simulations.

\paragraph{{\bf Optimization among convex functions}}

It is clear that the optimization problem 
$$\min_{\varrho}\quad \frac 12W_2^2(\varrho,\mu)+F(\varrho)$$
can be formulated in terms of transport maps as
$$\min_{T:\Omega\to\Omega}\quad \frac 12 \int_{\Omega}|T(x)-x|^2\dd\mu(x)+F(T_\#\mu).$$
Also, it is possible to take advantage of Brenier's theorem which characterizes optimal transport maps as gradient of convex functions, and recast it as
$$\min_{ u\,\mbox{ convex : }\nabla u\in\Omega}\quad \frac 12 \int_{\Omega}|\nabla u(x)-x|^2\dd\mu(x)+F((\nabla u)_\#\mu).$$
It is useful to note that in the last formulation the convexity of $ u$ is not necessary to be imposed, as it would anyway come up as an optimality condition. On the other hand, very often the functional $F$ involves explicity the density of the image measure  $(\nabla u)_\#\mu$ (as it is the case for the typical example $\mathcal F$), and in this case convexity of $ u$ helps in computing this image measure. Indeed, whenever $ u$ is convex we can say that the density of $\varrho:=(\nabla u)_\#\mu$ (if $\mu$ itself is absolutely continuous, with a density that we will denote by $\varrho_0$) is given by\footnote{This same formula takes into account the possibility that $\nabla u$ could be non-injective, i.e. $ u$ non-strictly convex, in which case the value of the density could be $+\infty$ due to the determinant at the denominator which would vanish.}
$$\varrho=\frac{\varrho_0}{\det(D^2 u)}\circ(\nabla u)^{-1}.$$

Hence, we are facing a calculus of variations problem in the class of convex functions. A great difficulty to attack this class of problems is how to discretize the space of convex functions. The first natural approach would be to approximate them by piecewise linear functions over
a fixed mesh.  In this case, imposing convexity becomes a local feature, and the number of linear
constraints for convexity is
proportional to the size of the mesh. Yet, Chon\'e and Le Meur showed in \cite{ChoLeM}
that we cannot approximate in this way all convex functions, but only those satisfying some extra constraints on their Hessian (for instance, those which also have a positive mixed second derivative $\partial^2 u/\partial x_i\partial x_j$). Because of this difficulty, a different approach is needed. For instance, Ekeland and Moreno-
Bromberg used the representation of a convex function as a maximum of affine
functions \cite{EkeMor}, but this needed many more linear constraints; Oudet and M\'erigot \cite{MerOud} decided to test convexity on a grid different (and less refined) than that where the functions are defined\dots\ These methods give somehow satisfactory answers for functionals involving $ u$ and $\nabla u$, but are not able to handle terms
involving the Monge-Amp\`ere operator $\det(D^2 u)$. 

The method proposed in \cite{BenCarMerOud}, that we will roughly present here, does not really use a prescribed mesh. The idea is the following: suppose that $\mu$ is a discrete measure of atomic type, i.e. of the form $\sum_j a_j\delta_{x_j}$. A convex defined on its support $S:=\{x_j\}_j$ will be a function $ u:S\to\R$ such that at each point $x\in S$ the subdifferential
$$\partial u(x):=\{p\in\R^d\,:\,  u(x)+p\cdot(y-x)\leq  u(y)\mbox{ for all }y\in S\}$$
 is non-empty. Also, the Monge-Amp\`ere operator will be defined by using the sub-differential, and more precisely the equality
 $$\int_B \det(D^2 u(x))\dd x=|\partial u(B)|$$
 which is valid for smooth convex functions $ u$ and arbitrary open sets $B$. An important point is the fact that whenever $f$ is superlinear, functionals of the form $\int f(\varrho(x))\dd x$ impose, for their finiteness, the positiviy of $\det(D^2 u)$, which will in turn impose that the sub-differential has positive volume, i.e. it is non-empty, and hence convexity\dots\
 
 More precisely, we will minimize over the set of pairs $( u,P):S\to\R\times\R^d$ where $P(x)\in\partial u(x)$ for every $x\in S$. For every such pair $( u,P)$ we weed to define $G( u,P)$ which is meant to be $(\nabla u)_\#\mu$, and define $F(G( u,P))$ whenever $F$ has the form $F=\mathcal F+\mathcal V+\mathcal W$.
 We will simply define 
 $$\mathcal V(G( u,P)):=\sum_{j}a_j V(P(x_j))\quad \mbox{ and }\mathcal W(G( u,P)):=\sum_{j,j'}a_j a_{j'} W(P(x_j)-P(x_{j'})),$$
 which means that we just use $P_\#\mu$ instead of $(\nabla u)_\#\mu$. Unfortunately, this choice is not adapted for the functional $\mathcal F$, which requires absolutely continuous measures, and $P_\#\mu$ is atomic.
 In this case, instead of concentrating all the mass $a_j$ contained in the point $x_j\in S$ on the unique point $P(x_j)$, we need to spread it, and we will spread it uniformly on the whole subdifferential $\partial u(x_j)$. This means that we also define a new surrogate of the image measure $(\nabla u)_\#\mu$, called $G^{ac}( u,P)$ (where the superscript {\it ac} stands for {\it absolutely continuous}), given by
 $$G^{ac}( u,P):=\sum_j \frac{a_j}{|A_j|}\lcal^d\res A_j,$$
 where $A_j:=\partial u(x_j)\cap\Omega$ (the intersection with $\Omega$ is done in order to take care of the constraint $\nabla u\in\Omega$).
 Computing $\mathcal F(G^{ac}( u,P))$ gives hence
 $$\mathcal F(G^{ac}( u,P))=\sum_j  |A_j| f\left(\frac{a_j}{|A_j|}\right).$$
  It is clear that the discretization of $\mathcal V$ and $\mathcal W$ in terms of $G( u,P)$ are convex functions of $( u,P)$ (actually, of $P$, and the constraint relating $P$ and $ u$ is convex) whenever $V$ and $W$ are convex; concerning $\mathcal F$, it is possible to prove, thanks to the concavity properties of the determinant or, equivalently, to the Brunn-Minkowski inequality (see for instance \cite{Schneider-convex}) that $\mathcal F(G^{ac}( u,P))$ is convex in $ u$ as soon as $f$ satisfies McCann's condition. Globally, it is not surprising to see that we face a convex variational problem in terms of $ u$ (or of $\nabla u$) as soon as $F$ is displacement convex in $\varrho$ (actually, convexity on generalized geodesics based at $\mu$ should be the correct notion).
  
  Then we are lead to study the variational problem
\begin{equation}\label{quentin}
\min_{ u,P}\quad \frac 12 \sum_j  a_j|P(x_j)-x_j|^2+\mathcal V(G( u,P))+\mathcal W(G( u,P))+\mathcal F(G^{ac}( u,P))
\end{equation}
under the constraints $P(x_j)\in A_j:=\partial u(x_j)\cap\Omega$.
Note that we should incorporate the scale factor $\tau$ in the functional $F$ which means that, for practical purposes, convexity in $P$ is guaranteed as soon as $V$ and $W$ have second derivatives which are bounded from below (they are semi-convex) and $\tau$ is small (the quadratic term coming from $W_2^2$ will always overwhelm possible concavity of the other terms). The delicate point is how to compute the subdifferentials $\partial u(x_j)$, and optimize them (i.e. compute derivatives of the relevant quantity w.r.t. $ u$). 

This is now possible, and in a very efficient way, thanks to tools from computational geometry. Indeed, in this context, subdifferentials are exactly a particular case of what are called Laguerre cells, which in turn are very similar to Voronoi cells. We remind that, given some points $(x_j)_j$, their Voronoi cells $V_j$ are defined by
$$V_j:=\left\{x\in\Omega\,:\, \frac12|x-x_j|^2\leq \frac12|x-x_{j'}|^2\,\mbox{ for all }j'\right\}$$
 (of course the squares of the distances could be replaced by the distances themselves, but in this way it is evident that the cells $V_j$ are given by a finite number of linear inequalities, and are thus convex polyhedra; the factors $\frac12$ are also present only for cosmetic reasons). Hence, Voronoi cells are the cells of points which are closer to one given point $x_j\in S$ than to the others.
 
 In optimal transport a variant of these cells is more useful: given a set of values $\psi_j$, we look for the cells (called {\it Laguerre cells})
 $$W_j:=\left\{x\in\Omega\,:\, \frac 12|x-x_j|^2+\psi_j\leq \frac12|x-x_{j'}|^2+\psi_{j'}\,\mbox{ for all }j'\right\}.$$
This means that we look at points which are closer to $x_j$ than to the other points $x_{j'}$, up to a correction\footnote{If the points $x_j$ are the locations of some ice-cream sellers, we can think that $\psi_j$ is the price of an ice-cream at $x_j$, and the cells $W_j$ will represent the regions where customers will decide to go to the seller $j$, keeping into account both the price and the distance.} given by the values $\psi_j$. It is not difficult to see that also in this case cells are convex polyhedra. And it is also easy to see that, if $\varrho$ is an absolutely continuous measure on $\Omega$ and $\mu=\sum_j a_j\delta_{x_j}$, then finding an optimal transport map from $\varrho$ to $\mu$ is equivalent to finding values $\psi_j$ such that $\varrho(W_j)=a_j$ for every $j$ (indeed, in this case, the map sending every point of $W_j$ to $x_j$ is optimal, and $-\psi_j$ is the value of the corresponding Kantorovich potential at the point $x_j$). Finally, it can be easily seen that the Laguerre cells corresponding to $\psi_j:= u(x_j)-\frac12|x_j|^2$ are nothing but the subdifferentials of $ u$ (possibly intersected with $\Omega$). 

Handling Laguerre cells from the computer point of view has for long been difficult, but it is now state-of-the-art in computational geometry, and it is possible to compute very easily their volumes (incidentally, also find some points $P$ belonging to them, which is useful so as to satisfy the constraints of Problem \eqref{quentin}), as well as the derivatives of their volumes (which depend on the measures of each faces) w.r.t. the values $\psi_j$. For the applications to semi-discrete\footnote{One measure being absolutely continuous, the other atomic.} optimal transport problems, the results are now very fast (with discretizations with up to 10$^6$ points in some minutes, in 3D; the reader can have a look at \cite{Mer,KitMerThi,levy3D} but also to Section 6.4.2 in \cite{OTAM}), and the same tools have been used for the applications to the JKO scheme that we just described.

In order to perform an iterated minimization, it is enough to discretize $\varrho_0$ with a finite number of Dirac masses located at points $x_j$, to fix $\tau>0$ small, then to solve \eqref{quentin} with $\mu=\varrho^\tau_{k}$ and set $\varrho^\tau_{k+1}:=G( u,P)$ for the optimal $( u,P)$. Results, proofs of convergence and simulations are in \cite{BenCarMerOud}.
 
\section{The heat flow in metric measure spaces}\label{heat}

In this last section we will give a very sketchy overview of an interesting research topic developed by Ambrosio, Gigli, Savar\'e and their collaborators, which is in some sense a bridge between 
\begin{itemize}
\item the theory of gradient flows in $\mathbb W_2$, seen from an abstract metric space point of view (which is not the point of view that we underlined the most in the previous section), 
\item and the current research topic of analysis and differential calculus in metric measure spaces.
\end{itemize}
This part of their work is very ambitious, and really aims at studying analytical and geometrical properties of metric measure spaces; what we will see here is only a starting point.

The topic that we will briefly develop here is concerned with the heat flow, and the main observation is the following: in the Euclidean space $\R^d$ (or in a domain $\Omega\subset\R^d)$, the heat flow $\partial_t \varrho=\Delta\varrho$ may be seen as a gradient flow in two different ways:
\begin{itemize}
\item first, it is the gradient flow in the Hilbert space $L^2(\Omega)$, endowed with the standard $L^2$ norm, of the functional consisting in the Dirichlet energy $\mathcal{D}(\varrho)=\int|\nabla\varrho|^2\dd x$ (a functional which is set to $+\infty$ if $\varrho\notin H^1(\Omega)$); in this setting, the initial datum $\varrho_0$ could be any function in $L^2(\Omega)$, but well-known properties of the heat equation guarantee $\varrho_0\geq 0\impl\varrho_t\geq 0$ and, if $\Omega$ is the whole space, or boundary conditions are Neumann, then $\int\varrho_0\dd x=1\impl\int\varrho_t\dd x=1$; it is thus possible to restrict to probability densities (i.e. positive densities with mass one);
\item then, if we use the functional $\mathcal E$ of the previous section (the entropy defined with $f(t)=t\log t$), the heat flow is also a gradient flow in the space $ \WW_2(\Omega)$.
\end{itemize}

A natural question arises: is the fact that these two flows coincide a general fact? How to analyze this question in a general metric space? In the Euclidean space this is easy: we just write the PDEs corresponding to each of these flows and, as they are the same PDE, for which we know uniqueness results, then the two flows are the same.

First, we realize that the question is not well-posed if the only structure that we consider on the underlining space is that of a metric space. Indeed, we also need a reference measure (a role played by the Lebesgue measure in the Euclidean space). Such a measure is needed in order to define the integral $\int|\nabla\varrho|^2\dd x$, and also the entropy $\int \varrho\log\varrho \,\dd x$. Roughly speaking, we need to define ``$\dd x$''.  

Hence, we need to consider {\it metric measure spaces}, $(X,d,m)$, where $m\geq 0$ is a reference measure (usually finite) on the Borel tribe of $X$. The unexperienced reader should not be surprised: metric measure spaces are currently the new frontier of some branches of geometric analysis, as a natural generalization of Riemannian manifolds. In order not to overburden the reference list, we just refer to the following papers, already present in the bibliography of this survey for other reasons:\cite{usersguide,AmbGigSav-heat,AGV3,AGV4,cheeger,GIGLIHeat,GigKuwOht,haj,haj2,h-k,LotVil,sha,sturm}.

\subsection{Dirichlet and Cheeger energies in metric measure spaces}\label{5.1}
In order to attack our question about the comparison of the two flows, we first need to define and study the flow of the Dirichlet energy, and in particular to give a suitable definition of such an energy. This more or less means defining the space $H^1(X)$ whenever $X$ is a metric measure space (MMS). This is not new, and many authors studied it: we cite in particular \cite{haj,haj2,cheeger,sha}). Moreover, the recent works by Ambrosio, Gigli and Savar\'e (\cite{AmbGigSav-heat,AGV4}) presented some results in this direction, useful for the analysis of the most general case (consider that most of the previous results require a {\it doubling} assumption and the existence of a {\it Poincar\'e inequality}, see also\cite{h-k}, and this assumption on $(X,d,m)$ is not required in their papers). One of the first definition of Sobolev spaces on a MMS had been given by Hai\l asz, who used the following definition
$$f\in H^1(X,d,m)\;\mbox{ if there is } g\in L^2(X,m) \mbox{ such that } |f(x)-f(y)|\leq d(x,y)(g(x)+g(y)).$$
This property characterizes Sobolev spaces in $\R^d$ by choosing 
$$g=const\cdot M\left[|\nabla f|\right],$$
where $M[u]$ denotes the maximal function of $u$: $M[u](x):=\sup_{r>0} \fint_{B(x,r)}u$ (the important point here is the classical result in harmonic analysis guaranteeing $||M[u]||_{L^2}\leq C(d)||u||_{L^2}$) and $c$ is a suitable constant only depending on the dimension $d$. As this definition is not local, amost all the recent investigations on these topics are rather based on some other ideas, due to Cheeger (\cite{cheeger}), using the relaxation starting from Lipschitz functions, or to Shanmuganlingam (\cite{sha}), based on the inequality
$$|f(x(0))-f(x(1))|\leq \int_0^1 |\nabla f(x(t)||x'(t)|\,\dd t$$ 
required to hold on almost all curves, in a suitable sense. The recent paper \cite{AGV4} resents a classification of the various weak notions of modulus of the gradient in a MMS and analyzes their equivalence. On the contrary, here we will only choose one unique definition for $\int|\nabla f|^2\dd m$, the one which seems the simplest.

For every Lipschitz function $f$ on $X$, let us take its {\it local Lipschitz constant} $|\nabla f|$, defined in \eqref{const lip loc}, and set $\mathcal{D}(f):=\int |\nabla f|^2(x) \dd m$. Then, by relaxation, we define the {\it Cheeger Energy}\footnote{The name has been chosen because Cheeger also gave a definition by relaxation; moreover, the authors did not wish to call it Dirichlet energy, as generally this name is used fro quadratic forms.} $\mathcal{C}(f)$: 
$$\mathcal{C}(f):=\inf\left\{\liminf_n \;\mathcal{D}(f_n)\;:\;f_n\to f \mbox{ in }L^2(X,m),\; f_n\in \Lip(X)\right\}.$$
We then define the Sobolev space $H^1(X,d,m)$ as the space of functions such that $\mathcal{C}(f)<+\infty$. This space will be a Banach space, endowed with the norm $f\mapsto \sqrt{\mathcal{C}(f)}$ and the function $f\mapsto \mathcal C(f)$ will be convex. We can also define $-\Delta f$ as the element of minimal norm of the subdifferential $\partial \mathcal{C}(f)$ (an element belonging to the dual of $H^1(X,d,m)$). Beware that, in general, the map $f\mapsto -\Delta f$ will not be linear (which corresponds to the fact that the norm $\sqrt{\mathcal{C}(f)}$ is in general not Hilbertian, i.e. it does not come from a scalar product).

Definining the flow of $\mathcal C$ in the Hilbert space $L^2(X,m)$ is now easy, and fits well the classical case of convex functionals on  Hilbert spaces or, more generally, of monotone maximal operators (see \cite{BrezOMM}). This brings very general existence and uniqueness results.

\subsection{A well-posed gradient flow for the entropy}
A second step (first developed in \cite{GIGLIHeat} and then generalized in \cite{AmbGigSav-heat}) consists in providing existence and uniqueness conditions for the gradient flow of the entropy, w.r.t. the Wasserstein distance $W_2$. To do so, we consider the funcitonal $\mathcal E$, defined on the set of densities $f$ such that $\varrho:=f\cdot m$ is a probability measure via $\mathcal E(f):=\int f\log f \,\dd m$ and we look at its gradient flow in $ \WW_2$ in the EDE sense. In order to apply the general theory of Section \ref{metricth}, as we cannot use the notion of weak solutions of the continuity equation, it will be natural to suppose that this functional $\mathcal E$ is $\lambda$-geodesically convex for some $\lambda\in \R$. This means, in the sense of Sturm and Lott-Villani, that the space $(X,d,m)$ is a MMS with {\it Ricci curvature bounded from below}. We recall here the corresponding definition, based on the characteristic property 
already evoked in the Section \ref{geodconvW2}, which was indeed a theorem (Proposition \ref{SLV}) in the smooth case.
\begin{definition}
A metric measure space $(X,d,m)$ is said to have a Ricci curvature bounded from below by a constant $K\in\R$ in the sense of Sturm and Lott-Villani if the entropy functional  $\mathcal E:\pical(X)\to\R\cup\{+\infty\}$ defined through
$$\mathcal E(\varrho)=\begin{cases}\int f\log f \dd m &\mbox{ if }\varrho=f\cdot m\\
					+\infty &\mbox{ if $\varrho$ is not absolutely continuous w.r.t. $m$}\end{cases}$$
is $K$-geodesically convex in the space $ \WW_2(X)$.
In this case we say that $(X,d,m)$ satisfies the condition\footnote{The general notation $CD(K,N)$ is used to say that a space has curvature bounded from below by $K$ and dimension bounded from above by $N$ (CD stands for ``curvature-dimension'').} $CD(K,\infty)$.
\end{definition}

Note that the EVI formulation is not available in this case, as we do not have the geodesic convexity of the squared Wasserstein. Moreover, on a general metric space, the use of generalized geodesics is not always possible. This is the reason why we will define the gradient flow of $\mathcal E$ by the EDE condition and not the EVI, but this requires to prove via other methods the uniqueness of such a gradient flow. To do so, Gigli introduced in \cite{GIGLIHeat} an interesting strategy to prove uniqueness for EDE flows, which is based on the following proposition.

\begin{proposition}
If $F:\pical(X)\to \R\cup\{+\infty\}$ is a strictly convex functional (w.r.t. usual convex combinations $\mu_s:=(1-s)\mu_0+s\mu_1$, which is meaningful in the set of probability measures), such that $|\nabla^-F |$ is an upper gradient for $F$ and such that $|\nabla^- F|^2$ is convex, then for every initial measure $\bar\mu$ there exists at most one gradient flow $\mu(t)$ in the EDE sense for the functional $F$ satisfying $\mu(0)=\bar\mu$.
\end{proposition}

In particular, this applies to the functional $\mathcal E$: the strict convexity is straightforward, and the squared slope can be proven to be convex with the help of the formula \eqref{point sup lamcon} (it is interesting to observe that, in the Euclidean case, an explicit formula for the slope is known:
\begin{equation}\label{toFischer}
|\nabla^- \mathcal E|^2(\varrho)=\int \frac{|\nabla f|^2}{f}\,\dd x,
\end{equation}
whenever $\varrho=f\cdot \lcal^d$).

\subsection{Gradient flows comparison}
The last point to study (and it is not trivial at all) is the fact that every gradient flow of $\mathcal{C}$ (w.r.t. the $L^2$ distance) is also an EDE gradient flow of $\mathcal E$ for the $W_2$ distance. This one (i.e. $(\mathcal{C},L^2)\impl (\mathcal E,W_2)$) is the simplest direction to consider in this framework, as computations are easier. This is a consequence of the fact that the structure of gradient flows of convex functionals in Hilbert spaces is much more well understood. In order to do so, it is useful to compute and estimates  
$$\frac{d}{dt}\mathcal E(f_t),\quad \mbox{where $f_t$ is a gradient flow of $\mathcal{C}$ in $L^2(X,m)$}.$$
This computation is based on a strategy essentially developed in \cite{GigKuwOht} and on a lemma by Kuwada. The initial proof, contained in \cite{GigKuwOht}, is valid for Alexandroff spaces\footnote{These spaces, see \cite{BGP}, are metric spaces where triangles are at least as fat as the triangles of a model comparison manifold with constant curvature equal to $K$, the comparison being done in terms of the distances from a vertex of a triangle to the points of a geodesic connecting the two other vertices. These spaces can be proven to have always an integer dimension $d\in \mathbb N\cup\{\infty\}$, and can be consideres as MMS whenever $d<\infty$, by endowing them with their Hausdorff measure $\haus^d$. Note anyway that the comparison manifold with constant curvature can be, anyway, taken of dimension $2$, as only triangles appear in the definition.}. The generalization of the same result to arbitrary MMS satisfying $CD(K,\infty)$ is done in \cite{AmbGigSav-heat}.

\begin{proposition} If $f_t$ is a gradient flow of $\mathcal{C}$ in $L^2(X,\haus^d)$, then we have the following equality with the Fischer information:
$$-\frac{d}{dt}\mathcal E(f_t)=\mathcal{C}(2\sqrt{f_t}).$$
Moreover, for every $\varrho=f\cdot\haus^d\in \pical(X)$ we have 
$$\mathcal{C}(2\sqrt{f})\geq |\nabla^- \mathcal E|^2(\varrho)$$
(where the slope of $\mathcal E$ is computed for the $W_2$ distance\footnote{As in \eqref{toFischer}.}). 
Also, if we consider the curve $\varrho_t=f_t\cdot\haus^d$, it happens that $\varrho_t$ in an AC curve in the space $ \WW_2(X)$ and 
$$|\varrho'|(t)^2\leq \mathcal{C}(2\sqrt{f_t}).$$
these three estimates imply that $\varrho_t$ is a gradient flow of $\mathcal E$ w.r.t. $W_2$.
\end{proposition}

Once this equivalence is established, we can wonder about the properties of this gradient flow. The $L^2$ distance being Hilbertian, it is easy to see that the C$^2$G$^2$ property is satisfied, and hence this flow also satisfies EVI. On the contrary, it is not evident that the same is true when we consider the same flow as the gradient flow of $\mathcal E$ for the distance $W_2$. Indeed, we can check that the following three conditions are equivalent (all true or false depending on the space $(X,d,m)$, which is supposed to satisfy $CD(K,\infty)$; see \cite{AGV3} for the proofs):
\begin{itemize}
\item the unique EDE gradient flow of $\mathcal E$ for $W_2$ also satisfies EVI;
\item the heat flow (which is at the same time the gradient flow of $\mathcal E$ for $W_2$ and of $\mathcal{C}$ for $L^2$) depends linearly on the initial datum;
\item (if we suppose that $(X,d,m)$ is a Finsler manifold endowed with its natural distance and its volume measure), $X$ is a Riemannian manifold.
\end{itemize}

As a consequence, Ambrosio, Gigli and Savar\'e proposed in \cite{AGV3} a definition of MMS having a {\it Riemanniann ricci curvature} bounded from below by requiring both to satisfy the $CD(K,\infty)$ condition, and the linearity of the heat flow (this double condition is usually written $RCD(K,\infty)$). This is the notion of infinitesimally Hilbertian space that we mentioned at the end of Section \ref{metricth}.

It is important to observe (but we will not develop this here) that these notions of Ricci bounds (either Riemannian or not) are stable via measured Gromov-Hausdorff convergence (a notion of convergence similar to the Gromov-Hausdorff convergence of metric spaces, but considering the minimal Wasserstein distance between the images of two spaces via isometric embeddings into a same space). This can be surprising at a first sight (curvature bounds are second-order objects, and we are claiming that they are stable via a convergence which essentially sounds like a uniform convergence of the spaces, with a weak convergence of the measures), but not after a simple observation: also the class of convex, or $\lambda$-convex, functions is stable under uniform (or even weak) convergence! but, of course, proving this stability  is not a trivial fact.

\end{document}